\documentclass[11pt,singlespacing,a4paper, reqno]{amsart}
\usepackage{amsfonts}
\usepackage{amsthm}
\usepackage{amsmath}
\usepackage{amscd}
\usepackage[latin2]{inputenc}
\usepackage{t1enc}
\usepackage[mathscr]{eucal}
\usepackage{indentfirst}
\usepackage{graphicx}
\usepackage{graphics}
\usepackage{pict2e}
\usepackage{epic}
\numberwithin{equation}{section}
\usepackage[margin=3.4cm]{geometry}
\usepackage{epstopdf} 
\usepackage{nicefrac}
\usepackage{calrsfs}
\usepackage{amssymb}
\usepackage{color}
\usepackage[dvipsnames]{xcolor}
\usepackage{fancyhdr}
\usepackage{graphicx}
\usepackage{mathtools}
\usepackage[pdfborder={0 0 0},
  colorlinks=true,
  citecolor=newblue,
  linkcolor=newblue2,
  urlcolor=newblue]{hyperref}
\usepackage{amsthm}
\usepackage{caption,setspace}
\graphicspath{{./Images/}} 
\usepackage{subcaption}
\usepackage{lipsum}
\usepackage{accents}
\usepackage{amsmath, amsthm, amssymb}
\newgeometry{top=2cm,bottom=2cm,left=2cm,right=2cm}
\usepackage{amsmath}
\usepackage{amsthm}
\usepackage{amsfonts}
\usepackage{tabulary}
\usepackage{booktabs}
\usepackage{color}
\usepackage{bbm}
\usepackage{multicol}
\usepackage{framed}
\usepackage[dvipsnames]{xcolor}
\usepackage{dcolumn}

\makeatletter
\newcommand*{\toccontents}{\@starttoc{toc}}
\makeatother

\def\XXint#1#2#3{{\setbox0=\hbox{$#1{#2#3}{\int}$ }
\vcenter{\hbox{$#2#3$ }}\kern-.6\wd0}}

\usepackage{xr}
\definecolor{blue(ryb)}{rgb}{0.01, 0.28, 1.0}
\hypersetup{
     colorlinks   = true,
     citecolor    = blue(ryb),
     linkcolor = blue(ryb)
}

\newcommand{\N}{\mathbb{N}}

\newcommand{\R}{\mathbb{R}}

\newcommand{\eps}{\varepsilon}

\newcommand{\tgrad}{\nabla_\Gamma}
\renewcommand{\L}{_{L^2}}
\renewcommand{\d}{\mathrm{d}}
\newcommand{\md}{\partial^\bullet}
\DeclareMathOperator*{\esssup}{ess\,sup}

\newcolumntype{2}{D{.}{}{2.0}}

\definecolor{light-gray}{gray}{0.9}

\newtheorem{theorem}{Theorem}[section]
\newtheorem{corollary}[theorem]{Corollary}
\newtheorem{lemma}[theorem]{Lemma}
\newtheorem{proposition}[theorem]{Proposition}

\theoremstyle{definition}

\theoremstyle{definition}
\newtheorem{problem}[theorem]{Problem}
\theoremstyle{definition}
\newtheorem{remark}[theorem]{Remark}
\theoremstyle{definition}
\newtheorem{example}[theorem]{Example}

\mathtoolsset{showonlyrefs}

\allowdisplaybreaks

\begin{document}

\title{Cahn-Hilliard equations on an evolving surface}

\author{Diogo Caetano, Charles M. Elliott}

\begin{abstract}
We describe a functional framework suitable to the analysis of the Cahn-Hilliard equation on an evolving surface whose evolution is assumed to be given a priori. The model is derived from balance laws for an order parameter with an associated Cahn-Hilliard energy functional and we establish well-posedness for general regular potentials, satisfying some prescribed growth conditions, and for two singular nonlinearities -- the thermodynamically relevant logarithmic potential and a double obstacle potential. We identify, for the singular potentials, necessary conditions on the initial data and the evolution of the surfaces for global-in-time existence of solutions, which arise from the fact that integrals of solutions are preserved over time, and prove well-posedness for initial data on a suitable set of admissible initial conditions. We then briefly describe an alternative derivation leading to a model that instead preserves a weighted integral of the solution, and explain how our arguments can be adapted in order to obtain global-in-time existence without restrictions on the initial conditions. Some illustrative examples and further research directions are given in the final sections. 
\end{abstract}

\maketitle

\numberwithin{equation}{section}

\section{Introduction}

In this paper, we study well-posedness of Cahn-Hilliard equations on a prescribed moving surface in $\R^3$, for the different cases of smooth, logarithmic and double obstacle nonlinearities and with a constant mobility. Even though there already exists some work concerning this problem on evolving domains, mostly from the point of view of numerical analysis, different models have been proposed and rigorous well-posedness results are still missing. We start this article by introducing these models and by discussing our main results.

\subsection*{The models and main results}

In what follows, we fix a parametrised evolving surface $\{\Gamma(t)\}_{t\in [0,T]}$ with velocity field $\mathbf{V} = \mathbf{V}_\nu + \mathbf{V}_\tau$, with $\mathbf{V}_\nu$ being the normal velocity of the evolution and $\mathbf{V}_\tau$ the tangential component of the velocity of the parametrization. We denote by $u$ a scalar order parameter satisfying, for all subregions $\Sigma(t)\subset \Gamma(t)$, the following balance law
\begin{align}
\dfrac{d}{dt} \int_{\Sigma(t)} u = -\int_{\partial \Sigma(t)} q\cdot \mu = -\int_{\Sigma(t)} \tgrad\cdot q,
\end{align}
where $\mu$ denotes the outer unit conormal and $q=q_d+q_a$ is a flux consisting of an advective term, $q_a = u\mathbf{V}_a$ with $\mathbf V_a$ a tangential vector field, and a diffusive term $q_d = - M(u)\tgrad w$, where $M(\cdot)$ is a mobility function and $$w=-\Delta_\Gamma u + F'(u)$$ denotes the chemical potential of the system, defined as the functional derivative of the Cahn-Hilliard functional 
\begin{align}\label{intro:energy}
\mathrm{E}^\mathrm{CH} (u) = \int_{\Gamma(t)} \dfrac{|\tgrad u|^2}{2} + F(u).
\end{align}
The functional \eqref{intro:energy} measures the total free energy with the gradient term accounting for the surface energy of the interface and the nonlinearity $F$ representing the homogeneous free energy, which will be:
\begin{itemize}
\item[(i)] a smooth potential generalising the frequently considered quartic polynomial
\begin{align}\label{intro:smooth}\tag{F$_\text{s}$}
F(r) = \dfrac{(r^2-1)^2}{4}, \quad r\in \R;
\end{align}
\item[(ii)] the thermodynamically relevant logarithmic potential defined by 
\begin{align}\label{intro:log}\tag{F$_\text{log}$}
F(r) = (1-r)\log(1-r) + (1+r)\log(1+r) + \dfrac{1-r^2}{2}, \quad r\in [-1,1];
\end{align}
\item[(iii)] a double obstacle potential
\begin{align}\label{intro:obs}\tag{F$_\text{obs}$}
F(r) = \begin{cases}
(1-r^2)/2 &\text{ if } |r|\leq 1 \\
+\infty &\text { otherwise }
\end{cases}.
\end{align}
\end{itemize}

We hence obtain, for the cases \eqref{intro:smooth}, \eqref{intro:log}, the system
\begin{align}\label{eq:model1}\tag{CH$_1$}
\begin{split}
\partial^\bullet u + u\tgrad\cdot\mathbf{V} - \tgrad\cdot \big(u \left(\mathbf{V}_\mathbf{\tau}-\mathbf{V}_\mathbf{a}\right)\big)  &= \tgrad\cdot \left(M(u) \tgrad w\right) \\
-\Delta_\Gamma u + F'(u) &= w,
\end{split}
\end{align}
and for \eqref{intro:obs} the system reads as 
\begin{align}\label{eq:model1obs}\tag{CH$'_1$}
\begin{split}
\partial^\bullet u + u\tgrad\cdot\mathbf{V} - \tgrad\cdot \big(u \left(\mathbf{V}_\mathbf{\tau}-\mathbf{V}_\mathbf{a}\right)\big)  &= \tgrad\cdot \left(M(u) \tgrad w\right) \\
-\Delta_\Gamma u + \partial I(u) - u &\ni w,
\end{split}
\end{align}
where $\partial I$ denotes the subdifferential of the indicator function of the set $[-1,1]$. For the system with \eqref{intro:smooth}, we establish in Section \ref{sec:smooth} well-posedness results analogous to those known for the equation on a fixed domain. However, it turns out that, for the singular potentials \eqref{intro:log}, \eqref{intro:obs}, some conditions relating the boundedness of the solutions, the evolution of the domains and the effect of the Cahn-Hilliard dynamics are necessary in order to obtain global in time results. These are the results of Sections \ref{sec:log}, \ref{sec:obs}. This same derivation is considered in \cite{EllRan15}, \cite{OcoSti16}. 

More recently, in \cite{OlsXuYus21, YusQuaOls20, ZimTosLan19} different derivations of a model have been proposed by considering instead a conservation law for the quantity $\rho c$, where $\rho$ denotes the total density of the system and $c$ is a scaled difference of concentrations, which can loosely be seen as the analogue of $u$ in the previous model. We do not intend to make any physical considerations about the models, but it is nonetheless interesting to collect the equations these authors consider and to see how our techniques could be adapted to analyse them.

The derivation in  \cite{ZimTosLan19} starts from the balance laws
\begin{align}\label{intro:rho_deriv}
\dfrac{d}{dt} \int_{\Sigma(t)} \rho = 0
\end{align}
for the total density $\rho$ and
\begin{align}\label{intro:newconservation}
\dfrac{d}{dt} \int_{\Sigma(t)} \rho c = -\int_{\partial \Sigma(t)} q\cdot \mu = -\int_{\Sigma(t)} \tgrad\cdot q,
\end{align}
where again $q$ is a flux for the  quantity $\rho c$. From \eqref{intro:rho_deriv} we thus obtain
\begin{align}\label{intro:rho}
\md \rho + \rho \tgrad \cdot \mathbf{V} = 0,
\end{align}
which can be viewed for each $x\in \Gamma(t)$ as an ODE for $\rho$. The free energy functional and the associated chemical potential are now defined as
\begin{align}
\tilde{\mathrm{E}}^\mathrm{CH}(c) = \int_{\Gamma(t)} \dfrac{|\tgrad c|^2}{2} + \rho F(c)  \quad \text{ and } \quad w = \dfrac{1}{\rho} \dfrac{\delta \tilde{\mathrm{E}}^\mathrm{CH}}{\delta c} = -\dfrac{1}{\rho} \Delta_\Gamma c + F'(c).
\end{align}
As a consequence we are led to the Cahn-Hilliard system
\begin{align}\label{eq:model2}\tag{CH$_2$}
\begin{split}
\rho \partial^\bullet c - \tgrad\cdot \big(\rho c \left(\mathbf{V}_\mathbf{\tau}-\mathbf{V}_\mathbf{a}\right)\big)  &= \tgrad\cdot \left(\rho M(c) \tgrad w\right) \\
-\Delta_\Gamma c + \rho F'(c) &= \rho w,
\end{split} 
\end{align} 
and for \eqref{intro:obs} the system reads as 
\begin{align}\label{eq:model2obs}\tag{CH$'_2$}
\begin{split}
\rho \partial^\bullet c - \tgrad\cdot \big(\rho c \left(\mathbf{V}_\mathbf{\tau}-\mathbf{V}_\mathbf{a}\right)\big)  &= \tgrad\cdot \left(\rho M(c) \tgrad w\right) \\
-\Delta_\Gamma c + \rho \partial I(c) - \rho c  &\ni \rho w.
\end{split}
\end{align}
In Section \ref{sec:alternative} we briefly explain how to adapt the arguments in Sections \ref{sec:smooth}, \ref{sec:log}, \ref{sec:obs} to establish well-posedness for \eqref{eq:model2}, \eqref{eq:model2obs}, by formally obtaining the necessary a priori bounds. This system is particularly interesting to us not only because it can be dealt with by essentially the same techniques, but also because we can now prove well-posedness without any extra conditions on the surfaces or the initial data, as opposed to the first model with the singular potentials. The reason for this is that \eqref{intro:rho} determines a relation between the weight function $\rho$ and the Jacobian determinant of the flow map, so that the evolution of the surfaces is now encoded into the equation via the presence of $\rho$. As we will see, this allows for global well-posedness for any initial data in $H^1(\Gamma_0)$ not identically equal to $\pm 1$.

In \cite{YusQuaOls20}, \cite{OlsXuYus21} yet another derivation is proposed (the former for a Cahn-Hilliard system and the latter for an Allen-Cahn equation). In \cite{YusQuaOls20}, a conservation law as in \eqref{intro:newconservation} is considered for $\rho c$ and the chemical potential is defined as the functional derivative of the energy \eqref{intro:energy}, and a derivation as above leads, for instance for the smooth potential \eqref{intro:smooth} and with a constant mobility $M\equiv 1$, to the system
\begin{align}\label{intro:ols1}\tag{CH$_3$}
\begin{split}
\rho \partial^\bullet c - \tgrad\cdot \big(\rho c \left(\mathbf{V}_\mathbf{\tau}-\mathbf{V}_\mathbf{a}\right)\big)  &= \Delta_\Gamma w\\
-\Delta_\Gamma c + F'(c) &= w.
\end{split}
\end{align}
In \cite{OlsXuYus21}, the authors consider a third different energy 
\begin{align}
\mathrm{E}(c) = \int_{\Gamma(t)} \rho \left( \dfrac{|\tgrad c|^2}{2} + F(c) \right),
\end{align}
and define the chemical potential as the functional derivative of $\mathrm{E}$ with respect to $c$. Again, proceeding in the same way as above, this leads to the system
\begin{align}\label{intro:ols2}\tag{CH$_4$}
\begin{split}
\rho \partial^\bullet c - \tgrad\cdot \big(\rho c \left(\mathbf{V}_\mathbf{\tau}-\mathbf{V}_\mathbf{a}\right)\big)  &= \Delta_\Gamma w\\
- \tgrad \cdot (\rho \tgrad c) + \rho F'(c) &= w.
\end{split}
\end{align}
Although these are similar to the models we consider in this article, the a priori estimates do not quite work out in the same way as for \eqref{eq:model1}, \eqref{eq:model2}, and as such we shall leave a detailed analysis of \eqref{intro:ols1}, \eqref{intro:ols2} for future work. In spite of this, we observe that, for the case of a smooth potential, local in time (or even global in time for the case of potentials with quadratic growth) existence is easy to establish using our techniques. 

In summary, in this work we present a rigorous derivation for \eqref{eq:model1}, \eqref{eq:model1obs} and establish existence, uniqueness and stability of weak solutions. We also address extra regularity results for all the potentials. We then briefly consider \eqref{eq:model2}, \eqref{eq:model2obs} and explain how our arguments can be adapted to obtain the same type of results, including global in time existence for the singular potentials without restricting the set of admissible initial conditions. The study of \eqref{intro:ols1}, \eqref{intro:ols2} seems to require a new approach for the a priori estimates, so we leave it for future work.


\subsection*{Background and motivation}

Our interest in this equation is part of the more general problem of understanding equations on non-static domains. The study of partial differential equations on moving domains, of which evolving hypersurfaces are an example, has been a very active area of research recently and many applications to physics, materials science, biology, among other sciences, have been considered. Generalising systems that are usually considered in stationary domains to spaces that evolve with time has been seen to lead to more accurate and realistic models, for instance for the study of surface dissolution of binary alloys, for phenomena of cell motility, and also for the modelling of elastic membranes (such as lipid bilayer membranes or cell tissues). Some references for these applications are \cite{BarEllMad11, EilEll08, EllStiVen12, ErlAziKar01, GarLamSti14, VenSekGaf11}.
These are also challenging problems from the point of view of mathematics. Indeed, equations on non-cylindrical domains have been considered as early as in \cite{Lio57, Bai65}, and more recently in \cite{Nae15} with applications to the Stokes problem and in \cite{LanSonTanThu21} for general quasilinear parabolic problems. Many other examples can be found in the references of these two articles. Of particular relevance to our work are \cite{DziEll07-a, DziEll13-a, DecDziEll05-a, Vie14}, where the authors consider partial differential equations in both fixed and moving surfaces, and also \cite{AlpEllSti15a, AlpEllSti15b} in which an abstract functional framework is established to treat linear PDEs in domains that evolve with time. 

The corresponding Cahn-Hilliard equation on a flat domain in $\R^n$ was introduced in \cite{CahHil58} to study spinodal decomposition in binary alloys; more precisely, it models the phase separation of an alloy consisting of two components when the temperature of the system has been quenched to a temperature below the critical temperature, leading to a spatially separated two-phase structure as opposed to the uniform mixed state of equilibrium. See also \cite{Cah61}. It has more recently been understood to provide a good model to describe later stages of the evolution of phase transition phenomena. The introduction of the equation in \cite{CahHil58} and subsequent studies in natural sciences, see for example \cite{NovSeg84} and included references, generated also interest on the equation from the mathematics community. Initial well posedness and numerical studies of the equation, \cite{EllSon86}, considered the problem with a polynomial nonlinearity. From then on, many generalizations of the equation were considered and several applications to different areas of mathematics were explored; some examples are \cite{Ell89, EllFre89, EllFreMil89, EllLar92} and \cite{EllGar96a, DaiDu14, DaiDu16} for the degenerate version of the equation. See also \cite{GarKno20, Hei15} where the gradient flow structure of the problem is exploited, \cite{CahEllNov96, Gar13} for relations with geometric flows and \cite{PraDeb96} for a stochastic version. It is also worthwhile mentioning the articles \cite{Peg86, CahEllNov96, OcoSti16} where an asymptotic analysis is performed for the parameter $\eps\to 0$; the latter does so for the equation on an evolving surface. We mention particularly the work in \cite{DebDet95, EllLuc91}, where the logarithmic potential is considered (see also the survey \cite{CheMirZel11}), and the article \cite{BloEll91, BloEll92}, where the authors analyse the double obstacle problem. These served as motivation for a significant part of this paper. For a general overview we refer also to the survey articles \cite{Ell89, NovCoh08} or the recent book \cite{Mir19}. 

The results we present in this article are mostly motivated by those in \cite{BloEll91, EllLuc91, EllRan15, GarKno20}, which we generalise to the case of moving surfaces in $\R^3$. As mentioned above, the biggest difficulties arise for \eqref{eq:model1}, \eqref{eq:model1obs} with the singular potentials, due to an interplay between the boundedness of solutions, the evolution of the surfaces and the initial data. As we shall see, different regimes need to be studied for these cases. 


\subsection*{Structure of the paper} 

Finally, let us describe the structure of the paper. We begin by introducing the necessary analytic background and notation for posing PDEs on evolving spaces in Section 2, and we proceed to a derivation of our model and the statement of our results in Section 3. Section 4 is devoted to the study of \eqref{eq:model1} with a smooth nonlinearity satisfying some polynomial growth conditions, and for this case we establish well-posedness of the equation by using on the Galerkin method. This is the starting point for the approximation method we use to tackle the singular potentials in Section 5,  for which it turns out that the moving nature of the domains has an impact on existence of (global in time) solutions. We identify some necessary conditions for well-posedness, and prove existence and uniqueness of solutions in these regimes by approximation with more regular nonlinearities. In Section 6 we analyse \eqref{eq:model2}, \eqref{eq:model2obs}, which does not preserve the integral of solutions, and Section 7 contains some simple examples. We finish with a discussion of our results and some open questions we propose to address in future work. Our results are, to the best of our knowledge, new in the literature, and generalise the classical results for the Cahn-Hilliard equation on a fixed domain or surface.

\subsection*{Notation} 


For simplicity of notation we will frequently omit the differentials $\d t$, $\d \Gamma$ on the integrals. It should be clear that, in time, we always integrate with respect to Lebesgue measure and over the surfaces we use the surface measure. For a function $u\in L^1(\Gamma)$, we denote its mean value over the surface $\Gamma$ by 
\begin{align*}
(u)_\Gamma := \dfrac{1}{|\Gamma|} \int_{\Gamma} u,
\end{align*}
where $|\Gamma|$ denotes the measure of $\Gamma$. As for the constants appearing in estimates, we use $C_1, C_2, \dots$ positive constants which might be different in different equations. More precisely, inside each equation the constants appear with indices in increasing order, but every time we start a new estimate the first constant will again be denoted $C_1$. If only one constant is involved, then we denote it simply by $C$. We believe this keeps the notation clear and helps keep track, inside each estimate, of when a new constants appear.

\numberwithin{equation}{subsection}

\section{Preliminaries}\label{sec:prelim}

\subsection{Evolving surfaces}

Fix $T>0$ and a $C^2$-evolving surface $\{\Gamma(t)\}_{t\in [0,T]}$ in $\R^3$. More precisely, this means we have a regular, closed, connected, orientable $C^2$-surface $\Gamma_0$ in $\R^3$ and a smooth \textit{flow map} 
\begin{align*}
\Phi\colon [0,T]\times \Gamma_0 \to \R^3
\end{align*}
such that 
\begin{itemize}
\item[(i)] denoting $\Gamma(t) := \Phi_t^0(\Gamma_0)$, the map $$\Phi_t^0 := \Phi(t, \cdot)\colon \Gamma_0\to \Gamma(t)$$ is a $C^2$-diffeomorphism, with inverse map $$ \Phi_0^t \colon \Gamma(t)\to \Gamma_0;$$
\item[(ii)] $\Phi_0^0 = \text{id}_{\Gamma_0}$. 
\end{itemize}
%

It follows from the definition above that, for each $t\in [0,T]$, $\Gamma(t)$ is also a regular, closed, connected, orientable $C^2$-surface. We can also naturally define diffeomorphisms between the surfaces at different instants $s, t\in [0,T]$ by 
\begin{align*}
\Phi_t^s := \Phi_t^0\circ \Phi_0^s \colon \Gamma(s)\to \Gamma(t).
\end{align*}
Now in order to analytically treat problems in an evolving surface it is convenient to assume that $\Phi$ is the flow of some prescribed vector field in $\R^3$, which we will do in this article:

\vskip 3mm

\hskip -5mm \textbf{Assumption $\mathbf{(A_\Phi)}$}:
There exists a velocity field $\mathbf{V}\colon [0, T]\times \R^3\to \R^3$ with regularity 
\begin{align}
\mathbf V\in C^0\left([0,T]; C^2(\R^3, \R^3)\right)
\end{align}
such that, for any $t\in [0,T]$ and every $x\in \Gamma_0$, 
\begin{align}
\dfrac{d}{dt} \, \Phi_t^0(x) &= \mathbf{V}\left(t, \Phi_t^0(x)\right), \quad \text{ in } [0,T] \\
\Phi_0^0(x) &= x.
\end{align}

\vskip 3mm


We denote the tangential and normal components of $\mathbf{V}$ by $\mathbf{V_\tau}, \mathbf{V_\nu}$, respectively. Note that, due to compactness, there exists $C_{\mathbf{V}}>0$ independent of $t$ such that, 
\begin{align}\label{eq:velbound}
\|\mathbf{V_\tau}(t)\|_{C^2(\Gamma(t))}, \, \|\mathbf{V_\nu}(t)\|_{C^2(\Gamma(t))} \leq \|\mathbf{V}(t)\|_{C^2(\Gamma(t))} \leq C_{\mathbf{V}}, \quad \text{ for } t\in [0,T].
\end{align}
In this setting we can define the \textit{normal material derivative} of a scalar quantity $u$ on $\Gamma(t)$ by
\begin{align}
\partial^\circ u := \tilde{u}_t + \nabla \tilde{u} \cdot \mathbf{V_\nu},
\end{align}
and its \textit{material time derivative} as
\begin{align}\label{eq:strongder}
\partial^\bullet u := \partial^\circ u + \nabla \tilde{u}\cdot \mathbf{V_\tau} = \tilde{u}_t + \nabla \tilde{u} \cdot \mathbf{V},
\end{align}
where $\tilde{u}$ denotes any extension of $u$ to a neighbourhood of $\Gamma(t)$. This last definition takes into account not only the evolution of the domains but also the movement of points in the surface. The following transport formula will be useful throughout the article: 

\begin{theorem}\label{thm:transportforms}
Let $\Sigma(t)\subset \Gamma(t)$ be evolving with velocity $\mathbf{V}=\mathbf{V_\nu}+\mathbf{V_\tau}$, where $\mathbf{V_\nu}, \mathbf{V_\tau}$ denote, respectively, the normal and tangential components of $\mathbf{V}$. Define 
\begin{align}
\mathbf{V}_{\partial M} := \mathbf{V}_\tau\cdot \mu,
\end{align}
where $\mu$ denotes the outer unit conormal to $\partial \Sigma(t)$. Then
\begin{align}
\dfrac{d}{dt} \int_{\Sigma(t)} u = \int_{\Sigma(t)} \partial^\bullet u + u \tgrad \cdot \mathbf{V} = \int_{\Sigma(t)} \partial^\circ u + u\tgrad\cdot\mathbf{V_\nu} + \int_{\partial \Sigma(t)} u \mathbf{V}_{\partial M}, 
\end{align}
for every function $u$ for which the quantities above make sense.
\end{theorem}


In our context, by taking $\Sigma(t)=\Gamma(t)$ and recallng that $\Gamma(t)$ is a closed surface we obtain 
\begin{align}
\dfrac{d}{dt} \int_{\Gamma(t)} u = \int_{\Gamma(t)} \partial^\bullet u + u \tgrad \cdot \mathbf{V} = \int_{\Gamma(t)} \partial^\circ u + u\tgrad\cdot\mathbf{V_\nu},
\end{align}
which shows that the time evolution of integral quantities depends only on the normal component of the velocity. 

Denoting now by $J_t^0$, respectively $J_0^t$, the change of area element from $\Gamma_0$ to $\Gamma(t)$, respectively from $\Gamma(t)$ to $\Gamma_0$, satisfying
\begin{align}\label{eq:changeofvariables}
\int_{\Gamma(t)} \eta = \int_{\Gamma_0} \tilde \eta \, J_t^0 , \quad \int_{\Gamma_0} \tilde \eta = \int_{\Gamma(t)} \eta \, J_0^t
\end{align}
where, given $\eta\colon \Gamma(t)\to \R$, we are writing $\tilde \eta\colon \Gamma_0\to \R$ for the function
\begin{align}
\tilde{\eta}(p) = \eta \big(\Phi_0^t (p)\big), \quad \forall p\in \Gamma_0.
\end{align}
Combining \eqref{eq:changeofvariables} with the transport theorem above, it follows that, for each $p\in \Gamma_0$, 
\begin{align}
\dfrac{d}{dt} J_t^0(p) = J(t,p) \, \tgrad \cdot \mathbf{V}\left(t, \Phi(t,p)\right),
\end{align}
from where
\begin{align}\label{eq:formulaJ}
J_t^0(p) = \exp\left\{\int_0^t \tgrad \cdot \mathbf{V}\left(s, \Phi(s,p)\right) \right\}.
\end{align}
Due to \eqref{eq:velbound}, this can then be used to prove that there exists a constant $C_A>0$, depending only on the bound for $\mathbf{V_\nu}$, such that
\begin{align}\label{eq:areas}
0<\dfrac{|\Gamma_0|}{C_A} \leq |\Gamma(t)| \leq C_A |\Gamma_0|, \quad \forall t\in [0,T].
\end{align}
Under the stronger assumption $\mathbf{(A_\Phi)}$ we can prove the following additional results:

\begin{lemma}\label{lem:assumptionsgive}
Suppose $\mathbf{(A_\Phi)}$ holds.
\begin{enumerate}
\item[(a)] Given $t\in [0,T]$ and $u\in L^2(\Gamma(t))$ or $u\in H^1(\Gamma(t))$, define $\phi_{-t} u = u \circ \Phi_t^0$. Then
\begin{align}
\phi_{-t}\colon L^2(\Gamma(t))\to L^2(\Gamma_0), \quad \phi_{-t}\colon H^1(\Gamma(t))\to H^1(\Gamma_0),
\end{align}
and these maps are isomorphisms between the two spaces with constants of continuity independent of $t$, and the same result is true for $\phi_t:=(\phi_{-t})^{-1}$; \vskip 1mm
\item[(b)] Given $u\in L^2(\Gamma_0)$ (resp. $H^1(\Gamma_0)$), the map 
\begin{align}
t\mapsto \|\phi_t u\|_{L^2(\Gamma(t))} \quad \text{\big(resp. } t\mapsto \|\phi_t u\|_{H^1(\Gamma(t))}\big)
\end{align}
 is continuous; \vskip 1mm
\item[(c)] \textnormal{[Poincar\'{e} inequality]} For any $t\in [0,T]$, there exists a constant $C_P>0$, independent of $t$, such that, for any $u\in H^1(\Gamma(t))$,
\begin{align}\label{eq:poincare}
\|u - (u)_{\Gamma(t)}\|_{L^2(\Gamma(t))}^2 \leq C_P \|\nabla_{\Gamma(t)} u\|^2_{L^2(\Gamma(t))};
\end{align} \vskip 1mm
\item[(d)] \textnormal{[Sobolev embedding theorems]} For any $t\in [0,T]$, the following continuous embeddings hold with continuity constants independent of $t$:
\begin{itemize}
\item[(d1)] for all $p\in [1,+\infty)$, $H^1(\Gamma(t))\hookrightarrow L^p(\Gamma(t))$;
\item[(d2)] for $k, r\geq 0$ integers with $k-r>1$, $H^k(\Gamma(t))\hookrightarrow C^r(\Gamma(t))$;
\item[(d3)] for $k, r\geq 0$ integers and $\alpha\in (0,1)$ with $k-r-\alpha\geq 1$, $H^k(\Gamma(t))\hookrightarrow C^{r+\alpha}(\Gamma(t))$.
\end{itemize}
\end{enumerate}
\end{lemma}

\begin{remark}\label{rem:compatibility}
In the spirit of \cite[Definition 2.4]{AlpEllSti15a}, given a family of Banach spaces $\{X(t)\}_{t\in [0,T]}$ with maps $\phi_t\colon X_0\to X(t)$, we call the pair $(X(t), \phi_t)_{t\in [0,T]}$ \textit{compatible} if: 
\begin{itemize}
\item[(i)] $\phi_t\colon X_0\to X(t)$ is a linear homeomorphism such that $\phi_0= \text{Id}_{X_0}$, with inverse $\phi_{-t} = (\phi_t)^{-1}$;
\item[(ii)] $\phi_t$, $\phi_{-t}$ are bounded with continuity constant independent of $t$;
\item[(iii)] the map $t\mapsto \|\phi_t u\|_{X(t)}$ is continuous, for each $u\in X_0$.
\end{itemize}
The statements in (a), (b) above then simply say that the families $(L^2(\Gamma(t)), \phi_t)_{t\in [0,T]}$ and $(H^1(\Gamma(t)), \phi_t)_{t\in [0,T]}$ are compatible pairs of Hilbert spaces. 
\end{remark}

The proofs follow from calculations involving the change of variables and integration by parts formulas (see, for example, \cite[Section 3]{Vie14}). We emphasize that, in parts (c) and (d), the constants can be chosen independently of $t$, and the statements in (d) make use of the fact that $\Gamma(t)$ is $2$-dimensional.

\subsection{Time-dependent Bochner spaces}

Observe that, in part (b) of the result above, we used the flow map to define a pullback operator for functions defined on the hypersurface by
\begin{align}\label{eq:pull}
\phi_{-t} u = u \circ \Phi_t^0, \quad t\in [0,T], \,\, u\in L^2(\Gamma(t)) \text{ or } u\in H^1(\Gamma(t)) 
\end{align}
Now if we denote, for $t\in [0,T]$, $H(t)=L^2(\Gamma(t))$ and $V(t) = H^1(\Gamma(t))$, then the results in (b), (c) above show that the pairs $(H, \phi_{(\cdot)})$ and $(V, \phi_{(\cdot)}|_V)$ are compatible in the sense of \cite{AlpEllSti15a}. We can thus consider, for $t\in [0,T]$, the \textit{evolving Hilbert spaces} $L^2(\Gamma(t))$ and $H^1(\Gamma(t))$, in which case the work in \cite{AlpEllSti15a} gives rise to the following time-dependent Bochner spaces: \vskip 1mm

\begin{itemize}
\item[(i)] The separable Hilbert spaces $L^2_{L^2}$ and $L^2_{H^1}$ consisting of (equivalence classes) of functions 
\begin{figure}[h]
\begin{minipage}{0.4\textwidth}
\begin{align}
u\colon [0, T]&\to \bigcup_{t\in[0,T]} L^2(\Gamma(t))\times \{t\} \\
 t&\mapsto (\bar{u}(t), t)
\end{align}
\end{minipage}
\quad or \quad
\begin{minipage}{0.4\textwidth}
\begin{align}
u\colon [0, T]&\to \bigcup_{t\in[0,T]} H^1(\Gamma(t))\times \{t\} \\
 t&\mapsto (\bar{u}(t), t)
\end{align}
\end{minipage}
\end{figure} 
 such that $\phi_{-(\cdot)}\bar{u}(\cdot)\in L^2(0, T; L^2(\Gamma_0))$ or $\phi_{-(\cdot)}\bar{u}(\cdot)\in L^2(0, T; H^1(\Gamma_0))$; we identify $u\equiv \bar{u}$. These spaces are endowed with the inner products 
\begin{align}
(u,v)_{L^2_{L^2}} &:= \int_0^T (u(t), v(t))_{L^2(\Gamma(t))}, \quad u, v\in L^2_{L^2}, \\
(u,v)_{L^2_{H^1}} &:= \int_0^T (u(t), v(t))_{H^1(\Gamma(t))} , \quad u, v\in L^2_{H^1}.
\end{align}
We also identify $(L^2_{L^2})^* \equiv L^2_{L^2}$ via the Riesz map, and set $L^2_{H^{-1}} := (L^2_{H^1})^*$. \vskip 1mm
\item[(ii)] For $X(t)=L^2(\Gamma(t))$ or $X(t)=H^1(\Gamma(t))$, the Banach space $L^\infty_{X}$, of those functions $u\in L^2_{X}$ such that $t\mapsto \phi_{-t}u(t) \in L^\infty(0,T; X(0))$, with the norm
\begin{align}
\|u\|_{L^\infty_{X}} := \esssup_{t\in [0,T]} \|u(t)\|_{X(t)};
\end{align}

\item[(iii)] For $X(t)=L^2(\Gamma(t))$ or $X(t)=H^1(\Gamma(t))$, the Banach space $C^0_{X}$ of functions $u\in L^2_{X}$ such that $t\mapsto \phi_{-t}u(t)\in C^0([0,T]; X(0))$, with the norm 
\begin{align}
\|u\|_{C^0_{X}} := \sup_{t\in [0,T]} \|u(t)\|_{X(t)};
\end{align}

\item[(iv)] More generally, for $X(t)=L^2(\Gamma(t))$ or $X(t)=H^1(\Gamma(t))$ and $k\in\N$, the space of $k$ times differentiable functions $C^k_{X}$ of functions $u\in L^2_{X}$ such that $t\mapsto \phi_{-t}u(t)\in C^k([0,T]; X(0))$; \vskip 1mm

\item[(v)] For $X(t)=L^2(\Gamma(t))$ or $X(t)=H^1(\Gamma(t))$, the space of \textit{test functions} $\mathcal{D}_{X}(0,T)$ consisting of $u\in L^2_{X}$ such that $$t\mapsto \phi_{-t}u(t)\in \mathcal{D}((0,T); X(0)):=C_c^\infty((0,T); X(0)).$$
\end{itemize}

\begin{remark}\label{rem:compat_spaces}
The definitions above extend naturally to the case of a general compatible pair $(X(t), \phi_t)_{t\in [0,T]}$, with $X(t)$ a Banach space for each $t$.
\end{remark}

It remains to define an appropriate notion of a time derivative for functions in the above time-dependent spaces. For regular functions, we do it in the natural way, by pulling back to the reference domain $X(0)$, differentiating in time, and pushing forward to return to $X(t)$.  For $X(t)=L^2(\Gamma(t))$ or $X(t)=H^1(\Gamma(t))$, the \textit{(strong) time derivative} of a function $u\in C^1_{X}$ is defined to be
\begin{align}\label{eq:newstrong}
\partial^\bullet u(t) := \phi_t \dfrac{d}{dt} \phi_{-t}u(t) \in C^0_{X}.
\end{align}
By considering values along curves that follow the evolution, it can be seen that, whenever both definitions \eqref{eq:newstrong} and \eqref{eq:strongder} make sense, they coincide. 

The definition \eqref{eq:strongder} can now be abstracted to a weaker sense as follows. Let $u\in L^2_{H^1}$. A functional $v\in L^2_{H^{-1}}$ is said to be the \textit{weak time derivative} of $u$, and we write $v=\partial^\bullet u$, if, for any $\eta\in \mathcal{D}_{H^1}(0,T)$, we have
\begin{align}
\int_0^T \langle v(t), \eta(t)\rangle_{H^{-1}\times H^1}  = -\int_0^T (u(t),\md\eta(t))\L - \int_0^T \int_{\Gamma(t)} u(t)\eta(t) \nabla_{\Gamma(t)}\cdot \mathbf{V}(t).
\end{align}
Observe that $\md \eta$ is the strong material derivative of $\eta$. Of course, for a regular function, the strong time derivative \eqref{eq:strongder} coincides with the weak derivative defined above. We are then interested in looking for solutions to parabolic PDEs lying in the space
\begin{align}
H^1_{H^{-1}} := \left\{ u\in L^2_{H^1} \colon \md u\in L^2_{H^{-1}}\right\}.
\end{align}
Often the following criterion for weak differentiability is easier to apply (see \cite[Lemma 3.5]{AlpEllSti15a}):

\begin{lemma}\label{lem:criterionweakderiv}
Let $u\in L^2_{H^1}$ and $g\in L^2_{H^{-1}}$. Then $u$ is weakly differentiable with $\md u = g$ if and only if, for all $\eta\in H^1(\Gamma_0)$, 
\begin{align}
\dfrac{d}{dt} \left( u(t), \phi_t \eta\right)_{L^2(\Gamma(t))} = \langle g(t), \phi_t \eta\rangle_{H^{-1}(\Gamma(t))\times H^1(\Gamma(t))} + \int_{\Gamma(t)} u(t) \phi_t \eta \nabla_{\Gamma(t)} \cdot \mathbf{V}(t).
\end{align}
\end{lemma}

We conclude this section with two results which will be of use later in the text. The first one relates some of the spaces defined above. 

\begin{proposition}\label{prop:inclusions}
Suppose assumption $(\mathbf{A_\Phi})$ holds.
\begin{itemize}
\item[(a)] $L^2_{H^{1}}\hookrightarrow L^2_{L^2}\cong (L^2_{L^2})^* \hookrightarrow L^2_{H^{-1}}$ with each inclusion continuous and dense; \vskip 1mm
\item[(b)] The space $H^1_{H^{-1}}$ is continuously embedded in $C^0_{L^2}$; \vskip 1mm
\item[(c)] \textnormal{[Aubin--Lions Lemma]} The space $H^1_{H^{-1}}$ is compactly embedded in $L^2_{L^2}$.
\end{itemize} 
\end{proposition}

For details on the above, see \cite{AlpEllSti15a} (or \cite[Sections 3, 4]{AlpEllSti15b} for a more summarized version), and \cite[Lemma B.3]{AlpEllTer17} for the Aubin--Lions-type result. We will also need the following result which serves as a replacement for the Aubin--Lions lemma when no estimate on the time derivative is available. It is a generalisation of \cite[Theorem 1]{RakTem01} to the evolving space context. For generality, we state and prove it in the context of abstract families of Hilbert spaces. See Remarks \ref{rem:compatibility}, \ref{rem:compat_spaces} for the definitions of compatibility and the abstract spaces below.

\begin{theorem}\label{lem:temam_general}
Let $T<\infty$ and $(H(t), \phi_t)_{t\in [0,T]}$, $(V(t), \phi_t)_{t\in [0,T]}$ be compatible families of separable Hilbert spaces such that $V(t)\subset H(t)$ is compact and dense. Let $(u_n)_n\subset L^2_V$ be such that $u_n\rightharpoonup u$ in $L^2_V$. Then $u_n\to u$ in $L^2_H$ if and only if 
\begin{itemize}
\item[(i)] $u_n(t)\rightharpoonup u(t)$ in $H(t)$ for a.e. $t$;
\vskip 2mm
\text{and}
\vskip 2mm
\item[(ii)] it holds that
\begin{align}\label{eq:lemma_gen}
\lim_{\underset{E\subset [0,T]}{|E|\to 0}} \sup_{n\in\N} \int_E \|u_n(t)\|_{H(t)}^2 \, \d t = 0.
\end{align}
\end{itemize}
\end{theorem}

\begin{proof}
Since $u_n\rightharpoonup u$ in $L^2_V$, then by compatibility we also have $\phi_{-(\cdot)} u_n \rightharpoonup \phi_{-(\cdot)} u$ in $L^2(0,T; V)$.

Suppose $u_n \to u$ in $L^2_H$, then also $\phi_{-(\cdot)} u_n \to \phi_{-(\cdot)} u$ in $L^2(0,T; H)$. By Theorem \cite[Theorem 1]{RakTem01} we have, on the one hand, that $\phi_{-t} u_n(t) \rightharpoonup \phi_{-t} u(t)$ in $H_0$, which implies again by compatibility that $u_n(t) \rightharpoonup u(t)$ in $H(t)$, and, on the other hand, 
\begin{align}
\begin{split}
0 = \dfrac{1}{C}\lim_{\underset{E\subset [0,T]}{|E|\to 0}} \sup_{n\in\N} \int_E \|\phi_{-t} u_n(t)\|_{H_0}^2 \, \d t &\leq \lim_{\underset{E\subset [0,T]}{|E|\to 0}} \sup_{n\in\N} \int_E \|u_n(t)\|_{H(t)}^2 \, \d t \\
&\leq C\lim_{\underset{E\subset [0,T]}{|E|\to 0}} \sup_{n\in\N} \int_E \|\phi_{-t} u_n(t)\|_{H_0}^2 \, \d t = 0,
\end{split}
\end{align}
which implies \eqref{eq:lemma_gen}. Thus (i) and (ii) hold.

Conversely, suppose (i) and (ii) are satisfied. Then arguing as above with the pullback map $\phi_{-t}$ it follows that (i) and (ii) hold for $\phi_{-t} u_n(t)$ on $H_0$. By Theorem \cite[Theorem 1]{RakTem01}, $\phi_{-(\cdot)} u_n\to \phi_{-t} u$ in $L^2(0,T; H)$, and compatibility implies $u_n \to u$ in $L^2_H$, proving the result.
\end{proof} 

In the next proposition, we collect two formulas for the time derivatives of the $L^2$-inner product, and a proof can be found e.g. in \cite[Lemma 5.2]{DziEll07-a}. 

\begin{proposition}\label{prop:derivinnerprod}
Let $u, v\in H^1_{H^{-1}}$. Then the function $t\mapsto (u(t), v(t))_{L^2(\Gamma(t))}$ is absolutely continuous, and we have
\begin{align}
\dfrac{d}{dt} (u, v)_{L^2} &= \langle \partial^\bullet u, v\rangle_{H^{-1}\times H^1} + \langle \partial^\bullet v, u\rangle_{H^{-1}\times H^1} + \int_{\Gamma(t)} uv\nabla_{\Gamma}\cdot \mathbf{V}.
\end{align}
If additionally $\tgrad \md u, \tgrad \md v\in L^2_{L^2}$, then also $t\mapsto (\nabla_{\Gamma(t)} u(t), \nabla_{\Gamma(t)}v(t))_{L^2(\Gamma(t))}$ is absolutely continuous, and
\begin{align}
\dfrac{d}{dt} (\tgrad u, \tgrad v)\L = (\tgrad\partial^\bullet u, \tgrad v)\L + (\tgrad u, \tgrad \partial^\bullet v)\L + \int_{\Gamma(t)} \mathrm{B}(\mathbf{V}) \tgrad u\cdot \tgrad v,
\end{align}
where, for a vector field $v=(v_1, v_2, v_3)$, $\mathrm{B}(v)= (\nabla_\Gamma \cdot v)I - 2D(v)$ and the matrix $D(v)=(D_{ij}(v))_{i,j=1}^{3}$ is defined by
\begin{align}
D_{ij}(v) = \dfrac{1}{2}\sum_{k=1}^{3} \delta_{ik} \underline{D}_k v_j + \delta_{jk} \underline{D}_k v_i  = \dfrac{\underline{D}_iv_j + \underline{D}_jv_i}{2}.
\end{align}
\end{proposition}

We now have all the background needed to state our problem. 

\begin{remark}
We presented the precise framework necessary to treat the Cahn-Hilliard equation, but the setting described can be made more general.
\begin{itemize}
\item[(i)] On the one hand, it can be extended with the natural changes to compact hypersurfaces $\Gamma(t)\subset\R^{n+1}$, with any $n\in \N$, see for instance \cite{DziEll07-a}.  Restriction to dimension $n=2$ plays a role only in the statement of the Sobolev embeddings in Lemma \ref{lem:assumptionsgive}(d).
\item[(ii)] On the other hand, and as discussed in Remarks \ref{rem:compatibility}, \ref{rem:compat_spaces}, the work in \cite{AlpEllSti15a} allows to define $L^p_X$ for any $p\in [1,+\infty]$ and a family of time-dependent Banach spaces $\{X(t)\}_{t\in [0,T]}$ with maps $\phi_{-t}\colon X(t)\to X_0$ satisfying the same properties as the pullback map in \eqref{eq:pull}.
\end{itemize}
\end{remark}

\section{Problem setup}\label{sec:model}

For the rest of the text, fix $T>0$ and an evolving $C
^2$-surface $\{\Gamma(t)\}_{t\in [0,T]}$ in $\R^3$ satisfying assumption $(\mathbf{A_\Phi})$.

\subsection{Derivation of the equation}

We start by deriving the Cahn-Hilliard equation on the evolving surface $\{\Gamma(t)\}_{t\in[0,T]}$ from a conservation law. Fix $t\in [0,T]$, let $u$ represent a scalar quantity on $\Gamma(t)$ and let $q$ be a surface flux for $u$. Consider the following balance law for an arbitrary portion $\Sigma(t)\subseteq \Gamma(t)$ evolving with the given purely normal velocity $\mathbf{V_\nu}$
\begin{align}
\dfrac{d}{dt}\int_{\Sigma(t)} u = -\int_{\partial \Sigma(t)} q\cdot \mu,
\end{align}
where $\partial \Sigma(t)$ is the boundary of $\Sigma(t)$ and $\mu$ is the conormal on $\partial \Sigma(t)$, i.e. the unit normal to $\partial \Sigma(t)$ which is tangential to $\Sigma(t)$. Note that the normal component of $q$ does not contribute to the flux, and hence we can assume that $q$ is tangential. Using integration by parts we have 
\begin{align}
\int_{\partial \Sigma(t)} q\cdot  \mu = \int_{\Sigma(t)} \nabla_\Gamma\cdot q - \int_{\Sigma(t)} q\cdot \nu H = \int_{\Sigma(t)} \nabla_\Gamma\cdot q,
\end{align}
where $H$ denotes the mean curvature. The second term vanishes because $q$ is tangential. 

On the other hand, using the transport formula in Theorem \ref{thm:transportforms} gives
\begin{align}
\dfrac{d}{dt}\int_{\Sigma(t)} u = \int_{\Sigma(t)} \partial^\circ u + u\nabla_\Gamma \cdot \mathbf{V}_\nu,
\end{align}
so combining the previous equations yields 
\begin{align}
\int_{\Sigma(t)} \partial^\circ u + u\nabla_\Gamma \cdot \mathbf{V}_\nu + \nabla_\Gamma \cdot q = 0,
\end{align}
for any portion $\Sigma(t)\subseteq \Gamma(t)$. This implies that
\begin{align}\label{eq:chnormal}
\partial^\circ u + u\nabla_\Gamma \cdot \mathbf{V}_\nu + \tgrad\cdot q = 0 \quad \text{ on } \Gamma(t).
\end{align}


We are interested in considering the surface flux $q$ to be the sum of a diffusive term of the form $q_d = -M(u)\nabla_\Gamma w$ and an advection term $q_a = u\mathbf{V}_a$, with an advective tangential velocity $\mathbf{V}_a$.
In the diffusive flux, $M(u)$ is a \textit{mobility term} and $w$, the \textit{chemical potential}, is defined by
\begin{align}
w = - \Delta_\Gamma u +F'(u),
\end{align} 
where the homogeneous free energy $F$ will be taken of different types in the following sections. The chemical potential $w$ is the functional derivative of the free energy 
\begin{align}\label{eq:CHfunctional}
\mathrm{E}^\mathrm{CH}[u] = \int_{\Gamma(t)} \dfrac{|\nabla_\Gamma u|^2}{2} + F(u),
\end{align}
which we will refer to throughout the text as the \textit{Cahn-Hilliard energy functional}. Substituting $q=q_d+q_a$ in \eqref{eq:chnormal} leads to
\begin{align}
\partial^\circ u + u \tgrad\cdot \mathbf{V_\nu} + \tgrad \cdot (u\mathbf{V}_a) = \tgrad \cdot \left( M(u) \tgrad \left(- \Delta_\Gamma u + F'(u)\right)\right),
\end{align}
which we rewrite using the definition of the material time derivative to obtain the 4th order advective Cahn-Hilliard equation
\begin{align}
\partial^\bullet u + u\tgrad\cdot \mathbf{V}-\tgrad\cdot\left(u(\mathbf{V}_\tau-\mathbf{V}_a)\right) =\tgrad \cdot \left( M(u) \tgrad \left(- \Delta_\Gamma u + F'(u)\right)\right).
\end{align}
We will treat the problem above as the system of two second-order equations
\begin{align}
\begin{split}\label{eq:CHsystem_mob}
\partial^\bullet u + u\tgrad\cdot \mathbf{V} - \tgrad\cdot\left(u\mathbf{V}^\tau_a\right) &= \tgrad\cdot\left(M(u) \tgrad w\right)  \\
- \Delta_\Gamma u + F'(u) &= w 
\end{split} 
\end{align}
together with some initial condition for the function $u$, where we denote $\mathbf{V}^\tau_a := \mathbf{V}_\tau - \mathbf{V}_a$. There are no boundary conditions because the hypersurfaces are closed.

We shall make some simplifying assumptions on the system above. The thermodynamically relevant mobility term $M(u)$ should vanish at the pure phases $\pm 1$, effectively restricting diffusion to the interface between the two components. It is generally taken to be $M(u)=1-u^2$, and \eqref{eq:CHsystem} becomes a degenerate system. Even in the classical setting, there are not many results available for the degenerate Cahn-Hilliard equation (see \cite{EllGar96a, DaiDu14, DaiDu16}), and it is common to consider instead the constant mobility Cahn-Hilliard problem. We will do so in this text. Taking $M\equiv 1$, we are thus led to the system
\begin{align}
\begin{split}\label{eq:CHsystem}
\partial^\bullet u + u\tgrad\cdot \mathbf{V} - \tgrad\cdot\left(u\mathbf{V}^\tau_a\right) &= \Delta_\Gamma w  \\
- \Delta_\Gamma u + F'(u) &= w 
\end{split},
\end{align}
and this is the problem we analyse over the subsequent sections. We also make some assumptions on the different velocity fields involved in the derivation:
\begin{itemize}
\item[(i)] \eqref{eq:velbound} holds;
\item[(ii)] for the advective velocity $\mathbf{V}_\mathbf{a}$, we assume a uniform bound $\|\mathbf{V}_\mathbf{a}\|_{L^\infty}\leq C_\mathbf{a}$.
\end{itemize}

We now have all the conditions to formulate the weak form of \eqref{eq:CHsystem}.

\subsection{Problem setup}

We start by introducing some notation. Define, for $t\in [0,T]$, the bilinear forms:
\begin{itemize}
\item[(i)] for $\eta, \varphi\in L^2(\Gamma(t))$, the terms of order $0$
\begin{align}
m(t; \eta, \varphi) &:= \int_{\Gamma(t)} \eta \, \varphi \\
g(t; \eta, \varphi) &:= \int_{\Gamma(t)}\eta\, \varphi \,\nabla_{\Gamma(t)}\cdot \mathbf{V}(t);
\end{align}
\item[(ii)] for $\eta\in L^2(\Gamma(t))$, $\varphi\in H^1(\Gamma(t))$, the first order term
\begin{align}
a_N(t; \eta, \varphi) := \int_{\Gamma(t)} \eta \, \mathbf{V}_a^\tau(t) \cdot \nabla_{\Gamma(t)} \varphi;
\end{align}
\item[(iii)] for $\eta, \varphi\in H^1(\Gamma(t))$, the second order terms
\begin{align}
a_S(t; \eta, \varphi) &:= \int_{\Gamma(t)} \nabla_{\Gamma(t)} \eta \cdot \nabla_{\Gamma(t)} \varphi  \\
b(t; \eta, \varphi) &:= \int_{\Gamma(t)} \mathrm{B}(\mathbf{V}(t))\,\nabla_{\Gamma(t)} \eta\cdot\nabla_{\Gamma(t)} \varphi;
\end{align}
\item[(iv)] for $\eta\in H^{-1}(\Gamma(t))$, $\varphi\in H^1(\Gamma(t))$, the duality pairing
\begin{align}
m_\ast(t; \eta, \varphi) := \langle \eta, \varphi \rangle_{H^{-1}(\Gamma(t)), \, H^1(\Gamma(t))}
\end{align}
\end{itemize}
For simplicity of notation, we will from now on omit the dependence on $t\in [0,T]$. The weak form of the equations in \eqref{eq:CHsystem} then reads as
\begin{align}\label{eq:firsteq_setup}
\begin{split}
m_\ast (\partial^\bullet u, \eta) + g(u, \eta) + a_N(u, \eta) + a_S (w, \eta) &= 0, \\
a_S(u, \eta) + m(F'(u), \eta) - m(w, \eta) &= 0,
\end{split}
\end{align}
where the equations are satisfied for almost all $t\in [0,T]$ and all test functions $\eta\in L^2_{H^1}$. As usual, this is obtained by assuming $u, w$ are sufficiently smooth functions satisfying \eqref{eq:CHsystem}, multiplying the equations by an arbitrary $\eta\in L^2_{H^1}$, integrating over $\Gamma(t)$ and using integration by parts. 

As for the nonlinear term $F'(u)$, as we explained in the Introduction, it will be taken of smooth, logarithmic and double obstacle type, see \eqref{intro:smooth}, \eqref{intro:log}, \eqref{intro:obs}, and the second equation in \eqref{eq:firsteq_setup} is to be interpreted as a variational inequality for the third case. Regardless of the choice of the potential, we observe that, if $u$ solves the first equation of \eqref{eq:firsteq_setup}, then its integral is preserved over time; indeed, by testing the first equation with $\eta=1$, we obtain from Theorem \ref{thm:transportforms} that
\begin{align}\label{eq:integral_pres}
\dfrac{d}{dt} \int_{\Gamma(t)} u(t) = \langle \md u(t), 1\rangle_{H^{-1}, \, H^1} + \int_{\Gamma(t)} u(t) \tgrad\cdot\mathbf{V} = 0.
\end{align}
This is an important feature of the Cahn-Hilliard dynamics, and it will play an important role in the analysis of the equation, particularly for the study of the singular potentials.

\section{Smooth potentials}\label{sec:smooth}

In this section, we establish well-posedness of the Cahn-Hilliard system for the case of a potential $F$ satisfying, for some positive constants $\alpha_1, \alpha_2, \alpha_3, \alpha_4 > 0$ and real numbers $\beta_0, \beta_1, \beta_2, \beta_3, \beta_4\in \R$:
\begin{itemize}
\item[(A1)] $F(\varphi)\geq \beta_0$; \vskip 1mm
\item[(A2)] $F = F_1 + F_2$, where: \vskip 1mm
\begin{itemize}
\item[(A2.1)] $F_1$ and $F_2$ are $C^2(\R)$; \vskip 1mm
\item[(A2.2)] $F_1\geq 0$ is convex and satisfies, for some $q\in [1, \infty)$, $|F'_1(\varphi)| \leq \alpha_1 |\varphi|^q + \beta_1$; \vskip 1mm
\item[(A2.3)] $|F_1'(\varphi)| \leq \alpha_2 F_1(\varphi) + \beta_2$ and  $|\varphi F_1'(\varphi)| \leq \alpha_3 F_1(\varphi) + \beta_3$; \vskip 1mm
\item[(A2.4)] $|F'_2(\varphi)| \leq \alpha_4|\varphi| + \beta_4$.
\end{itemize}
\end{itemize}

The growth conditions above prescribe the behaviour of the potential at infinity, and are motivated by the ones given in \cite{GarKno20}. They require the nonlinearity to behave polynomially at infinity, and to be the sum of a convex term with a non-convex part which is essentially quadratic. These encapsulate the idea expressed in the introduction that we consider potentials whose graphs have a W-shape, and are satisfied in particular by the double-well polynomial potentials
\begin{align}
F(r) = \dfrac{r^4}{4} - \dfrac{r^2}{2}, \quad F(r) = \dfrac{(r^2-1)^2}{2},
\end{align}
which are frequently considered in the literature. In Section \ref{sec:nonsmooth}, we will establish well-posedness for the singular potentials by regularisation of these nonlinearities, and the conditions above are also general enough to be applied to these approximating problems.

The goal of this section is to prove well-posedness for the following problem.

\begin{problem}
Given $u_0\in H^1(\Gamma_0)$ and a potential $F$ satisfying the assumptions (A1)-(A2) above, the \textit{Cahn-Hilliard system with a smooth potential} is the following problem: find a pair $(u, w)$ such that:
\begin{itemize}
\item[(a)] $u\in H^1_{H^{-1}}\cap L^\infty_{H^1}$ and $w\in L^2_{H^1}$; \vskip 1mm
\item[(b)] the equations 
\begin{align}\label{eq:prob1}\tag{$\text{CH}_{\text{s}}$}
\begin{split}
m_\ast(\partial^\bullet u, \eta)+ g(u, \eta) + a_N(u, \eta) + a_S (w, \eta) &= 0, \\
a_S(u, \eta) + m(F'(u), \eta) - m(w, \eta) &= 0.
\end{split}
\end{align}
hold, for all $\eta\in L^2_{H^1}$ and almost every $t\in [0, T]$; \vskip 1mm
\item[(c)] $u(0)=u_0$ almost everywhere in $\Gamma_0$.
\end{itemize}
The pair $(u,w)$ is called a \textit{weak solution} of \eqref{eq:prob1}.
\end{problem}

\begin{remark}
Part (a) in the definition above ensures that our notion of a weak solution makes sense. Indeed:
\begin{itemize}
\item[(i)] the nonlinear term in \eqref{eq:prob1} is well-defined; since $u(t)\in H^1(\Gamma(t))$ for a.a. $t\in [0,T]$, the growth conditions (A1)-(A2) for $F$ imply that, for such $t$,
\begin{align}\label{eq:f_is_l2}
\begin{split}
\|F'(u(t))\|_{L^2(\Gamma(t))} &\leq C_1 \, \|\, |u(t)|^{q} \,\|_{L^2(\Gamma(t))} + C_2 \| \, |u| \, \|_{L^2(\Gamma(t))} \\
&\leq C_3 \, \|u\|_{H^1(\Gamma(t))}^q + C_2 \| u\|_{L^2(\Gamma(t))},
\end{split}
\end{align}
due to the Sobolev embedding theorem in Lemma \ref{lem:assumptionsgive}d), so that $F'(u(t))\in L^2(\Gamma(t))$;
\item[(ii)] the initial condition in (c) is also well-defined because $u\in H^1_{H^{-1}}\hookrightarrow C^0_{L^2}$, see Lemma \ref{prop:inclusions}b).
\end{itemize}  
\end{remark}

We prove existence of a solution by employing an evolving space Galerkin method.


\subsection{Galerkin approximation}\label{subsec:galerkin}

In order to define the approximating spaces, we pick a basis $\{\chi_j^0\colon j\in\N\}\subset H^1(\Gamma_0)$ consisting of smooth functions such that $\chi_0^1$ is constant, which we transport using the flow map to $\{\chi_j^t := \phi_t(\chi_j^0) \colon j\in\N\}\subset H^1(\Gamma(t))$ basis for $H^1(\Gamma(t))$. This definition implies the following transport formula for the basis functions,
\begin{align}
\md \chi_j^t \equiv 0, \quad \forall j\in \N,
\end{align}
which will be useful in setting up the approximating problem. We then define the approximation spaces as
\begin{align}
V_M(t) = \text{span}\{\chi_1^t, \dots, \chi_M^t\} \quad \text{ and } \quad L^2_{V_M} := \left\{\eta\in L^2_{H^1} \colon \eta(t)\in V_M(t), \,\, t\in [0,T]\right\}.
\end{align}
It follows that $L^2_{V_M}$ is dense in $L^2_{H^1}$ (and hence so it is in $L^2_{L^2}$). 

We consider also the $L^2$-projection operator 
\begin{align}
P_M^t\colon L^2(\Gamma(t))\to V_M(t) \subset H^1(\Gamma(t))
\end{align}
determined by the formula
\begin{align}
(P_M^t \eta - \eta, \varphi)_{L^2(\Gamma(t))} = 0, \quad \text{ for } \eta\in L^2(\Gamma(t)) \text{ and }  \varphi\in V_M(t).
\end{align}
It follows that $P_M^t$ satisfies, for all $\eta\in L^2(\Gamma(t))$,
\begin{align}
\|P_M^t \eta\|_{L^2(\Gamma(t))} \leq \|\eta\|_{L^2(\Gamma(t))} \quad \text{ and } \quad P_M^t \eta \to \eta \, \text{ in } L^2(\Gamma(t)).
\end{align}
We make the following additional assumption:

\vskip 3mm

\hskip -4mm \textbf{Assumption $\mathbf{(A_{P})}$}: We suppose that the projections $P_M^t$ satisfy:
\begin{itemize}
\item[(i)] At time $t=0$, if $\eta\in H^1(\Gamma_0)$ then 
\begin{align*}
\|P_M^0 \eta \|_{H^1(\Gamma_0)} \leq \|\eta\|_{H^1(\Gamma_0)} \quad \text{ and } \quad P_M^0 \eta \to \eta \,\, \text{ in } H^1(\Gamma_0).
\end{align*}
\item[(ii)] For any $\eps>0$, $t\in [0,T]$, and $\eta\in H^1(\Gamma(t))$, there exists $\tilde M\in \N$ such that we have the approximation estimate
\begin{align}\label{eq:projestimate}
\|P_M^t \eta- \eta\|_{L^2(\Gamma(t))} \leq \eps \,\|\eta\|_{H^1(\Gamma(t))}, \quad \forall M\geq \tilde M.
\end{align}
\end{itemize}

\vskip 3mm

The assumption above is reasonable, and satisfied by the usual choices of a Galerkin scheme; this can be seen with an easy calculation for the Fourier expansion, and in \cite{EllRan20} for a finite element approximation. We now set up the Galerkin approximation for \eqref{eq:prob1} in these spaces $L^2_{V_M}$ as follows.

\begin{problem}
The \textit{Galerkin approximation for} \eqref{eq:prob1} is the following problem:  for each $M\in\N$, find functions $u^M, w^M\in L^2_{V_M}$ with $\md u^M\in L^2_{V_M}$ such that, for any $\eta\in L^2_{V_M}$ and all $t\in [0, T]$,
\begin{align}\label{eq:prob1ga}\tag{$\text{CH}^M_{\text{s}}$}
\begin{split}
m_\ast(\md u^M, \eta) + g(u^M, \eta) + a_N(u^M, \eta) + a_S(w^M, \eta) &= 0, \\
a_S(u^M, \eta) + m(F'(u^M), \eta) - m(w^M, \eta) &= 0,
\end{split}
\end{align}
and $u^M(0)=P_M^0 u_0$ almost everywhere in $\Gamma_0$.
\end{problem}


We now proceed to find a local solution for the problem.

\begin{proposition}[Well-posedness for \eqref{eq:prob1ga}]
There exists a unique local solution pair to \eqref{eq:prob1ga}. More precisely, there exist functions $(u^M, w^M)$ satisfying \eqref{eq:prob1ga} on an interval $[0, t^*)$, $0<t^*\leq T$, together with the initial condition $u^M(0)=P_M^0 u_0$. The functions are of the form
\begin{align}\label{eq:linearcomb}
u^M(t) = \sum_{i=1}^M u_i^M(t) \chi_i^t, \quad w^M(t) = \sum_{i=1}^M w_i^M(t) \chi_i^t, \quad \quad t\in [0, t^*)
\end{align}
with coefficient functions $u_i^M\in C^1([0,t^*))$ and $w_i^M\in C^0([0, t^*))$, for every $i\in\{1, \dots, M\}$.
\end{proposition}

\begin{proof}
We write $u^M(t)=\sum_{i=1}^M u_i^M(t) \chi_i^t$ and $w^M(t)=\sum_{i=1}^M w_i^M(t) \chi_i^t$ and test \eqref{eq:prob1ga} with the basis function $\chi_j^t$ to write the problem above as
\begin{align}
\sum_{i=1}^M \dot{u}_i^M(t)m_{ij}(t) + u_i^M(t) g_{ij}(t) + u_i^M(t) a^N_{ij}(t) + w_i^M(t)a^S_{ij}(t)&=0 \\
\sum_{i=1}^M  u_i^M(t) a^S_{ij}(t) + f'_j(u^M(t)) - w_i^M(t) m_{ij}(t) &= 0,
\end{align}
which in matrix form reads as 
\begin{align}
M(t)\dot{\underline{u}}^M(t) + \underline{G}(t)(\underline{u}^M(t)) + A_N(t) \underline{u}^M(t) + A_S(t)\underline{w}^M(t) = 0 \\
 A_S(t)\underline{u}^M(t) +  \underline{F}'(\underline{u}^M(t)) - M(t)\underline{w}^M(t) = 0,
\end{align}
where we denote, for $i, j=1,\dots, M$ and any $t$, the solution vectors 
\begin{align}
\underline{u}^M(t) = (u_1^M(t), \dots, u_M^M(t)), \quad \underline{w}^M(t) = (w_1^M(t), \dots, w_M^M(t)),
\end{align}
the coefficient matrices
\begin{align}
\big(M(t)\big)_{ij} = m_{ij}(t):= m(t; \chi_i^t, \chi_j^t), &\quad \big(G(t)\big)_{ij} = g_{ij}(t) := g(t; \chi_i^t, \chi_j^t), \\
\big(A_S(t)\big)_{ij}= a^S_{ij}(t) := a_S(t; \chi_i^t, \chi_j^t), &\quad \big(A_N(t)\big)_{ij}= a^N_{ij}(t) := a_N(t; \chi_i^t, \chi_j^t)
\end{align}
and the nonlinear term by
\begin{align}
\underline{F}'(\underline{u}^M(t))_j &= F_j'(u^M(t)) := m(t; F'(u^M(t)), \chi_j^t).
\end{align}
Solving the first equation for $\underline{w}^M(t)$  and substituting on the second we obtain a semilinear first-order ODE for $\underline{u}^M$ together with an initial condition $u^M(0)=P_M^0 u_0$. Since $F'$ is $C^1$, the result follows from the general theory of ODEs.
\end{proof}
We are now ready to establish a priori bounds for the equation. Before we state the next result, let us recall the definition of the Cahn-Hilliard energy functional:
\begin{align}
\mathrm{E}^\mathrm{CH}[u] = \int_{\Gamma(t)} \dfrac{|\nabla_{\Gamma(t)} u|^2}{2} + F(u).
\end{align}

\begin{proposition}\label{prop:propenergy}
There exists $\tilde{M}\in\N$ such that the local solution $(u^M,w^M)$ of \eqref{eq:prob1ga} satisfies the energy estimate
\begin{align}\label{eq:energy}
\sup_{[0, t^*)}\mathrm{E}^\mathrm{CH}[u^M] + \dfrac{1}{2}\int_0^{t^*} \|\nabla_\Gamma w^M\|_{L^2}^2 &\leq C, \quad \text{ for } M\geq \tilde{M},
\end{align}
where the constant $C>0$ depends only on the final time $T$, the constants in (A1)-(A2) for the potential $F$ and the initial condition $u_0$. In particular, solutions $(u^M, w^M)$ are defined on $[0,T]$, and we have
\begin{align}\label{eq:energy}
\sup_{[0, T]}\mathrm{E}^\mathrm{CH}[u^M] +  \dfrac{1}{2}\int_0^{T} \|\nabla_\Gamma w^M\|_{L^2}^2  &\leq C, \quad \text{ for } M\geq \tilde{M}.
\end{align}
\end{proposition}

\begin{proof}
Differentiate the energy functional and use the equations \eqref{eq:prob1ga} to obtain
\begin{align}
\dfrac{d}{dt}\mathrm{E}^\mathrm{CH}[u^M] &= a_S\left(\partial^\bullet u^M, u^M\right) + b\left(u^M, u^M\right) + m\left(F'(u^M), \partial^\bullet u^M\right) + g\left(F(u^M), 1\right) \\
&= m\left(w^M, \partial^\bullet u^M\right) + b\left(u^M, u^M\right) + g\left(F(u^M), 1\right) \\
&= - g\left(u^M, w^M\right) - a_N\left(u^M, w^M\right) - a_S(w^M, w^M) + b\left(u^M, u^M\right) \\
&\hskip 8cm + g\left(F(u^M), 1\right),
\end{align}
which gives
\begin{align}\label{eq:1}
\dfrac{d}{dt}\mathrm{E}^\mathrm{CH}[u^M] + \|\tgrad w^M\|^2\L = - g(u^M, w^M) - a_N\left(u^M, w^M\right) + b\left(u^M, u^M\right) + g\left(F(u^M), 1\right)
\end{align}

To estimate the terms on the right hand side, we use the assumptions on the velocity field to obtain, for the last two terms,
\begin{align}\label{eq:eqdependent}
b\left(u^M, u^M\right) + g\left(F(u^M), 1\right) &\leq C_\mathbf{V} \|\tgrad u^M\|^2\L + C_\mathbf{V}\, m\left(F(u^M), 1\right) + C \\
&\leq C + C_\mathbf{V} \, \mathrm{E}^\mathrm{CH}[u^M].
\end{align}
Using the assumption on the tangential velocities and Young's inequality yields
\begin{align}
|a_N (u^M, w^M)| \leq C \|u^M\|^2_{L^2} + \dfrac{1}{2} \|\tgrad w^M\|^2_{L^2}.
\end{align}
Observe that testing the first equation with $\eta=1$ shows that the integral of $u^M$ is preserved, exactly as in \eqref{eq:integral_pres}, and thus there exist $C_1, C_2>0$ such that
\begin{align}\label{eq:l2h1estimate}
\|u^M\|^2_{L^2} \leq C_1 \|\tgrad u^M\|^2_{L^2} + C_2,
\end{align}
from where
\begin{align}
|a_N(u^M, w^M)| \leq C_1 \|\tgrad u^M\|^2_{L^2} + C_2 + \dfrac{1}{2} \|\tgrad w^M\|^2_{L^2}.
\end{align}
Combining the above leads to 
\begin{align}
\dfrac{d}{dt} \mathrm{E}^\mathrm{CH}[u^M] + \dfrac{1}{2}\|\tgrad w^M\|^2_{L^2} \leq -g(u^M, w^M) +  C_0 + C_1\, \mathrm{E}^\mathrm{CH}[u^M].
\end{align}
Now we note that by definition of the projections $P_M^t$ we have
\begin{align}
g(u^M, w^M) = m\big(w^M, P_M(u^M\tgrad\cdot\mathbf{V})\big),
\end{align}
and so testing the second equation with $\eta=P_M(u^M\tgrad\cdot\mathbf{V})$ leads to 
\begin{align}
\begin{split}\label{eq:alternativeg}
|g(u^M, w^M)| &\leq \big|a_S \big(u^M, u^M\tgrad\cdot\mathbf{V}\big)\big| + \big| m \big( F'(u^M), P_M(u^M\tgrad\cdot\mathbf{V})\big) \big| \\
&\leq C_0 + C_1 \|\tgrad u^M\|^2_{L^2} + \big| m \big( F'(u^M), P_M(u^M\tgrad\cdot\mathbf{V})\big) \big| \\
&\leq C_0 + C_1 \mathrm{E}^\mathrm{CH}[u^M] +  \big| m \big( F'(u^M), P_M(u^M\tgrad\cdot\mathbf{V})\big) \big|
\end{split}
\end{align}
As for the remaining term, we note that
\begin{align}
\big| m \big( F'(u^M), P_M(u^M\tgrad\cdot\mathbf{V})\big) \big| &\le \big| m \big( F'(u^M), u^M\tgrad\cdot\mathbf{V}) \big| \\
&\hskip 1cm + \big| m \big( F'(u^M), \left(P_M(u^M\tgrad\cdot\mathbf{V}) -u^M\tgrad\cdot\mathbf{V}\right) \big) \big| \\
&\hskip -2cm \leq C_1 \, m(F(u^M), 1) + \|F'(u^M)\|_{L^2} \, \|P_M(u^M\tgrad\cdot\mathbf{V}) -u^M\tgrad\cdot\mathbf{V}\|_{L^2}.
\end{align}
Using Assumption $\mathbf{(A_P)}$, for any $\eps>0$ we can  choose $\tilde M\in \N$ sufficiently large so that
\begin{align}
\|P_M(u^M\tgrad\cdot\mathbf{V}) -u^M\tgrad\cdot\mathbf{V}\|_{L^2} \leq \eps \, \|u^M\,\tgrad\cdot \mathbf{V}\|_{H^1} \leq C \, \eps \, \|u^M\|_{H^1},
\end{align}
which leads to, combining the estimate above with Remark \eqref{eq:f_is_l2},
\begin{align}
\big| m \big( F'(u^M), P_M(u^M\tgrad\cdot\mathbf{V})\big) \big| &\leq C_1 \, \mathrm{E}^\mathrm{CH}(u^M) + C_2 \, \eps \, \|u^M\|^{q+1}_{H^1} \\
&\leq C_0 + C_1 \, \mathrm{E}^\mathrm{CH}(u^M) + C_3\, \eps \left(\mathrm{E}^\mathrm{CH}[u^M]\right)^{q+1}. 
\end{align}
All in all, we have proved that
\begin{align}\label{eq:energyfinal}
\dfrac{d}{dt} \mathrm{E}^\mathrm{CH}[u^M] + \dfrac{1}{2}\|\tgrad w^M\|^2_{L^2} \leq C \left( C_0 + \, \mathrm{E}^\mathrm{CH}[u^M] + \, \eps \, \left(\mathrm{E}^\mathrm{CH}[u^M]\right)^{q+1} \right),
\end{align}
for some constants $C, C_0>0$ independent of both $M$ and $t$. By picking $\eps>0$ small enough, we can apply the generalised Gronwall inequality in Lemma \ref{lem:gengronwall} to obtain a uniform bound 
\begin{align}\label{eq:energyfinal}
\sup_{[0,t_\ast]} \mathrm{E}^\mathrm{CH}[u^M] + \int_0^{t_\ast} \dfrac{1}{2}\|\tgrad w^M\|^2_{L^2} \leq C_1 + C_2 \mathrm E^\mathrm{CH}(P_M^0 u_0) \leq C_3, \,\, M \geq \tilde{M},
\end{align}
where $C_3$ is independent of $M$ and $t_\ast$ due to Assumption $\mathbf{(A_P)}$(i).
\end{proof}

\begin{remark}\label{rem:extraassump}
Observe that Assumption $\mathbf{(A_P)}$(ii) was used in order to obtain a small coefficient in front of the higher order term in \eqref{eq:energyfinal}. However, in the case that the potential has quadratic growth at infinity (i.e. $q=1$ in (A2.2)), we can obtain the global energy estimate without resorting to Assumption $\mathbf{(A_P)}$. Indeed, in this case we would have $\|F'(u^M)\|_{L^2} \leq C_1 + C_2 \|u^M\|_{L^2}$, and so we directly replace \eqref{eq:alternativeg} with
\begin{align}
|g(u^M, w^M)| \leq C_0 + C_1 \|\tgrad u^M\|^2_{L^2} + C_2 \|u^M\|^2_{L^2} \leq C_3 + C_4 \mathrm{E}^\mathrm{CH}[u^M], 
\end{align}
which then leads, instead of \eqref{eq:energyfinal}, to 
\begin{align}
\dfrac{d}{dt} \mathrm{E}^\mathrm{CH}[u^M] + \dfrac{1}{2} \|\tgrad w^M\|^2_{L^2} \leq C_1 + C_2 \mathrm{E}^\mathrm{CH}[u^M].
\end{align}
The same conclusion as above now follows from the classical Gronwall inequality, valid for all $M$. 
\end{remark}
%
%

For the rest of the section, we always consider $M\geq \tilde{M}$ so that the energy estimate above is satisfied. Observe that, unlike the case of a stationary domain, in our setting the energy \eqref{eq:energy} does not necessarily decrease along a solution to the Cahn-Hilliard equation. It is also important to note that the constants involved in the previous estimates are independent of $M$ but rely strongly on those constants appearing in assumptions (A1)-(A2) for the potential. 

As a consequence of the result above:

\begin{corollary}[A priori bounds]\label{cor:cor1}
We have uniform bounds for $u^M$, $w^M$
in $L^\infty_{H^1}$, $L^2_{H^1}$, 
respectively. More precisely, 
\begin{align}
 \sup_{t\in [0,T]} \|u^M(t)\|_{H^1(\Gamma(t))} + \int_0^T \|w^M(t)\|_{H^1(\Gamma(t))}^2  \leq C 
\end{align}
where $C>0$ is a constant depending only on the initial condition, on the potential and on the final time $T$. 

In particular, we also have that $u^M$ is bounded in $L^\infty_{L^p}$, for all $p\in [1,+\infty)$, and $F'(u^M)$ is bounded in $L^\infty_{L^2}$. 
\end{corollary}

\begin{proof}
The energy estimate \eqref{eq:energy} immediately yields a uniform bound for $\nabla_\Gamma u^M$ in $L^\infty_{L^2}$, and so \eqref{eq:l2h1estimate} implies that $u^M$ is uniformly bounded in $L^\infty_{H^1}$.

To show the bound for the $H^1$-norm of $w^M$, we first note that testing the second equation with $\eta=1$ leads to 
\begin{align}
m \left( w^M, 1\right) = m \left( F'(u^M), 1 \right),
\end{align}
and this is uniformly bounded due to \eqref{eq:f_is_l2} and the uniform estimate for $u^M$ in $H^1$. Combined with the uniform bound for $\nabla_\Gamma w^M$ in $L^2_{L^2}$ given by the previous result, a uniform bound for $w^M$ in $L^2_{H^1}$ follows again from an application of Poincar\'{e}'s inequality.

Finally, the fact that $u^M$ is bounded in $L^\infty_{L^p}$ follows from the Sobolev embedding result in Lemma \ref{lem:assumptionsgive}d), and this combined with the growth assumptions on $F$ imply that $F'(u^M)$ is bounded in $L^\infty_{L^2}$. 

\end{proof}

\subsection{Passage to the limit}

The a priori bounds in the previous result allow us to obtain limit functions $u\in L^\infty_{H^1}$ and $w\in L^2_{H^1}$ such that, as $M\to\infty$ and up to taking subsequences,
\begin{align}
u^M \overset{\ast}{\rightharpoonup} u \, \text{ in } \, L^\infty_{H^1} \quad \text{ and } \quad w^M \rightharpoonup w \, \text{ in } \, L^2_{H^1}.
\end{align}

We obtain the stronger convergence $u^M\to u$ in $L^2_{L^2}$ by using Lemma \ref{lem:temam_general}. Since the condition in (ii) follows from the uniform bound for $u^M$ in $L^\infty_{H^1}$, it suffices to prove:

\begin{lemma}
For a.e. $t\in [0,T]$, $u^M(t) \rightharpoonup u(t)$ in $L^2(\Gamma(t))$.
\end{lemma}

\begin{proof}
Define $$f^M(t) = (u^M(t), \chi_j^t)_{L^2(\Gamma(t))} \quad \text{ and } \quad f(t) = (u(t), \chi_j^t)_{L^2(\Gamma(t))}.$$ Using the equations for $u^M$ we have, for $M$ sufficiently large and $t\in [0,T]$,
\begin{align}\label{eq:eq_1_new}
f^M(t) - f^M(0)= \int_0^t a_S(w^M, \chi_j^\tau) - a_N(u^M, \chi_j^\tau) \, \d\tau.
\end{align}
so that, for $t,s\in [0,T]$ with $s<t$, we have
\begin{align*}
f^M(t) - f^M(s) = \int_s^t a_S(w^M, \chi_j^\tau) - a_N(u^M, \chi_j^\tau) \, \d\tau
\end{align*}
and therefore
\begin{align*}
|f^M(t) - f^M(s)|  
\leq C_1 \|\chi_j^0\|_{H^1(\Gamma_0)} |t-s|^{1/2} 
\end{align*}
where the constant $C$ depends only on the uniform bounds for $u^M$, $w^M$ in $L^2_{H^1}$ and compatibility of the pair $(H^1(\Gamma(t)), \phi_t)_t$. The estimate above implies equicontinuity of the sequence $(f^M)_M$. A diagonal argument also shows that $f^M(t) \to f(t)$ for $t$ in a countable, dense subset of $[0,T]$ and hence \eqref{eq:eq_1_new} follows from \cite[Problem 19.14c]{Zei90}. This now implies by linearity that for any $N\in\N$ and $v\in V_N(t)$ we have
\begin{align*}
(u^M(t) - u(t), v)_{L^2 (\Gamma(t))} \to 0.
\end{align*}
Now for general $v\in L^2(\Gamma(t))$ we can take $v^N\in V_N(t)$ such that $v^N\to v$ in $L^2(\Gamma(t))$ and we then have, using the above and the continuity of the inner product,
\begin{align*}
|(u^M(t) - u, v)_{L^2(\Gamma(t))}| &\leq |(u^M(t)-u, v-v^N)_{L^2(\Gamma(t))}| + |(u^M(t)-u, v^N)_{L^2(\Gamma(t))}| \\
&\leq C \|v - v^N\|_{L^2(\Gamma(t))} + |(u^M(t) - u, v^N)_{L^2(\Gamma(t))}|
\end{align*}
Letting $M\to\infty$ and then $N\to\infty$ yields the conclusion, finishing the proof.
\end{proof}

By Theorem \ref{lem:temam_general}, it follows that $u^M\to u$ in $L^2_{L^2}$. As a consequence, also up to a subsequence we have
\begin{align}\label{eq:conv_new}
u^M(t) \to u(t) \, \text{ in } L^2(\Gamma(t)) \quad \text{ and } \quad u^M(t) \to u(t) \, \text{ pointwise a.e. in } \Gamma(t).
\end{align}

The Sobolev embedding in Lemma \ref{lem:assumptionsgive}d) additionally implies $u\in L^\infty_{L^p}$, for all $p\in [1,+\infty)$. Due to continuity of $F'$, from \eqref{eq:conv_new} it follows that $F'(u^M(t))\to F'(u(t))$ pointwise a.e. in $\Gamma(t)$. We also observe that, using the equation and the transport formula in \eqref{thm:transportforms}, we have
\begin{align}
\dfrac{1}{2}\dfrac{d}{dt} \|u^M\|_{L^2(\Gamma(t))}^2 &= \, m(\md u^M, u^M) + \dfrac{1}{2} g(u^M, u^M) \\
&= -a_N(u^M, u^M)-a_S(w^M, u^M) - \dfrac{1}{2} g(u^M, u^M).
\end{align}
Integrating over $[0,T]$ and using the a priori bounds leads to
\begin{align}
\|u^M(T)\|_{L^2(\Gamma(T))}^2 \leq C_1 \|P_M^0 u_0\|_{L^2}^2 + C_2 \leq C_3,
\end{align}
from where we additionally obtain $z\in L^2(\Gamma(T))$ such that $u^M(T)\rightharpoonup z$ in $L^2(\Gamma(T))$.

\begin{remark}
It is important to note that we have not yet obtained the existence of a weak time derivative for the limit $u$. In the classical setting, $\md u$ is obtained as a limit of the sequence $(\md u^M)_M$, but in the current time-dependent framework we do not  obtain a uniform estimate for $(\md u^M)_M$ in $L^2_{H^{-1}}$ by the usual duality arguments. In particular, we do not have a uniform estimate for $(u^M)_M$ in $H^1_{H^{-1}}$, which precludes us from applying the Aubin-Lions compactness lemma in Proposition \ref{prop:inclusions}(c) and obtaining stronger convergence results for $(u^M)_M$. 
\end{remark}

In the next result we show the existence of $\md u$ by integrating the first equation by parts, in order to transfer the time derivative to the test functions, and passing the obtained equations to the limit.

\begin{proposition}\label{prop:existence}
Let $(u,w,z)$ be the limit functions above.
\begin{itemize}
\item[(i)] There exists $\md u\in L^2_{H^{-1}}$, and $u\in C^0_{L^2}$.
\item[(ii)] We have $u(0)=u_0$ and $u(T)=z$.
\end{itemize}
\end{proposition}

\begin{proof}
\underline{(i).} For any $\eta\in L^2_{V_M}$ with $\md \eta\in L^2_{V_M}$, integrating over $[0,T]$ we can write the first equation of the system as
\begin{align}
m(u^M(T), \eta(T)) - m(P_M u_0, \eta(0)) + \int_0^T a_N(u^M, \eta) + a_S(w^M, \eta) - m(u, \md \eta) = 0
\end{align}
For $j \leq M$, take $\eta(t) = \psi(t)\chi_j^t$ with $\psi \in C^1([0,T])$ to get
\begin{align}
m(u^M(T), \psi(T)\chi_j^T) - m(P_M u_0, \psi(0)\chi_j^0) &+ \int_0^T \psi(t) a_N(u^M, \chi_j^t) + \psi(t) a_S(w^M, \chi_j^t) \\
&\hskip 1cm = \int_0^T \psi'(t) m(u^M, \chi_j^t) = 0.
\end{align}
Passing to the limit $M \to \infty$, we obtain 
\begin{align}
\begin{split}\label{eq:initial_cond_12}
m(z, \psi(T)\chi_j^T) - m(u_0, \psi(0)\chi_j^0) &+ \int_0^T \psi(t) a_N(u, \chi_j^t) + \psi(t) a_S(w, \chi_j^t) \\
&\hskip 1cm = \int_0^T \psi'(t) m(u, \chi_j^t).
\end{split}
\end{align}

Given $\eta \in H^1(\Gamma_0)$, there exist coefficients $a_j \in \mathbb{R}$ and a sequence $\eta^M = \sum_{j=1}^M a_j \chi_j^0$ (note $\eta^M\in V_M(0)$ for all $M$) such that $\eta^M \to \eta$ in $H^1(\Gamma_0)$. Hence $\phi_t \eta^M = \sum_{j=1}^M a_j \chi_j^t$ converges to $\phi_t \eta$ in $H^1(\Gamma(t)).$ Multiplying \eqref{eq:initial_cond_12} by $a_j$ and summing over $j=1, ..., M$ gives
\begin{align}
    m(z, \psi(T) \phi_T \eta^M) - m(u_0, \psi(0) \eta^M) &+ \int_0^T \psi(t) \, a_N(u, \phi_t \eta^M) + \psi(t) \, a_S (w, \phi_t \eta^M) \\
&\hskip 1cm= \int_0^T \psi'(t) \, m(u, \phi_t \eta^M)
\end{align}
and if we furthermore take $\psi \in \mathcal{D}(0,T)$ this simplifies to
\begin{align}
\int_0^T \psi(t) \, a_N(u, \phi_t \eta^M) + \int_0^T \psi(t) \, a_S (w, \phi_t \eta^M) = \int_0^T \psi'(t) \, m(u, \phi_t \eta^M)
\end{align}
Letting $M \to \infty$ yields
\begin{align}
\int_0^T \psi(t) \, a_N(u, \phi_t \eta) + \int_0^T \psi(t) \, a_S(w, \phi_t\eta) = \int_0^T \psi'(t) \, m(u, \phi_t \eta)
\end{align}
which means that $t\mapsto m(u(t), \phi_t \eta)$ is weakly differentiable with
\begin{align}
\dfrac{d}{dt} m(u(t), \phi_t \eta) = - a_N(u, \phi_t \eta) - a_S(w, \phi_t \eta)
\end{align}
It then follows from Lemma \ref{lem:assumptionsgive} that $u\in H^1_{H^{-1}}$ and that $\md u\in L^2_{H^{-1}}$ satisfies
\begin{align}\label{eq:firsteq}
    \int_0^T m_* (\md u, \eta) + g(u, \eta) + a_N(u, \eta) + a_S(w, \eta) = 0,
\end{align}
as desired. The extra regularity $u\in C^0_{L^2}$ is a consequence of the continuous embedding $H^1_{H^{-1}}\hookrightarrow C^0_{L^2}$, see Proposition \ref{prop:inclusions}.

\underline{(ii).} Let $\eta\in H^1_{H^{-1}}$. Using the transport formula in Theorem \ref{thm:transportforms} and the equation for $u$, we have
\begin{align}
m(u(T), \eta(T)) - m(u(0), \eta(0)) &= \int_0^T m_*(\md u, \eta) + g(u, \eta) + m_*(\md \eta, u) \\
&= \int_0^T m_*(\md \eta, u) - a_N(u, \eta) - a_S(w, \eta)
\end{align}
Taking $\eta(t) = \psi(t) \chi_j^t$ for any $j \in \N$ and $\psi \in C^1([0,T])$, we obtain 
\begin{align}
\begin{split}
 m(u(T), \psi(T) \chi_j^T) - m(u(0), \psi(0) \chi_j^0) = \int_0^T \psi'(t) m(u, \chi_j^t) &- \int_0^T \psi(t) a_N(u, \chi_j^t) \\
 &\hskip 6mm - \int_0^T \psi(t) a_S(u, \chi_j^t)
    \label{eq:initial_cond_1}
\end{split}
\end{align}
Comparing this with \eqref{eq:initial_cond_12}, we see that
\begin{align}
m(u(T),\psi(T)\chi_j^T) - m(u(0), \psi(0)\chi_j^0)  = m(z, \psi(T)\chi_j^T) - m(u_0, \psi(0)\chi_j^0).
\end{align}
Picking $\psi$ with $\psi(0)=1$, $\psi(T)=0$ (resp. $\psi(0)=0$, $\psi(T)=1$) we obtain $u(0)=u_0$ (resp. $u(T)=z$).
\end{proof}

We can now show that indeed the limit pair $(u,w)$ solves \eqref{eq:prob1}.

\begin{proposition}\label{prop:smoothexistence}
The limit pair $(u,w)$ is a solution to \eqref{eq:prob1}. 
\end{proposition}

\begin{proof}
We have already proved that $u,w$ lie in the desired spaces, that $u(0)=u_0$, and it also follows from the previous proof that
\begin{align}\label{eq:new1}
\int_0^T m_*(\md u, \eta) + g(u, \eta) + a_N(u, \eta) + a_S(w, \eta) = 0, \quad \forall \eta\in L^2_{H^1}.
\end{align}

To obtain the second equation, let $\eta\in L^2_{H^1}$ and take $\eta^M\in L^2_{V_M}$ such that $\eta^M\to \eta$ in $L^2_{H^1}$. We immediately obtain for the linear terms that
\begin{align}\label{eq:limit_linear}
\int_0^T a_S(u^M, \eta^M)  &\longrightarrow \int_0^T  a_S( u, \eta)  , \quad \int_0^T m(w^M,\eta^M)  \longrightarrow \int_0^T  m(w,\eta) , 
\end{align}
For the nonlinear term, we note that $F'(u^M)\in L^2_{L^2}$ converges pointwise a.e. to $F'(u)\in L^2_{L^2}$ and satisfies $\|F'(u^M)\|_{L^2_{L^2}}\leq C$ due to the a priori bounds, so that the generalised Dominated Convergence Theorem \ref{lem:weak_dct} implies
\begin{align*}
F'(u^M) \rightharpoonup  F'(u) \quad \text{ in } L^2_{L^2},
\end{align*}
from where we also obtain
\begin{align}\label{eq:limit_nonlinear}
\int_0^T m\left(F'(u^M), \eta^M\right) \to \int_0^T m\left(F'(u), \eta\right).
\end{align}

Combining \eqref{eq:limit_linear} and \eqref{eq:limit_nonlinear} then gives
\begin{align}\label{eq:new2}
\int_0^T a_S(u,\eta) + m(F'(u), \eta) - m(w,\eta) \, dt &= 0.
\end{align}

Considering now test functions of the form $\psi(t) \eta(t,x)$ with $\psi\in C_c^\infty(0,T)$ and $\eta\in L^2_{H^1}$, equations \eqref{eq:new1} and \eqref{eq:new2} read as
\begin{align}
\int_0^T \psi \, \big(m_\ast(\partial^\bullet u, \eta)+ g(u,\eta) + a_N(u,\eta) + a_S(w,\eta) \big) \, dt &= 0 \\
\int_0^T \psi \, \big(a_S(u, \eta) + m(F'(u), \eta) - m(w,\eta)\big) \, dt &= 0,
\end{align}
and hence, for almost all $t\in [0,T]$, by the fundamental theorem of calculus of variations we obtain
\begin{align}
m_\ast(\partial^\bullet u, \eta) + g(u,\eta) + a_N(u,\eta) + a_S(w,\eta) &= 0, \\
a_S(u,\eta) + m(F'(u), \eta) - m(w,\eta) &= 0.
\end{align}
In other words, $(u,w)$ is a solution to \eqref{eq:prob1}, as desired.
\end{proof}

We finally establish stability of solutions with respect to initial data under an additional assumption on the non-convex part $F_2$ of the potential. See Appendix \ref{app:invlaplacian} for the definition of $\|\cdot\|_{-1}$.

\begin{proposition}[Stability]\label{prop:smoothuniqueness}
Suppose $F_2'$ is Lipschitz. For any $u_{1, 0}, u_{2, 0}\in H^1(\Gamma_0)$ with $(u_{1,0})_{\Gamma_0}=(u_{2,0})_{\Gamma_0}$, if $u_1, u_2$ denote the solutions of \eqref{eq:prob1} with $u_1(0)=u_{1,0}$ and $u_2(0)=u_{2,0}$, then there exist $C>0$, independent of $t$, such that, for almost all $t\in [0,T]$,
\begin{align}\label{eq:stability}
\|u_1(t)-u_2(t)\|_{-1}^2 \leq e^{Ct} \|u_{1,0}-u_{2,0}\|_{-1}^2.
\end{align}
In particular, if $F_2'$ is Lipschitz, there exists at most one weak solution to the Cahn-Hilliard system \eqref{eq:prob1}.
\end{proposition}

\begin{proof}
Suppose that we have two solution pairs $(u_1, w_1)$ and $(u_2, w_2)$, and let us denote $\xi^u=u_1-u_2$ and $\xi^w = w_1-w_2$. Subtracting the corresponding equations yields, for any $\eta\in L^2_{H^1}$,
\begin{align}
m_\ast(\md\xi^u, \eta) + g(\xi^u, \eta) + a_N(\xi^u, \eta) + a_S(\xi^w, \eta) &= 0, \label{eq:eq1} \\
a_S(\xi^u, \eta) + m(F'(u_1)-F'(u_2), \eta) - m(\xi^w, \eta) &= 0. \label{eq:eq2}
\end{align}
Recall that the mean value of both $u_1$ and $u_2$ must be preserved, and thus, for any $t\in [0, T]$, $\xi^u(t)$ has zero mean value over $\Gamma(t)$. Thus the inverse Laplacian  $\mathcal{G}\xi^u$ of $\xi^u$ is well defined and an element of $H^1_{H^{-1}}$ (see Appendix \ref{app:invlaplacian}), and we can test \eqref{eq:eq1} with $\mathcal{G}\xi^u$ to get
\begin{align}
m_\ast(\md\xi^u, \mathcal{G}\xi^u) + g(\xi^u, \mathcal{G}\xi^u) + a_N(\xi^u, \mathcal{G}\xi^u) + a_S(\xi^w, \mathcal{G}\xi^u) &= 0
\end{align}
which is equivalent to
\begin{align}
\begin{split}
\label{eq:eq3}
\dfrac{d}{dt} \|\xi^u\|^2_{-1} + a_N(\xi^u, \mathcal{G}\xi^u) + m(\xi^w, \xi^u) = m(\xi^u, \partial^\bullet \mathcal{G}\xi^u) = a_S (\mathcal{G}\xi^u, \partial^\bullet \mathcal{G}\xi^u).
\end{split}
\end{align}
Testing now \eqref{eq:eq2} with $\xi^u$ yields
\begin{align}\label{eq:truetest}
 \|\tgrad \xi^u\|\L^2 + m(F_1'(u_1)-F_1'(u_2), \xi^u) + m(F_2'(u_1)-F_2'(u_2), \xi^u)= m(\xi^w, \xi^u).
\end{align}
We estimate the nonlinear terms as follows. For the first term, observe that convexity of $F_1$ implies that $F_1'$ is monotone, and thus $$m(F'_1(u_1)-F'_1(u_2), \xi^u)\geq 0.$$ As for the second term, we use Lipschitz continuity of $F_2'$ to obtain $$|m(F_2'(u_1)-F_2'(u_2), \xi^u)\L| \leq L\|\xi^u\|^2\L$$ where $L>0$ is some positive constant. Hence, from \eqref{eq:truetest} we obtain
\begin{align}\label{eq:eq4}
\|\tgrad\xi^u\|\L^2 \leq m(\xi^w, \xi^u) +  C \|\xi^u\|\L^2.
\end{align}
Adding \eqref{eq:eq3} and \eqref{eq:eq4} we get
\begin{align}
\dfrac{d}{dt}\|\xi^u\|_{-1}^2 +  \|\tgrad \xi^u\|\L^2 &\leq C \|\xi^u\|\L^2 + a_S(\mathcal{G}\xi^u, \partial^\bullet \mathcal{G}\xi^u) - a_N (\xi^u, \mathcal{G}\xi^u) \\ 
&\leq C_1 \|\xi^u\|^2_{L^2} + a_S(\mathcal{G}\xi^u, \md \mathcal{G}\xi^u) + C_2 \|\xi^u\|^2_{-1}
\end{align}

We now estimate the terms on the right hand side. For the first one, we can use Young's inequality to get
\begin{align}
C_1 \|\xi^u\|\L^2 =C_1 a_S(\xi^u, \mathcal{G}\xi^u) \leq \dfrac{1}{2} \|\tgrad \xi^u\|\L^2 + C_2\|\xi^u\|^2_{-1},
\end{align}
and for the second one we have
\begin{align}
a_S (\mathcal{G}\xi^u, \partial^\bullet \mathcal{G}\xi^u)\L 
&=  \dfrac{1}{2} \dfrac{d}{dt} a_S (\mathcal{G}\xi^u, \mathcal{G}\xi^u) - \dfrac{1}{2}b(\mathcal{G}\xi^u, \mathcal{G}\xi^u) \leq \dfrac{1}{2}\dfrac{d}{dt}\|\xi^u\|_{-1}^2 + C\|\xi^u\|_{-1}^2.
\end{align}
In conclusion,
\begin{align}\label{eq:uniqend}
\dfrac{d}{dt}\|\xi^u\|_{-1}^2 +  \|\tgrad \xi^u\|\L^2 \leq C \|\xi^u\|_{-1}^2,
\end{align}
and an application of Gronwall's inequality implies \eqref{eq:stability}.

If $u_1(0)=u_2(0)$, then it follows from \eqref{eq:uniqend} that $\xi^u$ is constant, and since it has zero mean value it must be $u_1=u_2$. From \eqref{eq:eq2} we obtain $w_1=w_2$, giving uniqueness.
\end{proof}

We summarize our findings of this section in the following result.

\begin{theorem}\label{thm:wellposedsmooth}
Let $u_0\in H^1(\Gamma_0)$ and $F\colon\R\to\R$ be a potential satisfying assumptions (A1)-(A2). Then, there exists a pair $(u, w)$ with
\begin{align}
u\in H^1_{H^{-1}}\cap L^\infty_{H^1} \quad \text{ and } \quad w\in L^2_{H^1}
\end{align}
satisfying, for all $\eta\in L^2_{H^1}$ and a.a. $t\in [0, T]$,
\begin{align}\label{eq:prob1}\tag{$\text{CH}_{\text{s}}$}
\begin{split}
m_\ast(\md u, \eta) + g(u, \eta) + a_N(u, \eta) + a_S(w,\eta) &= 0, \\
a_S(u, \eta) + m(F'(u), \eta) - m(w,\eta) &= 0,
\end{split}
\end{align}
and $u(0)=u_0$ almost everywhere in $\Gamma_0$. The solution $u$ satisfies the additional regularity 
\begin{align}
u\in C^0_{L^2} \cap L^\infty_{L^p}, \quad \text{ for all } p\in [1,+\infty).
\end{align}

Furthermore, provided $F'_2$ is Lipschitz, then $u_{1, 0}, u_{2, 0}\in H^1(\Gamma_0)$ with $(u_{1,0})_{\Gamma_0}=(u_{2,0})_{\Gamma_0}$, if $u_1, u_2$ denote the solutions of \eqref{eq:prob1} with $u_1(0)=u_{1,0}$ and $u_2(0)=u_{2,0}$, then there exist $C>0$, independent of $t$, such that, for almost all $t\in [0,T]$,
\begin{align}\label{eq:stability2}
\|u_1(t)-u_2(t)\|_{-1}^2 \leq e^{Ct} \|u_{1,0}-u_{2,0}\|_{-1}^2.
\end{align}

In particular, if $F'_2$ is Lipschitz, the pair $(u,w)$ is unique.
\end{theorem}

The stability estimate \eqref{eq:stability2} above follows immediately from \eqref{eq:uniqend}. 

\subsection{Extra regularity}

To conclude, we study additional regularity properties of the solutions.

\begin{theorem}\label{thm:regularity}
Denote by $(u,w)$ a solution pair of the Cahn-Hilliard system with a smooth potential given by Theorem \ref{thm:wellposedsmooth}. 
\begin{itemize}
\item[(i)] We have the regularity $u\in L^2_{H^2}$; \vskip 1mm
\item[(ii)] If $\Gamma_0$ is a $C^3$-surface, the diffeomorphisms $\Phi_t^0$ are $C^3$ and $F''(u(t))\in L^2(\Gamma(t))$ for a.a. $t\in [0,T]$, then $u\in L^2_{H^3}$; \vskip 1mm
\item[(iii)] If $u_0\in H^2(\Gamma_0)$ and \[|F''(r)| \leq C_1 |r|^{q-2} + C_2, \] then we have $u\in L^\infty_{H^2}$, $\md u\in L^2_{L^2}$ and $w\in L^\infty_{L^2}\cap L^2_{H^2}$. \vskip 1mm
\item[(iv)] If $\Gamma_0$ is a $C^4$-surface, the diffeomorphisms $\Phi_t^0$ are $C^4$, $\Delta_\Gamma F'(u(t))\in L^2(\Gamma(t))$ for a.a. $t\in [0,T]$, and $u_0\in H^2(\Gamma_0)$, then $u\in L^2_{H^4}$, and $u$ is a \textit{strong solution} of the system, i.e. it satisfies
\begin{align}\label{eq:strong_eq}
\partial^\bullet u + u\tgrad\cdot\mathbf{V} - \tgrad\cdot \big(u \left(\mathbf{V}_\mathbf{\tau}-\mathbf{V}_\mathbf{a}\right)\big) = \Delta_\Gamma\left(- \Delta_\Gamma u + F'(u) \right) \quad \text{ in } L^2(\Gamma(t)).
\end{align}
 for almost all $t\in [0,T]$.
\end{itemize}
\end{theorem}

\begin{proof}
The second equation gives, for almost all $t\in [0,T]$, 
\begin{align}
-\Delta_{\Gamma(t)} u(t) = w(t) - F'(u(t)) \in L^2(\Gamma(t)),
\end{align}
with the estimate 
\begin{align}
\|\Delta_{\Gamma(t)} u(t) \|_{L^2(\Gamma(t))} \leq \| w(t)\|_{L^2(\Gamma(t))} + \|F'(u(t))\|_{L^2(\Gamma(t))} \leq 2 \|w(t)\|_{L^2(\Gamma(t))},
\end{align}
and elliptic regularity (see e.g. \cite[Lemma 3.2]{DziEll13-a}) then implies $u(t)\in H^2(\Gamma(t))$, satisfying
\begin{align}\label{eq:elliptic}
\|u(t)\|_{H^2(\Gamma(t))} &\leq \|\Delta_{\Gamma(t)} u(t)\|_{L^2(\Gamma(t))} + C \|u(t)\|_{H^1(\Gamma(t))} \\
&\leq 2 \|w(t)\|_{L^2(\Gamma(t))} +  C \|u(t)\|_{H^1(\Gamma(t))},
\end{align}
where $C>0$ can be taken to be independent of $t$. Integrating the above over $[0,T]$ gives
$u\in L^2_{H^2}$, as desired. 

In the second case, we have instead $-\Delta_{\Gamma(t)} u(t)\in H^1(\Gamma(t))$, and we can use the extra regularity of the surfaces and the nonlinear term to show that $u\in L^2_{H^3}$.

Now suppose $u_0\in H^2(\Gamma_0)$. Integrating by parts we have
\begin{align}
a_S(P_M u_0, w^M(0)) = a_S (u_0, w^M(0)) = - m(\Delta_{\Gamma} u_0, w^M(0)),
\end{align}
and thus from the second equation of \eqref{eq:prob1ga} at $t=0$ we obtain 
\begin{align}
\begin{split}\label{eq:w0}
\|w^M(0)\|^2_{L^2} &= a_S (u^M(0), w^M(0)) + m(F'(u^M(0)), w^M(0)) \\ 
&= - m(\Delta_{\Gamma} u_0, w^M(0)) + m(F'(u^M(0)), w^M(0)) \\
&\leq \left( \|\Delta_\Gamma u_0\|_{L^2} + C_1 \, \|u^M(0)\|^q_{H^1} \right)  \, \|w^M(0)\|_{L^2} \\
&\leq \left( \|\Delta_\Gamma u_0\|_{L^2} + C_2 \, \|u_0\|^q_{H^1} \right)  \, \|w^M(0)\|_{L^2},
\end{split}
\end{align}
which implies that $\|w^M(0)\|^2_{L^2} $ is uniformly bounded. To obtain $\md u\in L^2_{L^2}$, we differentiate the second equation to obtain, for all $\eta\in L^2_{V_M}$ with $\md\eta\in L^2_{V_M}$,
\begin{align}
a_S(\md u^M, \eta) + b(u^M, \eta) + m(F''(u^M)\,\md u^M, \eta) + g(F'(u^M), \eta) = m(\md w^M, \eta) + g(w^M, \eta).
\end{align}
The terms involving $\md \eta$ vanish because $\md \eta$ is still an admissible test function and $(u^M, w^M)$ is the solution pair to \eqref{eq:prob1ga}. Testing the above with $\eta=w^M$ gives
\begin{align}
\begin{split}\label{eq:seconddiff}
a_S(\md u^M, w^M) + b(u^M, w^M) + m(F''(u^M)\,\md u^M, w^M) &+ g(F'(u^M), w^M) \\
&= \dfrac{1}{2}\dfrac{d}{dt}\|w^M\|^2_{L^2} + \dfrac{1}{2}g(w^M, w^M).
\end{split}
\end{align}
Now taking $\eta=\md u^M$ in the first equation of \eqref{eq:prob1ga} gives
\begin{align}\label{eq:firsteq}
\|\md u^M\|^2_{L^2} + g(u^M, \md u^M) + a_N (u^M, \md u^M) + a_S(w^M, \md u^M) = 0,
\end{align}
and combining \eqref{eq:seconddiff} with \eqref{eq:firsteq} we obtain
\begin{align}
\dfrac{1}{2} \dfrac{d}{dt} \|w^M\|^2_{L^2} + &\|\md u^M\|^2_{L^2} = - g(u^M, \md u^M) - a_N(u^M, \md u^M) \\
&- \dfrac{1}{2} g(w^M, w^M) + b(u^M, w^M) + m(F''(u^M)\,\md u^M, w^M)  + g(F'(u^M), w^M).
\end{align} 
Using the uniform bounds for $u^M$, the first four terms on the right are estimated as
\begin{align}
- g(u^M, \md u^M) - a_N(u^M, \md u^M) - \dfrac{1}{2} g(w^M, w^M) + &b(u^M, w^M) \\
&\leq \dfrac{1}{4} \|\md u^M\|^2_{L^2} + C_1 + C_2 \, \|w^M\|^2_{H^1},
\end{align}
and the last two terms we estimate using the uniform bounds for $u^M$ and the Sobolev inequality:
\begin{align}
m(F''(u^M)\,\md u^M, w^M) + g(F'(u^M), w^M) \leq \dfrac{1}{4}\|\md u^M\|^2_{L^2} &+ C_1 \,\|F''(u^M)^2\|_{L^2} \, \|(w^M)^2\|_{L^2} \\
&\hskip 4mm + C_2 \|F'(u^M)\|_{L^2} \|w^M\|_{L^2}\\
&\hskip -2cm \leq \dfrac{1}{4}\|\md u^M\|^2_{L^2} + C_3 \|w^M\|_{H^1}^{2}.
\end{align}
Putting these two estimates together we obtain
\begin{align}
\|\md u^M\|^2_{L^2} + \dfrac{d}{dt} \|w^M\|^2_{L^2} \leq C_1 + C_2 \, \|w^M\|^2_{H^1},
\end{align}
and integrating yields
\begin{align}
\sup_{[0,T]} \|w^M\|_{L^2}^2 + \int_0^T \|\md u^M\|^2_{L^2} \leq \|w^M(0)\|^2_{L^2} + C_1 \leq C_2,
\end{align}
using \eqref{eq:w0}. Letting $M\to\infty$ yields $w\in L^\infty_{L^2}$ and $\md u\in L^2_{L^2}$, as desired. But then we note that the first equation actually gives $-\Delta_\Gamma w(t) \in L^2(\Gamma(t))$, and using elliptic regularity as in the first part of this proof implies $w\in L^2_{H^2}$. The inequality in \eqref{eq:elliptic} gives $u\in L^\infty_{H^2}$. 

Finally, in the conditions of (iv) we can deduce $u\in L^2_{H^4}$ from the fact that $w\in L^2_{H^2}$, and it follows that $u$ satisfies, for almost all $t\in [0,T]$, the equation
\begin{align}
\partial^\bullet u + u\tgrad\cdot\mathbf{V} - \tgrad\cdot \big(u \left(\mathbf{V}_\mathbf{\tau}-\mathbf{V}_\mathbf{a}\right)\big) = \Delta_\Gamma\left(- \Delta_\Gamma u + F'(u) \right) \quad \text{ in } L^2(\Gamma(t)),
\end{align}
finishing the proof.
\end{proof}

\begin{remark}
Combining the results above with the Sobolev embedding theorems in Lemma \ref{lem:assumptionsgive}d) and classical Schauder theory, one should obtain, for a solution pair $(u,w)$ a.a. $t\in [0,T]$ and every $\alpha\in (0,1)$,
\begin{itemize}
\item[(i)]  $u(t)\in C^\alpha(\Gamma(t))$;
\item[(ii)] in the conditions of Theorem \ref{thm:regularity}(ii), $u(t)\in C^{1+\alpha}(\Gamma(t))$;
\item[(iii)] in the conditions of Theorem \ref{thm:regularity}(iii), $w(t)\in C^{\alpha}(\Gamma(t))$. If additionally $F'(u(t))\in C^\alpha(\Gamma(t))$ then $u(t)\in C^{2+\alpha}(\Gamma(t))$;
\item[(iv)] in the conditions of Theorem \ref{thm:regularity}(iv), $u(t)\in C^{2+\alpha}(\Gamma(t))$. If additionally $F'(u(t))\in C^{2+\alpha}(\Gamma(t))$ then $u(t)\in C^{4+\alpha}(\Gamma(t))$ and $u$ is a classical solution, i.e. \eqref{eq:strong_eq} holds for all $t\in [0,T]$.
\end{itemize}
We leave the study of classical solutions and extra regularity for future work.
\end{remark}

\section{Non smooth potentials}\label{sec:nonsmooth}

In this section, we study the same problem with a logarithmic and a double obstacle potentials. These are non smooth potentials, and in both cases an appropriate weak formulation needs to be considered. For the former, the derivative is singular at $\pm 1$, and so we need to impose the constraint $|u|<1$ for a solution. In the latter case, the energy is not differentiable, and the second equation must be reinterpreted as a variational inequality. We shall see that, in both cases, a general statement obtained as above is not possible to obtain, and the choice of admissible initial conditions is related to the evolution of the surfaces. 

\subsection{Logarithmic potential}\label{sec:log}

In this section we define, for $r\in [-1,1]$,
\begin{align}\label{eq:logpotential}
F_\theta(r) &= \dfrac{\theta}{2\theta_c} \left((1+r)\log(1+r) + (1-r)\log(1-r)\right) + \dfrac{1-r^2}{2} \\
&=: \dfrac{\theta}{2\theta_c}F^{\mathrm{log}}(r) + \dfrac{1-r^2}{2},
\end{align}
where $\theta$ represents the temperature and $\theta_c$ is the critical temperature. For $\theta>\theta_c$, the potential is strictly convex with a global minimum at $u=0$. For $\theta<\theta_c$, $F_{\theta}$ is a double well potential with two global minima. Let us take for simplicity $\theta_c=1$ and assume $\theta<1$. Since in this section $\theta$ will be a fixed constant, we will henceforth omit the dependence on $\theta$ of both the potential and the solution. 

To use the logarithmic nonlinearity it is clear that we are only interested in solutions taking values in the interval $(-1,1)$. As in the previous section, we expect to have a solution pair $(u,w)$ solving the equations 
\begin{align}\label{eq:prob2before}
\begin{split}
m_\ast(\partial^\bullet u, \eta) + g(u, \eta) + a_N(u, \eta) + a_S(w,\eta) &= 0 \\
a_S(u,\eta) + m(F'(u), \eta) - m(w, \eta)&=0,
\end{split}
\end{align} 
satisfying $|u|<1$ almost everywhere and an initial condition $u(0)=u_0$, for some $u_0\in H^1(\Gamma_0)$ with $|u_0|\leq 1$. These conditions on $u_0$ are enough to guarantee that the energy at the initial time makes sense. We claim that this is not possible without extra conditions on the initial data and the evolution of the surface. Indeed, define, for $t\in [0,T]$ and $\eta\in H^1(\Gamma_0)$,
\begin{align}\label{eq:condition}
m_{\eta}(t) := \dfrac{1}{|\Gamma(t)|} \left|\int_{\Gamma_0} \eta \right|.
\end{align}
We start by proving a simple result which shows that well-posedness of the problem is related to the size of this function $m_\eta$.

\begin{proposition}\label{prop:illposed1}
Let $\{\Gamma(t)\}_{t\in [0,T]}$ be an evolving surface in $\R^3$ and $u_0\in H^1(\Gamma_0)$ satisfy $|u_0|\leq 1$. Suppose that there exists a subset of $[0,T]$ with positive measure on which $m_{u_0}(t)\geq 1$. Then there cannot exist a pair $(u,w)$ satisfying \eqref{eq:prob2before} and $|u|<1$ almost everywhere.
\end{proposition}

\begin{proof}
If such a pair exists, let $\tilde{I}\subset [0,T]$ be a subset of full measure in which equations \eqref{eq:prob3} hold, $|u|<1$ and $m_{u_0}(t)\geq 1$. We observe that, as before, the integral of the solution is preserved, and as a consequence, for $t\in \tilde{I}$,
\begin{align}
\left|\int_{\Gamma(t)} u(t)\right| = \left|\int_{\Gamma_0} u_0\right| \geq |\Gamma(t)|,
\end{align}
but since $|u|<1$ we also have
\begin{align}
\left|\int_{\Gamma(t)} u(t)\right| \leq \int_{\Gamma(t)} |u(t)| < |\Gamma(t)|,
\end{align}
which is a contradiction.
\end{proof}

An immediate consequence of the above and the fact that $m_{u_0}$ is continuous is the following:

\begin{corollary}\label{cor:corillposed1}
If there exists $t\in [0,T]$ such that $m_{u_0}(t)>1$, then there cannot exist a pair $(u,w)$ satisfying \eqref{eq:prob2before} and $|u|<1$ almost everywhere.
\end{corollary}

The results above show that, if we are to expect existence of solutions, then it must be $m_{u_0}(t)<1$ for almost every $t\in [0,T]$. In this article, we will focus in the case where $m_{u_0}(t)<1$ for \textit{every} $t\in [0,T]$. In particular, note that
\begin{align}
\dfrac{1}{|\Gamma_0|} \left| \int_{\Gamma_0} u_0 \right| = m_{u_0}(0) < 1.
\end{align}
%

For future use, let us define the set $\mathcal{I}_0$ of \textit{admissible initial conditions} by
\begin{align}
\mathcal{I}_0 := \left\{ \eta\in H^1(\Gamma_0) \colon |\eta|\leq 1 \text{ a.e. on } \Gamma_0, \,\, \mathrm{E}^\mathrm{CH}[\eta]<\infty, \,\, \text{ and } \,\, m_{\eta}<1 \right\}.
\end{align}
The first conditions ensure that the energy is defined at the initial time, and we motivated the second condition on the previous paragraph. Note that, since $m_{u_0}$ is continuous, there exists $\alpha\in [0, 1)$ such that $0\leq m_{u_0}(t) \leq \alpha <1$, for all $t\in [0,T]$. 

We can now state the problem we aim to solve in this section. Denote also $$\varphi(r) :=\left(F^\mathrm{log}\right)'(r), \quad r\in (-1,1).$$

\begin{problem}
Given an initial condition $u_0\in \mathcal{I}_0$, the \textit{Cahn-Hilliard system with a logarithmic potential} is the following problem: find a pair $(u,w)$ satisfying:
\begin{itemize}
\item[(a)] $u\in L^\infty_{H^1}\cap H^1_{H^{-1}}$ and $w\in L^2_{H^1}$; \vskip 1mm
\item[(b)] for almost all $t\in [0, T]$, $|u(t)|<1$ almost everywhere in $\Gamma(t)$; \vskip 1mm
\item[(c)] for every $\eta\in L^2_{H^1}$, the equations
\begin{align}\label{eq:prob3}\tag{$\text{CH}_{\text{log}}$}
\begin{split}
m_\ast(\partial^\bullet u, \eta)+ g(u, \eta) + a_N(u, \eta) + a_S(w,\eta) &= 0, \\
a_S(u,\eta) + \dfrac{\theta}{2} m(\varphi(u), \eta) - m(u,\eta) - m(w, \eta)&=0,
\end{split}
\end{align} 
for almost every $t\in [0, T]$; \vskip 1mm
\item[(d)] $u(0)=u_0$ almost everywhere in $\Gamma_0$.
\end{itemize}
The pair $(u,w)$ is called a \textit{weak solution} of \eqref{eq:prob3}.
\end{problem}

\subsubsection{Approximating problem}

To prove well-posedness for the problem above we will proceed by approximation. Define, for $r\in \R$ and $\delta\in (0,1)$,
\begin{align}
F^{\mathrm{log}}_{\delta}(r) = 
\dfrac{\theta}{2}
\begin{cases}
(1-r)\log(\delta) + (1+r)\log(2-\delta) + \frac{(1-r)^2}{2\delta} + \frac{(1+r)^2}{2(2-\delta)} - 1, & r\geq 1-\delta \\
(1+r)\log (1+r) + (1-r)\log(1-r), &|r| \leq 1-\delta \\
(1+r)\log(\delta) + (1-r)\log(2-\delta) + \frac{(1+r)^2}{2\delta} + \frac{(1-r)^2}{2(2-\delta)} - 1, &  r\leq -1+\delta
\end{cases}.
\end{align}
Then $F_\delta^\mathrm{log}\in C^2(\R)$. For simplicity of notation let us also denote 
\begin{align}
\varphi_\delta(r) &:= \left(F^\mathrm{log}_\delta\right)'(r), \quad r\in\R.
\end{align}
Consider now the following problem.

\begin{problem}
Given $\delta\in (0,1)$ and an initial condition $u_0\in \mathcal{I}_0$, find $(u_\delta,w_\delta)\in L^\infty_{H^1}\times L^2_{H^1}$, with $\partial^\bullet u_\delta\in L^2_{H^{-1}}$, satisfying, for all $\eta\in L^2_{H^1}$ and a.e. $t\in [0,T]$,
\begin{align}\label{eq:prob3eps}\tag{$\text{CH}^\delta_{\text{log}}$}
\begin{split}
m_\ast(\partial^\bullet u_\delta, \eta) + g( u_\delta, \eta) + a_N(u_\delta,\eta) + a_S(w_\delta,\eta) &= 0, \\
a_S(u_\delta, \eta) + \dfrac{\theta}{2} m(\varphi_\delta(u_\delta), \eta)- m(u_\delta, \eta) - m(w_\delta, \eta) &=0,
\end{split}
\end{align} 
and $u_\delta(0)=u_0$ almost everywhere in $\Gamma_0$.
\end{problem}

Observe that $\varphi_\delta$ is Lipschitz, so that $\varphi_\delta(u_\delta(t))\in H^1(\Gamma(t))$ for almost all $t\in [0,T]$. It is also simple to check that the potential $F_\delta^\mathrm{log}$ satisfies all the assumptions (A1)-(A2) in the previous section, and hence it follows from Theorem \ref{thm:wellposedsmooth} that we have existence and uniqueness of a solution pair $(u_\delta, w_\delta)$ for \eqref{eq:prob3eps} (observe that $F_\delta^\mathrm{log}$ has quadratic growth at infinity, and the observations for the Galerkin approximation problem in Remark \ref{rem:extraassump} hold). Note that, as we have previously mentioned, the uniform bounds we obtained in the previous section depend on the form of the potential, and consequently on $\delta$; this means that in order to let $\delta\to 0$ in \eqref{eq:prob3eps} we need new estimates. Before doing so, we need an additional lemma. Consider the Galerkin approximation for \eqref{eq:prob3eps}:

\begin{problem}
Given $\delta\in (0,1)$ and an initial condition $u_0\in \mathcal{I}_0$ with $\mathcal{E}(u_0)<\infty$, find $(u^M_\delta,w_\delta^M)\in L^2_{V_M}$ with $\partial^\bullet u_\delta^M\in L^2_{V_M}$ satisfying, for all $\eta\in L^2_{V_M}$ and $t\in [0, T]$, 
\begin{align}\label{eq:gaprob3eps}\tag{$\text{CH}^{M,\delta}_{\text{log}}$}
\begin{split}
m_\ast(\md u_\delta^M, \eta) + g(u_\delta^M, \eta) + a_N(u_\delta^M, \eta) + a_S(w_\delta^M, \eta) &= 0, \\
a_S(u_\delta^M, \eta) + \dfrac{\theta}{2} m(\varphi_\delta(u_\delta^M), \eta) - m(u_\delta^M, \eta) - m(w_\delta^M, \eta) &= 0,
\end{split}
\end{align} 
and $u_\delta^M(0)=P_M u_0$ almost everywhere in $\Gamma_0$.
\end{problem}

Let us also introduce the following approximate Cahn-Hilliard energy:
\begin{align}
\mathrm{E}^\mathrm{CH}_\delta [u] = \int_{\Gamma(t)} \dfrac{1}{2} |\tgrad u|^2 + F_\delta(u) .
\end{align}
We have the following uniform bound for the initial energy.

\begin{lemma}\label{lem:apriori2}
Let $(u^M_\delta, w_\delta^M)$ be the unique solution of \eqref{eq:gaprob3eps}. There exists a constant $C>0$, independent of both $M$ and $\delta$, such that $\mathrm{E}^\mathrm{CH}_\delta[0; u_\delta^M(0)] \leq C$.
\end{lemma}

\begin{proof}
For each $M\in\N$, we have
\begin{align}
\mathrm{E}^{\mathrm{CH}}_\delta[P_M u_0] & \leq \|\tgrad P_M u_0\|\L^2 + m(F_\delta^\mathrm{log}(P_M u_0), 1) + \dfrac{1}{2} m(1-(P_M u_0)^2, 1)  \\
&\leq C_1 + C_2 \|u_0\|^2_{H^1} + m(F_\delta^\mathrm{log}(P_M u_0), 1).
\end{align}
Now convexity of $F_\delta^\mathrm{log}$ implies that, for any $r, s\in \R$,
\begin{align}
F^\mathrm{log}_\delta(r)\geq F_\delta^\mathrm{log}(s) + \varphi_\delta(s)(r-s),
\end{align}
and combining this with $F^\mathrm{log}_\delta(r)\leq F^\mathrm{log}(r)$, for $r\leq 1$, leads to
\begin{align}
\mathrm{E}^\mathrm{CH}_\delta[P_M u_0] &\leq C_1 + C_2 \|u_0\|_{H^1}^2 + m(\varphi_\delta(P_M u_0), P_M u_0-u_0) + m(F_\delta^\mathrm{log}(u_0), 1) \\ 
&\leq C_1\|u_0\|_{H^1}^2 + C_2 + m(\varphi_\delta(P_M u_0), P_M u_0-u_0) + \mathrm{E}^\mathrm{CH}[u_0] \\
&\leq C_3 + m(\varphi_\delta(P_Mu_0), P_M u_0-u_0),
\end{align}
with $C_3$ independent of both $M$ and $\delta$. Now note that $\varphi_\delta'=F_{1,\delta}''$ is bounded (for each fixed $\delta$), so that $\varphi_\delta$ is Lipschitz continuous and thus $\varphi_\delta(P_Mu_0) \to \varphi_\delta (u_0) \text{ in } L^2(\Gamma_0).$ Consequently $$m(\varphi_\delta(P_M u_0), P_M u_0-u_0) \to 0 \quad \text{ as } M\to\infty,$$ and so 
\begin{align}
\limsup_{M\to\infty} \mathrm{E}^\mathrm{CH}_\delta [P_M u_0] \leq C_3;
\end{align}
in particular, $\mathrm{E}^\mathrm{CH}_\delta [P_M u_0]$ is uniformly bounded. 
\end{proof}

In the next result, we obtain uniform estimates for the solution of \eqref{eq:prob3eps}.

\begin{proposition}[A priori estimates for \eqref{eq:prob3eps}]\label{prop:apriori22}
Let $(u_\delta, w_\delta)$ be the unique solution of \eqref{eq:prob3eps}.
\begin{itemize}
\item[(a)] There exists a constant $C>0$, independent of $\delta$, such that 
\begin{align}
\sup_{t\in [0, T]} \mathrm{E}^\mathrm{CH}_\delta[t; u_\delta(t)] + \int_0^T \|\nabla_{\Gamma(t)} w_\delta(t)\|^2_{L^2(\Gamma(t))} \, dt &\leq C;
\end{align}
\item[(b)] $u_\delta$ and $w_\delta$ are uniformly bounded in $L^\infty_{H^1}$ and $L^2_{H^1}$, respectively; \vskip 1mm 
\item[(c)] $\partial^\bullet u_\delta$ is uniformly bounded in $L^2_{H^{-1}}$; \vskip 1mm
\item[(d)] $\varphi_\delta(u_\delta)$ is uniformly bounded in $L^2_{L^2}$.
\end{itemize}
\end{proposition}

\begin{proof}
a) As in the proof of Proposition \ref{prop:propenergy}, we have for the solution $(u_\delta^M, w_\delta^M)$ of \eqref{eq:gaprob3eps}
\begin{align}
\dfrac{d}{dt}\mathrm{E}^\mathrm{CH}[u_\delta^M] + \dfrac{1}{2}\|\tgrad w_\delta^M\|\L^2 \leq - g(u_\delta^M, w_\delta^M) + C_0 + C_1 \mathrm{E}^\mathrm{CH}_\delta[u_\delta^M],
\end{align}
with constants $C_0, C_1$ independent of the assumptions on the potential. Integrating between $0$ and $t$ we obtain
\begin{align}
\mathrm{E}^\mathrm{CH}_\delta[u_\delta^M] + \int_0^t \|\tgrad w_\delta^M\|\L^2 
&\leq C_1 + C_2\int_0^t \mathrm{E}^\mathrm{CH}_\delta [u_\delta^M] - \int_0^t g(u_\delta^M, w_\delta^M),
\end{align}
where $C_1>0$ is the constant (independent of $\delta$) given by the previous result. Letting $M\to\infty$ and using the convergence results for the Galerkin approximation this yields 
\begin{align}\label{eq:newenergy}
\mathrm{E}^\mathrm{CH}_\delta[u_\delta] + \int_0^t\|\tgrad w_\delta\|_{L^2}^2  \leq C_1 + C_2\int_0^t \mathrm{E}^\mathrm{CH}_\delta[u_\delta] + \int_0^t |g(u_\delta, w_\delta)|.
\end{align}
Both $C_1$ and $C_2$ are independent of $M$ and $\delta$.

Our aim now is to obtain an estimate on the last term on the right hand side above. Testing the second equation of \eqref{eq:prob3eps} with $u_\delta \tgrad\cdot\mathbf{V}$, we get
\begin{align}
g(u_\delta, w_\delta) &= m(\tgrad u_\delta, \tgrad (u_\delta \tgrad\cdot \mathbf{V})) + m(\varphi_\delta(u_\delta),u_\delta \tgrad\cdot \mathbf{V}) - m(u_\delta, u_\delta \tgrad\cdot\mathbf{V}) \\
&\leq C_1 \|\tgrad u_\delta\|^2\L + C_2\, m(\varphi_\delta(u_\delta), u_\delta) + C_3 \\
&\leq C_1 \|\tgrad u_\delta\|^2\L + C_4\|\tgrad u_\delta\|^2\L + |m(w_\delta, u_\delta)| + C_3,
\end{align}
where $C_1$, $C_3$, $C_4$ are independent of $\delta$. Taking $\eta\equiv 1$ in the second equation of \eqref{eq:prob3eps} yields
\begin{align}
m(w_\delta, 1) = - m(u_\delta, 1) + m(\varphi_\delta(u_\delta), 1),
\end{align}
and using the fact that $|\varphi_\delta(r)|\leq r\varphi_\delta(r) + 1$ and the second equation again we obtain
\begin{align}
\begin{split}\label{eq:est1}
|m(w_\delta, 1)| &\leq |m(u_\delta, 1)| + m(\varphi_\delta(u_\delta), u_\delta) + C_1 \\ 
&\leq C_2 |(u_0)_{\Gamma_0}| + \|u_\delta\|^2\L - \|\tgrad u_\delta\|^2\L + m(w_\delta, u_\delta) + C_1 \\
&\leq C_2|(u_0)_{\Gamma_0}| + C_3\|\tgrad u_\delta\|\L^2 + |m(w_\delta, u_\delta)| + C_1.
\end{split}
\end{align}
Recalling the definition of the inverse Laplacian (see Appendix \ref{app:invlaplacian}), we now have
\begin{align}
\begin{split}\label{eq:est3}
|m(w_\delta, u_\delta)| &\leq |m(w_\delta,u_\delta - (u_\delta)_\Gamma)| + |(u_\delta)_\Gamma| \, |m(w_\delta, 1)| \\
&= |a_S(w_\delta, \mathcal{G}(u_\delta-(u_\delta)_\Gamma))| + m_{u_0}(t) |m(w_\delta, 1)| \\
&\leq |a_S(w_\delta, \mathcal{G}(u_\delta-(u_\delta)_\Gamma))| + \alpha |m(w_\delta, 1)|
\end{split}
\end{align}
and hence, using Young and Poincar\'{e} inequalities we obtain
\begin{align}
\begin{split}\label{eq:est2}
|m(w_\delta, u_\delta)| &\leq \dfrac{1-\alpha}{4}\|\tgrad w_\delta\|\L^2 + C_1 \|\tgrad \mathcal{G}(u_\delta-(u_\delta)_\Gamma) \|^2\L + \alpha \, |m(w_\delta, 1)| \\
&\leq \dfrac{1-\alpha}{4}\|\tgrad w_\delta\|\L^2 + C_1 \|u_\delta-(u_\delta)_\Gamma\|^2\L + \alpha \, |m(w_\delta, 1)| \\
&\leq  \dfrac{1-\alpha}{4}\|\tgrad w_\delta\|\L^2 + C_2 \|\tgrad u_\delta\|^2\L + \alpha \, |m(w_\delta, 1)|.
\end{split}
\end{align}
Combining \eqref{eq:est1} and \eqref{eq:est2} yields
\begin{align}
|m(w_\delta, u_\delta)| \leq \dfrac{1}{4} \|\tgrad w_\delta\|\L^2 + C_1 \|\tgrad u_\delta\|^2\L + C_2,
\end{align}
and thus
\begin{align}
|g(u_\delta, w_\delta)| \leq \dfrac{1}{4} \|\tgrad w_\delta\|^2\L + C_1 \|\tgrad u_\delta\|\L^2 + C_2.
\end{align}
Hence from \eqref{eq:newenergy} we get, for $C_1$, $C_2$ independent of $M$ and $\delta$,
\begin{align}\label{eq:new3}
\mathrm{E}^\mathrm{CH}_\delta[u_\delta] + \dfrac{1}{2}\int_0^t\|\tgrad w_\delta\|_{L^2}^2  \leq C_1 + C_2\int_0^t \mathrm{E}^\mathrm{CH}_\delta[u_\delta],
\end{align}
The conclusion now follows from Gronwall's inequality. \vskip 1mm

b) From the energy estimate in \eqref{eq:new3}, we obtain a uniform bound for $\tgrad u_\delta$ in $L^\infty_{L^2}$, and since the mean value of $u_\delta$ is preserved it follows that $u_\delta$ is uniformly bounded in $L^\infty_{H^1}$. We also have a bound for $\tgrad w_\delta$ in $L^2_{L^2}$. To obtain a bound on the $L^2_{L^2}$-norm of $w_\delta$, we start by noticing that, using the uniform estimates for $u_\delta$, we obtain from \eqref{eq:est1} 
\begin{align}\label{eq:6}
|m(w_\delta, 1)| \leq C + |m(w_\delta, u_\delta)|,
\end{align}
and now we estimate \eqref{eq:est3} without using Young's inequality on the first term to get
\begin{align}
\begin{split}\label{eq:7}
|m(w_\delta, u_\delta)| &\leq  \|\tgrad w_\delta\|\L \, \|\tgrad \mathcal{G}(u_\delta-(u_\delta)_\Gamma)\|\L + \alpha\,|m(w_\delta, 1)|\\ 
&\leq \|\tgrad w_\delta\|\L \, \|u_\delta-(u_\delta)_\Gamma\|\L + \alpha \,|m(w_\delta, 1)| \\
&\leq C \|\tgrad w_\delta\|\L + \alpha\, |m(w_\delta, 1)|.
\end{split}
\end{align}
Combining \eqref{eq:6} and \eqref{eq:7} yields
\begin{align}
|m(w_\delta, 1)| \leq C_1 + C_2 \|\tgrad w_\delta\|\L + \alpha\, |m(w_\delta, 1)|,
\end{align}
and dividing through by $1-\alpha$ we get $|m(w_\delta, 1)| \leq C_1 + C_2 \|\tgrad w_\delta\|\L.$ So the desired bound follows once again by applying Poincar\'{e}'s inequality. \vskip 1mm

c) This follows by duality, since, for any $\eta\in L^2_{H^1}$, the previous estimates yield
\begin{align}
\langle \partial^\bullet u_\delta, \eta \rangle_{L^2_{H^{-1}}\times L^2_{H^1}} &= \int_0^T - g(u_\delta, \eta) - a_N(u_\delta, \eta)- a_S(w_\delta, \eta)  \leq C \|\eta\|_{L^2_{H^1}},
\end{align}
giving the conclusion. \vskip 1mm

d) Test the second equation of \eqref{eq:prob3eps} with $\varphi_\delta(u_\delta)$ to obtain 
\begin{align}
\|\varphi_\delta(u_\delta)\|\L^2 &= m(w_\delta, \varphi_\delta(u_\delta)) - a_S(u_\delta, \varphi_\delta(u_\delta))  - m(u_\delta, \varphi_\delta(u_\delta)) \\
&= m(w_\delta, \varphi_\delta(u_\delta)) - m(\varphi_\delta'(u_\delta)\nabla_\Gamma u_\delta, \nabla_\Gamma u_\delta)  - m(u_\delta, \varphi_\delta(u_\delta)) \\
&\leq m(w_\delta, \varphi_\delta(u_\delta)) - m(u_\delta, \varphi_\delta(u_\delta)),
\end{align}
since $\varphi_\delta'\geq 0$, and a standard application of Young's inequality gives
\begin{align}
\|\varphi_\delta(u_\delta)\|_{L^2}^2 &\leq C_1 \|w_\delta\|^2_{L^2} + C_2.
\end{align}
The conclusion follows from the uniform bound for $w_\delta$ in $L^2_{L^2}$.
\end{proof}

One extra challenge that we did not face in the previous section was the fact that, for the logarithmic case, the potential is only defined on the interval $[-1,1]$ and its derivative is now singular at $\pm 1$. Hence, for the second equation in \eqref{eq:prob3} to make sense, we must prove that the solution $u$ takes values only on $(-1,1)$ (up to a set of measure zero). This requires the next a priori estimate for the approximation.

\begin{lemma}\label{lem:oninterval}
There exist constants $C_1, C_2>0$, independent of $\delta$, such that
\begin{align}\label{eq:oninterval}
\int_{\Gamma(t)} [-1-u_\delta(t)]_+  + \int_{\Gamma(t)} [u_\delta(t)-1]_+ \leq \dfrac{C_1}{|\log(\delta)|}+C_2\delta,
\end{align}
where we denote, for a real-valued function $f$, $[f]_+ = \max(0, f)$. 
\end{lemma}

\begin{proof}
We start by observing that the uniform bounds imply that there exists a constant $C>0$, independent of $\delta$, such that, for any $\delta>0$ and $t\in [0, T]$,
\begin{align}\label{eq:8}
\int_{\Gamma(t)} F^\mathrm{log}_\delta(u_\delta)  \leq C.
\end{align}

Now let $0<\delta<1$ and denote, for $t\in [0, T]$,
\begin{align}
A_\delta^1(t) = \{u_\delta(t) > 1-\delta\}, \quad A_\delta^2(t) = \{|u_\delta(t)| \le 1-\delta\}, \quad A_\delta^3(t) = \{u_\delta(t) < -1+\delta\}.
\end{align}
We estimate the integrals over each of these sets as follows. For the integral over $A_\delta^2$, observe that, for $|r|<1$, we have $(1+r)\log(1+r)\geq 0$ and $(1-r)\log(1-r)\geq -e^{-1}$, so that
\begin{align}
 \int_{A_\delta^2(t)} F^\mathrm{log}_\delta(u_\delta)  \geq  \int_{A_\delta^2(t)} (1-u_\delta(t))\log(1-u_\delta(t))  \geq -\dfrac{1}{e}|\Gamma(t)| \geq -\dfrac{C}{e},
\end{align}
where $C$ is a constant independent of both $\delta$ and $t$. For the other integral terms, we note that $\log(\delta)<0$, and thus
\begin{align}
\int_{A^1_\delta(t)} F_{1,\delta}(u_\delta(t)) &\geq \int_{A^1_\delta(t)} (1-u_\delta(t))\log(\delta) + (1+u_\delta(t))\log(2-\delta) \, \d\Gamma - C \\
&\geq \log(\delta)\int_{A^1_\delta(t)} 1-u^\delta(t) \, \d\Gamma - C \\
&\geq \log(\delta)\int_{\{u_\delta(t) > 1\}} 1-u_\delta(t) \, \d\Gamma + \delta\log(\delta)|\Gamma(t)| - C \\
&\geq -\log(\delta)\int_{\Gamma(t)} [u_\delta(t)-1]_+ \, \d\Gamma + C\delta\log(\delta) - C,
\end{align}
Similarly,
\begin{align}
\int_{A^3_\delta(t)} F_{1,\delta}(u_\delta(t)) \, \d\Gamma &\geq \int_{A^3_\delta(t)} (1+u^\delta(t))\log(\delta) + (1-u_\delta(t))\log(2-\delta)\, \d\Gamma - C \\
&\geq \log(\delta)\int_{A^3_\delta(t)} 1+u^\delta(t) \, \d\Gamma - C \\
&\geq \log(\delta)\int_{\{u_\delta(t) < -1\}} 1+u_\delta(t) \, \d\Gamma + \delta\log(\delta)|\Gamma(t)| - C \\
&\geq -\log(\delta)\int_{\Gamma(t)} [-1-u_\delta(t)]_+ \, \d\Gamma + C\delta\log(\delta) - C.
\end{align}
Therefore
\begin{align}
|\log(\delta)|\int_{\Gamma(t)} [-1-u_\delta(t)]_+ \, \d\Gamma + |\log(\delta)| \int_{\Gamma(t)} [u_\delta(t)-1]_+ \, \d\Gamma \leq C_1+C_2\delta|\log(\delta)| 
\end{align}
and dividing through by $|\log(\delta)|$ we obtain
\begin{align}
\int_{\Gamma(t)} [-1-u_\delta(t)]_+ \, \d\Gamma + \int_{\Gamma(t)} [u_\delta(t)-1]_+ \, \d\Gamma \leq \dfrac{C_1}{|\log(\delta)|}+C_2\delta.
\end{align}
\end{proof}

\subsubsection{Passage to the limit}

As a consequence of the bounds established above, there exist $u\in L^2_{H^1}$, with $\partial^\bullet u\in L^2_{H^{-1}}$, and $w\in L^2_{H^1}$, such that, as $\delta\to 0$ (up to a subsequence)
\begin{align}\label{eq:convergence}
u_\delta \rightharpoonup u \, \text{ (weakly) in } L^2_{H^1}, \quad \partial^\bullet u_\delta \rightharpoonup \partial^\bullet u \, \text{ (weakly) in } L^2_{H^{-1}}, \quad w_\delta \rightharpoonup w \, \text{ (weakly) in } L^2_{H^1}.
\end{align}
By the Aubin-Lions compactness result in Proposition \ref{prop:inclusions}(c) we additionally obtain $u_\delta\to u$ strongly in $L^2_{L^2}$, and so also up to a subsequence (non-relabeled) we have, for almost every $t\in [0, T]$, $u_\delta(t)\to u(t)$ almost everywhere in $\Gamma(t)$. 
Letting $\delta\to 0$ in \ref{eq:oninterval} immediately yields

\begin{corollary}
For almost every $t\in [0, T]$, $|u(t)| \leq 1$ almost everywhere on $\Gamma(t)$.
\end{corollary}
%

We are not yet done since we still need to prove that the measure of the set where the solution is either $1$ or $-1$ is zero. First, we show the following auxiliary result.

\begin{lemma}\label{lem:aeconv}
Up to a subsequence, we have, for almost every $t\in [0, T]$,
\begin{align}
\lim_{\delta\searrow 0} \varphi_\delta(u_\delta(t)) = 
\begin{cases}
\varphi(u(t)), & \text{if } |u(t)|<1 \text{ a.e. in } \Gamma(t) \\
\infty, & \text{otherwise }
\end{cases}
\end{align}
almost everywhere in $\Gamma(t)$.
\end{lemma}

\begin{proof}
Let us fix $t\in [0,T]$ for which $u_\delta(t)\to u(t)$ a.e. in $\Gamma(t)$. Such a set has full measure in $[0,T]$. 

Suppose first that $|u(t)|<1$ a.e. in $\Gamma(t)$. Passing to a subsequence, we can further assume that $|u_\delta(t)|<1$ a.e. in $\Gamma(t)$. Let $x\in\Gamma(t)$ belong to the full measure subset of $\Gamma(t)$ for which $u_\delta(t,x)\to u(t,x)$, and let $\tau>0$ be arbitrary. Now choose:
\begin{itemize}
\item[(i)] $\delta_1>0$ such that, for $\delta<\delta_1$, $\varphi_\delta(u_\delta(t,x))=\varphi(u_\delta(t,x))$; \vskip 1mm
\item[(ii)] using continuity of $\varphi$ at $u(x,t)$, $\delta_2>0$ such that, if $y\in\Gamma(t)$ is such that $|y-u(x,t)|<\delta_2$, then $|\varphi(y)-\varphi(u(x,t))|<\tau$; \vskip 1mm
\item[(iii)] $\delta_3>0$ such that, for $\delta<\delta_3$, $|u_\delta(x,t)-u(x,t)|<\delta_2$;
\end{itemize}
Therefore, for $\delta<\min\{\delta_1, \delta_2, \delta_3\}$, we have
\begin{align}
|\varphi_\delta(u_\delta(t,x)) - \varphi(u(t,x))| &= |\varphi_\delta(u_\delta(t,x)) - \varphi_\delta(u(t,x))| < \tau,
\end{align}
proving the result for those points where $|u|<1$.

Now let $t\in [0,T]$ and $x\in\Gamma(t)$ be such that $|u(t,x)|=1$. If $u(t,x)=1$, then
\begin{align}
\varphi_\delta(u_\delta(t,x)) \geq \min\{\varphi(u_\delta(t,x)), \varphi(1-\delta)\} \to +\infty, \quad \text{ as } \delta\to 0,
\end{align}
and if $u(t,x)=-1$, then 
\begin{align}
\varphi_\delta(u_\delta(t,x)) \leq \max\{\varphi(u_\delta(t,x)),\varphi(-1+\delta)\} \to -\infty, \quad \text{ as } \delta\to 0,
\end{align}
concluding the proof.
\end{proof}

Combining this with the a priori bounds allows us to conclude that indeed the solution cannot take the values $\pm 1$ on a set of positive measure.

\begin{lemma}\label{lem:zero}
For almost every $t\in [0,T]$, the set $\{x\in\Gamma(t)\colon |u(t,x)| = 1\}$ has measure zero in $\Gamma(t)$.
\end{lemma}

\begin{proof}
We have, up to a subsequence,
\begin{align}\label{eq:limit}
\lim_{\delta\searrow 0} \varphi_\delta(u_\delta(t))u_\delta(t) = 
\begin{cases}
\varphi(u(t))u(t), & \text{if } |u(t)|<1 \text{ a.e. in } \Gamma(t) \\
+\infty, & \text{otherwise }
\end{cases}.
\end{align}
Now observe that, testing the second equation with $u_\delta$, we obtain
\begin{align}
\int_{\Gamma(t)} \varphi_\delta(u_\delta) u_\delta  \leq C_1 \|\tgrad u_\delta\|^2\L + C_2 \| w_\delta\|^2\L \leq C_3 + C_2 \|w_\delta\|^2_{L^2},
\end{align}
and therefore 
\begin{align}
\int_0^T \int_{\Gamma(t)} \varphi_\delta(u_\delta) u_\delta  \leq C_1 T + C_2 \int_0^T \|w_\delta\|^2\L \leq C_3.
\end{align}
Since $\varphi_\delta(r)r\geq 0$, Fatou's lemma implies that
\begin{align}
\int_0^T \int_{\Gamma(t)} \liminf_{\delta\to 0} \varphi_\delta(u_\delta)u_\delta \leq \liminf_{\delta\to 0} \int_0^T \int_{\Gamma(t)} \varphi_\delta(u_\delta(t))u_\delta \leq C.
\end{align}
But the limit of $\varphi_\delta(u_\delta)u_\delta$ is given by \eqref{eq:limit}, and thus it follows that, for almost every $t\in [0,T]$, it must be $|u(t)|<1$ almost everywhere in $\Gamma(t)$, proving the result.
\end{proof}

\begin{remark}
The result above then implies that  $u$ remains bounded in the interval $(-1,1)$ provided $u_0$ is given also in that interval. This is not the case for the smooth potentials considered in Section \ref{sec:smooth}, see \cite[Remark 4.10]{Mir19}, which means that the problem \eqref{eq:prob1} does not produce physically meaningful solutions. It is nonetheless important to analyse it, as the quartic potentials have been seen to be good approximations for \eqref{eq:logpotential} when a shallow quench is considered (i.e. when the temperature $\theta$ is close to the critical temperature for the system).
\end{remark}

We now have everything we need to conclude the proof of well-posedness. As in the previous section, taking the limit on the linear terms of the system is a simple consequence of the convergence results \eqref{eq:convergence}. As for the nonlinear term, the uniform bound for $\varphi_\delta(u_\delta)$ in $L^2_{L^2}$ combined with $\varphi_\delta(u_\delta)\to \varphi(u)$ a.e. gives, by Theorem \ref{thm:evolv_dct}, $\varphi_\delta(u_\delta) \rightharpoonup \varphi(u)$ in $L^2_{L^2}$. All in all:


\begin{theorem}
The pair $(u,w)$ given by \eqref{eq:convergence} is the unique weak solution of the Cahn-Hilliard system with a logarithmic potential. 
\end{theorem}

\begin{proof}
The fact that $(u,w)$ satisfies the equations and fulfils the initial condition follows exactly as in the proof of Lemmas \ref{prop:existence}, together with what we observed in the previous paragraph. The proof of uniqueness is also the same as in Proposition \ref{prop:smoothuniqueness}, by noting that also in this case $\varphi$ is monotone and $F_2'$ is Lipschitz.
\end{proof}

In summary, in this section we have proved:

\begin{theorem}\label{thm:wellposedlog}
Let $u_0\in \mathcal{I}_0$ and $F\colon [-1,1]\to\R$ be the logarithmic potential \eqref{eq:logpotential}. Then, there exists a unique pair $(u, w)$ with
\begin{align}
u\in L^\infty_{H^1}\cap H^1_{H^{-1}} \quad \text{ and } \quad w\in L^2_{H^1}
\end{align}
such that, for a.a. $t\in [0, T]$, $|u(t)|<1$ a.e. in $\Gamma(t)$, and satisfying, for all $\eta\in L^2_{H^1}$ and a.a. $t\in [0, T]$, 
\begin{align}
\begin{split}
m_\ast(\partial^\bullet u, \eta) + g(u, \eta) + a_N(u,\eta) + a_S(w, \eta) &= 0,\\
a_S(u,\eta) + \dfrac{\theta}{2} m(\varphi(u), \eta) - m(u,\eta) - m(w, \eta)&=0,
\end{split}
\end{align} 
and $u(0)=u_0$ almost everywhere in $\Gamma_0$. The solution $u$ satisfies the additional regularity
\begin{align}
u\in C^0_{L^2} \cap L^\infty_{L^p}, \quad \text{ for all } p\in [1,+\infty).
\end{align}

Furthermore, if $u_0, v_0\in \mathcal{I}_0$ satisfy $(u_0)_{\Gamma_0}=(v_0)_{\Gamma_0}$, and $u, v$ are the solutions of the system with $u(0)=u_0$, $v(0)=v_0$, then there exists a constant $C>0$ independent of $t$ such that, for almost all $t\in [0,T]$,
\begin{align}
\|u(t)-v(t)\|_{-1} \leq e^{Ct} \|u_0-v_0\|_{{-1}}.
\end{align}
\end{theorem}

\subsubsection{Extra regularity}

Exactly as in the previous section, we can immediately obtain an extra degree of regularity for $u$:

\begin{theorem}
For the solution $(u,w)$ given by the previous theorem, we have $u\in L^2_{H^2}$.
\end{theorem}

\begin{proof}
This follows as in Theorem \ref{thm:regularity} by noting that $u$ solves the problem
\begin{align*}
-\Delta_\Gamma u = w - F'(u) \in L^2_{L^2},
\end{align*}
and thus elliptic regularity theory applies. 
\end{proof}

As in the previous section, Lemma \ref{lem:assumptionsgive}d) implies that, for almost all $t\in [0,T]$, $u(t)\in C^\alpha(\Gamma(t))$ for every $\alpha\in (0,1)$. 

In this text we will not focus in obtaining higher regularity for the solution of the problem with a logarithmic potential, but it is interesting to analyse what challenges would arise in trying to do so. To obtain $L^2_{H^3}$-regularity (or higher) as we did in Theorem \ref{thm:regularity}, we have to prove that $F'(u)\in L^2_{H^1}$, which requires some integrability for $F''(u)$. Calculating
\begin{align*}
F''(r) = C \left( \dfrac{1}{1-r} - \dfrac{1}{1+r} + 1\right), \quad r\in (-1,1)
\end{align*}
shows that $F''(u)$ is integrable only if we can establish a uniform estimate
\begin{align}\label{eq:separation_pure_phases}
\|u\|_{L^\infty(\Gamma(t))} \leq 1-\xi,
\end{align}
for some $\xi>0$. The bound in \eqref{eq:separation_pure_phases} can be interpreted as a phenomenon of \textit{separation from the pure phases}; not only are the pure phases $\pm 1$ never reached, but there always remains at least a fixed amount of the other component. For the classical Cahn-Hilliard equation on a bounded domain $\Omega\subset \R^n$, it is well known that the solution separates from the pure phases in dimensions $n=1,2$, and the problem is still open for $n\ge 3$ (see \cite[Chapter 4]{Mir19} and references therein). Since we are working on $2$-dimensional manifolds, we expect \eqref{eq:separation_pure_phases} to be true, and consequently for $u$ to be more regular. This is left for future work, and we now turn to the last case of a double obstacle potential.

\subsection{Double obstacle potential}\label{sec:obs}
In this section, the homogeneous free energy is defined, for $r\in \R$, as
\begin{align}
F_o(r) = 
\begin{cases}
\frac{1}{2}(1-r^2), &\text{ if } |r|\leq 1 \\
+\infty, &\text{ if } |r|>1
\end{cases} 
\,\, =: \,\,
I(r) + \dfrac{1-r^2}{2}, 
\end{align}
where
\begin{align}
I(r) = 
\begin{cases}
0 &\text{ if } r\in [-1,1] \\
+\infty &\text{ otherwise}
\end{cases}.
\end{align}
This suggests that we will be interested in solutions lying in the set 
\begin{align}
K := \{\eta\in L^2_{H^1} \colon |\eta(t)| \leq 1 \text{ a.e. on } \Gamma(t) \text{ for a.e. } t\in [0, T]\}.
\end{align}
The energy functional for this problem is the same as in the previous sections, but since it is now non-differentiable the chemical potential $w$ is defined by $w+\Delta_\Gamma u + u \in \partial I(u)$, where $\partial I(\cdot)$ denotes the subdifferential of $I$ defined above. This means that, in the double obstacle case, we will be looking for a pair $(u,w)$, with $u\in K$, satisfying the system
\begin{align}\label{eq:prob3before}
\begin{split}
m_\ast(\partial^\bullet u, \eta) + g(u, \eta) + a_N(u,\eta) + a_S(w, \eta)= 0, \quad &\text{ for all } \eta\in L^2_{H^1}, \\
a_S(u, \eta- u) - m(u, \eta-u) \geq m(w, \eta-u), \quad &\text{ for all }\eta\in K,
\end{split}
\end{align}
and an initial condition $u(0)=u_0$ for some $u_0\in H^1(\Gamma_0)$ with $|u_0|\leq 1$ (this ensures the energy is defined at the initial time). Again, we show that some more assumptions on the initial condition and the surfaces need to be made. Recall the definition in \eqref{eq:condition}.

\begin{proposition}\label{prop:illposed2}
Let $\{\Gamma(t)\}_{t\in [0,T]}$ be an evolving surface in $\R^3$ and $u_0\in H^1(\Gamma_0)$ be such that $|u_0|\leq 1$.
\begin{itemize}
\item[(a)] If $m_{u_0}(t)>1$ for $t$ in a subset of positive measure of $[0,T]$, then there is no pair $(u,w)$ satisfying $u\in K$ and the system \eqref{eq:prob3before};
\item[(b)] If there exists $t\in [0,T]$ such that $m_{u_0}(t)>1$, then there is no pair $(u,w)$ satisfying $u\in K$ and the system \eqref{eq:prob3before};
\item[(c)] If $m_{u_0}(t)=1$ on a subinterval $I\subset (0,T]$ and $m_{u_0}(0)<1$, then there is no pair $(u,w)$ satisfying $u\in K$ and the system \eqref{eq:prob3before}.
\end{itemize}
\end{proposition}

\begin{proof}
For case (a), choosing $t\in [0,T]$ for which $m_{u_0}(t)>1$ and $u(t)\leq 1$, we would have, for almost all $t$,
\begin{align}
|\Gamma(t)| < \left| \int_{\Gamma_0} u_0 \right| = \left|\int_{\Gamma(t)} u(t)\right| \leq |\Gamma(t)|,
\end{align}
which is a contradiction. Part (b) follows immediately from (a) and the fact that $m_{u_0}$ is continuous.

Under the conditions in case (c), there would be a subset of $[0,T]$ with positive measure in which $m(t)=1$ and the equations \eqref{eq:prob2} are satisfied. Note that, for such $t$, $m(t)=1$ implies that $u(t)=1$, and the first equation reads as
\begin{align}
\int_{\Gamma(t)} (\tgrad\cdot\mathbf{V}) \eta + \int_{\Gamma(t)} \tgrad w\cdot\tgrad \eta = 0,
\end{align} 
for any $\eta\in L^2_{H^1}$. Testing with $\eta\equiv 1$, we obtain 
\begin{align}
\int_{\Gamma(t)} \tgrad\cdot \mathbf{V} = 0
\end{align}
for almost every $t\in I$. Since the function on the left hand side is continuous, actually equality must hold everywhere in $I$, and then in $I$ we have
\begin{align}
\dfrac{d}{dt} |\Gamma(t)| = \int_{\Gamma(t)} \tgrad\cdot \mathbf{V} = 0,
\end{align}
and hence $|\Gamma(t)| = |\Gamma_0|$ for $t\in I$. But this is a contradiction, since for almost all $t\in I$
\begin{align}
|\Gamma(t)| = \int_{\Gamma(t)} 1 = \int_{\Gamma(t)} u(t) = \int_{\Gamma_0} u_0 < |\Gamma_0|.
\end{align} 
\end{proof}

For the same reasons as before, we will restrict our analysis in this section to initial conditions given in the same set 
\begin{align}
\mathcal{I}_0 := \left\{ \varphi\in H^1(\Gamma_0) \colon |\varphi|\leq 1 \text{ a.e. on } \Gamma_0,\,\, \text{ and } \,\, m_{u_0}<1 \right\}.
\end{align} 
Note that functions in $\mathcal{I}_0$ automatically have finite energy. Our method of proof follows the same lines as the logarithmic problem.

\begin{problem}
Given $u_0\in \mathcal{I}_0$, the \textit{double obstacle Cahn-Hilliard system} is the following problem: find a pair $(u, w)$ satisfying:
\begin{itemize}
\item[(a)] $u\in H^1_{H^{-1}}\cap L^\infty_{H^1}\cap K$ and $w\in L^2_{H^1}$;
\item[(b)] for almost all $t\in [0, T]$, the system
\begin{align}\label{eq:prob2}\tag{$\text{CH}_{\text{o}}$}
\begin{split}
m_\ast( \partial^\bullet u, \eta)+ g(u, \eta) + a_N(u,\eta) + a_S(w, \eta)= 0, \quad &\text{ for all } \eta\in L^2_{H^1}, \\
a_S(u, \eta- u) - m(u, \eta-u) \geq m(w, \eta-u), \quad &\text{ for all }\eta\in K,
\end{split}
\end{align}
\item[(c)] the initial condition $u(0)=u_0$ almost everywhere in $\Gamma_0$.
\end{itemize}
The pair $(u,w)$ is called a \textit{weak solution} of \eqref{eq:prob2}. 
\end{problem}

\subsubsection{Approximating problem}

For $r\in \R$ and $\delta\in (0,1)$, define $$F_\delta(r) = F^\mathrm{obs}_\delta(r) + \dfrac{1-r^2}{2},$$ where
\begin{align}
F^\mathrm{obs}_\delta(r) = 
\begin{cases}
\frac{1}{2\delta}\left(r- (1+\frac{\delta}{2})\right)^2 + \frac{\delta}{24}, & \text{for } r\geq 1+\delta\\
\frac{1}{6\delta^2}(r-1)^3 , & \text{for } 1<r<1+\delta\\
0, & \text{for } |r|\leq 1 \\
-\frac{1}{6\delta^2}(r+1)^3, & \text{for } -1-\delta < r < -1\\
\frac{1}{2\delta}\left(r+ (1+\frac{\delta}{2})\right)^2 + \frac{\delta}{24}, & \text{for } r\leq -1-\delta
\end{cases}.
\end{align}
We will use this function to approximate the double obstacle potential defined above. We focus first on the following problem. Let us reuse the notation of the previous section and write $$\varphi_\delta(r):=\left(F^\mathrm{obs}_\delta\right)'(r), \quad r\in \R.$$

\begin{problem}
Given $u_0\in \mathcal{I}_0$, find a pair $(u_\delta, w_\delta)\in L^\infty_{H^1}\times L^2_{H^1}$, with $\partial^\bullet u_\delta\in L^2_{H^{-1}}$, satisfying, for all $\eta\in L^2_{H^1}$ and almost every $t\in [0,T]$,
\begin{align}\label{eq:prob2eps}\tag{$\text{CH}^\delta_{\text{o}}$}
\begin{split}
m_\ast(\partial^\bullet u_\delta, \eta) + g(u_\delta, \eta) + a_N(u,\eta) + a_S(w_\delta, \eta) &= 0,\\
a_S(u_\delta, \eta) + m(\varphi_\delta(u_\delta), \eta) - m(u_\delta, \eta) - m(w_\delta, \eta) &= 0, 
\end{split}
\end{align}
and the initial condition $u_\delta(0)=u_0$ almost everywhere in $\Gamma_0$. 
\end{problem}

\begin{remark}
Note that in this case we automatically obtain that the energy of the initial condition $u_0\in H^1(\Gamma_0)$ is bounded.
\end{remark}

This problem fits the framework of the general smooth potential we studied in the first section. In fact, it is easy to check that, for sufficiently small $\delta>0$, $F_\delta^\mathrm{obs}$ satisfies the assumptions (A1)-(A2) in Section 4, and thus Theorem \ref{thm:wellposedsmooth} implies that the problem \eqref{eq:prob2eps} is well-posed. We again observe that Remark \ref{rem:extraassump} holds for the Galerkin approximation of the problem above. As before, the a priori bounds are not independent of $\delta$, and thus we must obtain obtain new estimates before passing to the limit.

The Galerkin approximation for \eqref{eq:prob2eps} is the following problem.

\begin{problem}
Given $\delta>0$ and $u_0\in \mathcal{I}_0$, find $(u^M_\delta,w^M_\delta)\in L^2_{V_M}$ with $\partial^\bullet u_\delta^M\in L^2_{V_M}$ satisfying, for all $\eta\in L^2_{V_M}$ and $t\in [0, T]$, 
\begin{align}\label{eq:gaprob2eps}\tag{$\text{CH}^{M,\delta}_{\text{o}}$}
\begin{split}
m_\ast(\partial^\bullet u^M_\delta, \eta)+g(u^M_\delta, \eta) + a_N(u_\delta^M,\delta) + a_S(w^M_\delta, \eta) &= 0,\\
a_S(u^M_\delta, \eta) + m(\varphi_\delta(u^M_\delta), \eta) - m(u^M_\delta, \eta) - m(w^M_\delta, \eta)&=0,
\end{split}
\end{align} 
and $u_\delta^M(0)=P_M u_0$ almost everywhere in $\Gamma_0$.
\end{problem}

The a priori estimates for the double obstacle potential are analogous to those in the previous section. As in Lemma \ref{lem:apriori2}, we start by obtaining a uniform bound for the initial energy. We reuse the notation
\begin{align}
\mathrm{E}^\mathrm{CH}_\delta[u] = \int_{\Gamma(t)} \dfrac{1}{2} |\tgrad u|^2 + F_\delta(u) .
\end{align}

\begin{lemma}[A priori estimates for \eqref{eq:gaprob3eps}]
Let $(u^M_\delta, w_\delta^M)$ be the unique solution of \eqref{eq:gaprob2eps}. There exists a constant $C>0$, independent of both $M$ and $\delta$, such that $\mathrm{E}^\mathrm{CH}_\delta[0; u_\delta^M(0)] \leq C$.
\end{lemma}

\begin{proof}
The proof follows as in Lemma \ref{lem:apriori2}. Observe that in this case $F_\delta^\mathrm{obs}(u_0) = F^\mathrm{obs}(u_0) = 0$.
\end{proof}

We now obtain the a priori estimates for \eqref{eq:prob2eps}.

\begin{proposition}[A priori estimates for \eqref{eq:prob2eps}]\label{prop:apriori3}
Let $(u_\delta, w_\delta)$ be the unique solution of \eqref{eq:prob2eps}.
\begin{itemize}
\item[(a)] There exists a constant $C>0$, independent of $\delta$, such that 
\begin{align}
\sup_{t\in [0, T]} \mathrm{E}^\mathrm{CH}_\delta[t; u_\delta(t)] + \int_0^T \|\nabla_{\Gamma(t)} w_\delta(t)\|^2_{L^2(\Gamma(t))}  &\leq C;
\end{align}
\item[(b)] $u_\delta$ and $w_\delta$ are uniformly bounded in $L^\infty_{H^1}$ and $L^2_{H^1}$, respectively; \vskip 1mm
\item[(c)] $\partial^\bullet u_\delta$ is uniformly bounded in $L^2_{H^{-1}}$.
\end{itemize}
\end{proposition}

\begin{proof}
The proof of this result is exactly the same as that of Proposition \ref{prop:apriori22}. Observe that, also for the approximation of the double obstacle potential, we have for any $r\in\R$
\begin{align}
F_{\delta}^\mathrm{obs}(r)\geq 0, \quad r\varphi_\delta(r)\geq 0, \quad |\varphi'_\delta(r)|\leq r\varphi_\delta'(r)+1,
\end{align}
and hence the whole proof carries over unchanged to this case.
\end{proof}

\subsubsection{Passage to the limit}

As before, we have $u\in L^2_{H^1}$ with $\partial^\bullet u\in L^2_{H^{-1}}$ and $w\in L^2_{H^1}$, such that, as $\delta\to 0$ (up to a subsequence)
\begin{align}\label{eq:conv3}
u_\delta \rightharpoonup u \, \text{ weakly in } L^2_{H^1}, \quad \partial^\bullet u_\delta \rightharpoonup \partial^\bullet u \, \text{ weakly in } L^2_{H^{-1}}, \quad w_\delta \rightharpoonup w \, \text{ weakly in } L^2_{H^1}.
\end{align}
We can use the Aubin-Lions Lemma in Proposition \ref{prop:inclusions}(c) to further obtain $u_\delta\to u$ strongly in $L^2_{L^2}$, and so also up to a subsequence (non-relabeled) we have, for almost every $t\in [0, T]$, $u_\delta(t)\to u(t)$ almost everywhere in $\Gamma(t)$. 

Before passing to the limit in the problem we need to ensure $u\in K$. First, let us introduce some notation: for $\delta>0$, define
\begin{align}\label{eq:betadefinition}
\beta_\delta(r) := \delta \varphi_\delta(r), \,\, r\in \R,
\quad
\text{and}
\quad
\beta(r) := \lim_{\delta\to 0} \beta_\delta(r) = 
\begin{cases}
r-1, &\text{ for } r\geq 1 \\
0, &\text{ for } |r|\leq 1 \\
r+1, &\text{ for } r\leq -1 
\end{cases}
\end{align}
Observe that $\beta$ is Lipschitz continuous, and that we have, for any $r\in\R$,
\begin{align}\label{eq:beta}
|\beta(r)-\beta_\delta(r)| \leq \dfrac{\delta}{2}, \quad\quad 0\leq \beta'_\delta(r) \leq 1.
\end{align}

\begin{lemma}
With the notation above, $u\in K$. 
\end{lemma}

\begin{proof}
Test the second equation of \eqref{eq:prob2eps} with $\eta(t) = \beta_\delta(u_\delta(t))\in H^1(\Gamma(t))$ to obtain, for a.e. $t\in [0, T]$,
\begin{align}
a_S\left(u_\delta, \beta_\delta(u_\delta)\right) + \dfrac{1}{\delta}\|\beta_\delta(u_\delta)\|\L^2 \leq C \left(\|u_\delta\|\L^2 + \|w_\delta\|\L^2\right) + \dfrac{1}{2\delta} \|\beta_\delta(u_\delta)\|\L^2,
\end{align}
and hence from the uniform bounds we obtain $C_1, C_2>0$ such that
\begin{align}
a_S\left(u_\delta, \beta_\delta(u_\delta)\right)&+ \dfrac{1}{2\delta}\|\beta_\delta(u_\delta)\|\L^2 \leq C_1 + C_2\|w_\delta\|\L^2.
\end{align}
But note that, since $\beta_\delta'\geq 0$, $a_S\left(u_\delta, \beta_\delta(u_\delta)\right) \geq 0,$
and therefore we obtain 
\begin{align}
\dfrac{1}{2\delta}\|\beta_\delta(u_\delta)\|^2\L \leq C_1 + C_2\|w_\delta\|\L^2,
\end{align}
from where
\begin{align}
\int_0^T \|\beta_\delta(u_\delta)\|\L^2  \leq C_1\delta + C_2\delta\int_0^T\|w_\delta\|\L^2 \, dt \leq C_3\delta,
\end{align}
for some constant $C_3>0$. Thus $\beta_\delta(u_\delta)\to 0$ in $L^2_{L^2}$ as $\delta\to 0$. 

Now let $\eta\in L^2_{L^2}$. We have
\begin{align}
\int_0^T m(\beta(u), \eta) \leq \int_0^T \left(\|\beta(u)-\beta(u_\delta)\|\L+ \|\beta(u_\delta) - \beta_\delta(u_\delta)\|\L + \|\beta_\delta(u_\delta)\|\L\right)  \|\eta\|\L 
\end{align}
and using \eqref{eq:beta} it follows that the right hand side converges to $0$ as $\delta\to 0$. Thus $\beta(u(t)) = 0$ almost everywhere in $\Gamma(t)$ for almost all $t\in [0,T]$, which means that $u\in K$, as desired. 
\end{proof}

\begin{remark}
In contrast with the logarithmic potential problem, in which the solutions do not touch the pure phases $\pm 1$ (see Lemma \ref{lem:oninterval}), for the double obstacle case solutions are expected to attain the values $\pm 1$ on sets of positive measure.
\end{remark}

We finally turn to the well-posedness of \eqref{eq:prob2}.

\begin{lemma}
The pair $(u,w)$ satisfies, for all $\eta\in L^2_{H^1}$ and a.e. $t\in [0, T]$,
\begin{align}
m_\ast (\partial^\bullet u, \eta) + g(u, \eta) + a_N(u,\eta) + a_S(w, \eta) = 0,
\end{align}
and $u(0)=u_0$ almost everywhere in $\Gamma_0$.
\end{lemma}

\begin{proof}
The proof of this result is identical to that of the first part of Proposition \ref{prop:existence}, by replacing $M\to\infty$ with $\delta\to 0$, so we omit it. The proof for the initial condition uses only the first equation of the system, and hence it follows also as in Proposition \ref{prop:existence}.
\end{proof}

\begin{lemma}
The pair $(u, w)$ satisfies the variational inequality.
\end{lemma}

\begin{proof}
Let $\eta\in K$ and $\xi\in C^0([0,T])$ satisfy $\xi\geq 0$. Testing the second equation with $\eta-u_\delta$, integrating over $[0,T]$ and noting that $\beta_\delta(\eta)=0$ we have, for almost every $t$,
\begin{align}
\int_0^T \xi \, a_S(u_\delta, \eta - u_\delta) - \int_0^T \xi \, m(w_\delta + u_\delta, \eta - u_\delta) &= -\dfrac{1}{\delta} \int_0^T \xi \, m(\beta_\delta(u_\delta), \eta-u_\delta) \\ 
&=\dfrac{1}{\delta} \int_0^T \xi \, m(\beta_\delta(\eta)-\beta_\delta(u_\delta), \eta-u_\delta) \geq 0,
\end{align}
by monotonicity of $\beta_\delta$. To let $\delta\to 0$, we simply observe that, by the convergences in \eqref{eq:conv3},
\begin{align}
\int_0^T \xi \, m(w_\delta+u_\delta, \eta-u_\delta) \to \int_0^T \xi \, m(w+u, \eta-u).
\end{align}
Also, the weak convergence $\nabla_\Gamma u_\delta(t)\rightharpoonup \nabla_\Gamma u(t)$ gives $$\|\nabla_\Gamma u(t)\|_{L^2(\Gamma(t))} \leq \liminf_{\delta\to 0} \|\nabla_\Gamma u_\delta(t)\|_{L^2(\Gamma(t))},$$ for almost every $t\in [0, T]$, and hence 
\begin{align}
\limsup_{\delta\to 0} \int_0^T \xi \, a_S(u_\delta, \eta - u_\delta) &= \int_0^T \xi a_S(u, \eta) -\liminf_{\delta\to 0} \int_0^T \xi \|\tgrad u_\delta\|^2\L \leq \int_0^T a_S(u, \eta-u).
\end{align}
Therefore,
\begin{align}
\int_0^T \xi a_S(u, \eta-u) &- \int_0^T \xi m(w + u, \eta - u) \\
&\hskip 4mm \geq \limsup_{\delta\to 0} \, \int_0^T \xi \, \left(a_S(u_\delta, \eta - u_\delta) - m(w_\delta + u_\delta, \eta - u_\delta) \right) \geq 0,
\end{align}
concluding the proof. Since $\xi\geq 0$ is arbitrary, it follows that
\begin{align}
a_S(u, \eta-u) - m( u, \eta - u) \geq m (w, \eta - u),
\end{align}
as desired.
\end{proof}

Before proving uniqueness, we require the following auxiliary result. For the meaning of inequalities in the $H^1$ sense, see for instance \cite[Section II.5]{KinSta80}.

\begin{lemma}\label{lem:nonempty}
Let $(u,w)$ be a solution pair for \eqref{eq:prob2} and consider the open set
\begin{align}
U(t) = \{x\in\Gamma(t) \colon |u(t,x)| < 1 \text{ in the sense of } H^1\}.
\end{align}
Then $U(t)$ is non-empty for almost all $t\in [0,T]$.
\end{lemma}

\begin{proof}
Since $u\in L^\infty_{H^1}$, we have $u(t)\in H^1(\Gamma(t))$ for almost all $t\in [0,T]$, and so the set $U(t)$ is well defined for almost all $t\in [0, T]$. Fix one such $t$, and suppose for the sake of contradiction that $U(t)=\emptyset$. Then $|u(t,x)|\geq 1$ in the sense of $H^1$ for all $x\in\Gamma(t)$, which implies that $|u(t,x)|\geq 1$ almost everywhere in $\Gamma(t)$. Since $u\in K$, also $|u(t,x)|\leq 1$ almost everywhere in $\Gamma(t)$, and hence $|u(t)|=1$ almost everywhere. As a consequence, $\nabla_\Gamma u=0$ almost everywhere (see \cite[Section 7.4]{GilTru98}).

We now have three cases to consider. In the cases where $u=1$ a.e. or $u=-1$ a.e., then we have, respectively,
\begin{align}
m_{u_0}(t) = \dfrac{1}{|\Gamma(t)|} \left|\int_{\Gamma_0} u_0 \right| = \dfrac{1}{|\Gamma(t)|} \left| \int_{\Gamma(t)} u(t) \right| = 1 
\end{align}
which contradicts $u_0\in \mathcal{I}_0$. In the case when the sets
\begin{align}
U_1(t) = \{x\in\Gamma(t) \colon u(t,x) = 1 \}, \quad U_2(t) = \{x\in \Gamma(t) \colon u(t,x)=-1\}
\end{align}
both have positive measure, then we have, using the generalized Poincar\'{e} inequality with both sets $U_1$ and $U_2$ (see e.g. \cite[Theorem 12.23]{Leo09}) and the fact that $\nabla_\Gamma u = 0$ a.e.,
\begin{align}
\int_{\Gamma(t)} |u-(u)_{U_1}|^2 \leq 0 \quad \text{ and } \quad \int_{\Gamma(t)} |u-(u)_{U_2}|^2 \leq 0, 
\end{align}
from where $u(t) \equiv (u)_{U_1}> 0$ and $u(t) \equiv (u)_{U_2} < 0$, which is also a contradiction. Hence $U(t)$ cannot be empty, proving the result.
\end{proof}

\begin{lemma}
The solution pair $(u,w)$ is unique.
\end{lemma}

\begin{proof}
Suppose $(u^1, w^1)$ and $(u^2, w^2)$ are two solution pairs and let us denote $\xi^u=u^1-u^2$, $\xi^w=w^1-w^2$. Exactly as in the proof of Proposition \ref{prop:smoothuniqueness}, we obtain
\begin{align}
m_\ast(\md\xi^u, \eta) + g(\xi^u, \eta) + a_N(\xi^u, \eta) + a_S(\xi^w, \eta) = 0,\quad \forall \eta\in L^2_{H^1}, \label{eq:eqA}
\end{align}
from where we deduce
\begin{align}
\begin{split}
\label{eq:eqC}
\dfrac{d}{dt} \|\xi^u\|^2_{-1} + a_N(\xi^u, \mathcal{G}\xi^u) + m(\xi^w, \xi^u) = m(\xi^u, \partial^\bullet\mathcal{G}\xi^u).
\end{split}
\end{align}

Now, using the variational inequalities for both pairs we directly get
\begin{align}\label{eq:eqD}
\|\nabla_\Gamma \xi^u \|\L^2\leq m(\xi^w, \xi^u)+ \|\xi^u\|\L^2,
\end{align}
and combining \eqref{eq:eqC} and \eqref{eq:eqD} yields, as in Proposition \ref{prop:smoothuniqueness}, 
\begin{align}
\dfrac{d}{dt}\|\xi^u\|_{-1}^2 + \dfrac{1}{C_1} \| \tgrad \xi^u\|^2\L \leq C_2 \|\xi^u\|_{-1}^2.
\end{align}
An application of Gronwall's inequality gives uniqueness for $u$.

We now prove uniqueness for $w$. Note that from \eqref{eq:eqA} we conclude that that $w$ is unique up to a constant. Since $u\in L^\infty_{H^1}$, for almost every $t\in [0, T]$ we have $u(t)\in H^1(\Gamma(t))$. Fix one such $t$ and define
\begin{align}
U(t) = \{ x\in \Gamma(t) \colon |u(t,x)|<1 \text{ in the sense of } H^1 \}.
\end{align}
We proved in Lemma \ref{lem:nonempty} that $U(t)$ is a non-empty open set. Choose $\varphi\in C_0^\infty(U(t))$ and $\tau>0$ sufficiently small so that $\eta_{\pm}:=u\pm\tau\varphi\in K$. Testing the second equation with both $\eta_+$ and $\eta_-$ gives, for almost all $t$, the equalities
\begin{align}
a_S(u,\varphi) &= m(u+w^1, \varphi) \quad \text{and} \quad a_S(u,\varphi) = m(u+w^2, \varphi),
\end{align} 
 and subtracting these equations gives $m(\xi^w, \varphi) = 0$. But $\xi^w$ is constant and $\varphi$ is arbitrary, so it must be $\xi^w = 0$, i.e. $w^1=w^2$.
\end{proof}

We collect the results of this section in the next theorem. 

\begin{theorem}\label{thm:wellposeddoubleobs}
Let $u_0\in \mathcal{I}_0$. There exists a unique pair $(u, w)$ such that 
\begin{align}
u\in H^1_{H^{-1}}\cap L^\infty_{H^1}\cap K \quad \text{ and } \quad w\in L^2_{H^1}
\end{align}
satisfying, for almost every $t\in [0,T]$,
\begin{align}
\begin{split}
m_\ast (\md u, \eta) + g(u, \eta) + a_N(u, \eta) + a_S(w,\eta) = 0, \quad &\forall \eta\in L^2_{H^1}, \\
a_S(u, \eta - u) - m(u, \eta-u) \geq m(w,\eta-u), \quad &\forall \eta\in K,
\end{split}
\end{align}
and $u(0)=u_0$ almost everywhere in $\Gamma_0$. The solution $u$ satisfies the additional regularity 
\begin{align}
u\in C^0_{L^2} \cap L^\infty_{L^p}, \quad \text{ for all } p\in [1,+\infty).
\end{align}

Furthermore, if $u_0, v_0\in \mathcal{I}_0$ satisfy $(u_0)_{\Gamma_0}=(v_0)_{\Gamma_0}$, and $u, v$ are the solutions of the system with $u(0)=u_0$, $v(0)=v_0$, then there exist constants $C, K>0$ independent of $t$ such that, for almost all $t\in [0,T]$,
\begin{align}
\|u(t)-v(t)\|^2_{-1} \leq Ce^{Kt} \|u_0-v_0\|^2_{{-1}}.
\end{align}
\end{theorem}

\subsubsection{Extra regularity}

As in the previous sections, we conclude by establishing extra regularity for the solution $u$. 

\begin{theorem}\label{thm:reg_obs}
Let $(u,w)$ be the solution pair given by Theorem \ref{thm:wellposeddoubleobs}. Then we have the regularity $u\in L^2_{H^2}$.
\end{theorem}

\begin{proof}
Let us consider the solution $(u_\delta^M, w_\delta^M)$ of the Galerkin approximation \eqref{eq:gaprob2eps}. The second equation can be rewritten as 
\begin{align}
a_S(u^M_\delta, \eta) + \dfrac{1}{\delta} \, m(\beta_\delta(u^M_\delta), \eta) - m(u^M_\delta, \eta) - m(w^M_\delta, \eta)&=0, \quad \forall \eta\in L^2_{V_M},
\end{align} 
where $\beta_\delta$ is as in \eqref{eq:betadefinition}. Testing the above with $\eta=-\Delta_\Gamma u_\delta^M$ leads to 
\begin{align}
\|\Delta_\Gamma u^M_\delta\|_{L^2}^2 + \dfrac{1}{\delta}\, m\big(\beta_\delta'(u^M_\delta) \tgrad u_\delta^M, \tgrad u_\delta^M)\big) &= a_S (w^M_\delta, u^M_\delta) + \|\tgrad u_\delta^M\|^2_{L^2} \\
&\leq C_1 \|w_\delta^M\|^2_{H^1} + C_2 \|u_\delta^M\|^2_{H^1}
\end{align}
and since $\beta'_\delta\geq 0$ this implies that
\begin{align}
\|\Delta_\Gamma u_\delta^M\|^2_{L^2} \leq C_1 \|w_\delta^M\|^2_{H^1} + C_2 \|u_\delta^M\|^2_{H^1}.
\end{align}
The uniform bounds for $u_\delta^M$ in $L^\infty_{H^1}$ and for $w_\delta^M$ in $L^2_{H^1}$ yield 
\begin{align}
\|\Delta_\Gamma u_\delta^M\|_{L^2_{L^2}} \leq C, \quad \text{ for some } C>0,
\end{align}
and elliptic regularity theory then imply that $u_\delta^M$ is uniformly bounded in $L^2_{H^2}$. It follows that $u\in L^2_{H^2}$, as desired.
\end{proof}

Lemma \ref{lem:assumptionsgive}d) also implies $u(t)\in C^\alpha(\Gamma(t))$, for a.a. $t\in [0,T]$ and all $\alpha\in (0,1)$.

To deduce higher regularity properties of $u$ from the second equation, we need regularity results for the solution of the obstacle problem 
\begin{align*}
-\Delta_\Gamma u - u \geq w, \quad -1\leq u\leq 1.
\end{align*}
Even in the classical obstacle problem with a smooth right hand side, one obtains at most $H^2$-regularity for the solution (which corresponds to our result in Theorem \ref{thm:reg_obs}), and extra smoothness on the set where the solution does not coincide with the obstacle. As for the logarithmic potential, we leave such analysis for future work. 

\subsection{Relating the two singular models}

Recalling the definitions for the logarithmic free energy
\begin{align}\label{eq:logpotential2}
F_\theta(r) = \dfrac{\theta}{2\theta_c} \left((1+r)\log(1+r) + (1-r)\log(1-r)\right) + \dfrac{1-r^2}{2},
\end{align}
and the double obstacle potential
\begin{align}\label{eq:obspotential2}
F_o(r) = 
\begin{cases}
\frac{1}{2}(1-r^2), &\text{ if } |r|\leq 1 \\
+\infty, &\text{ if } |r|>1
\end{cases}, 
\end{align}
we can formally see \eqref{eq:logpotential2} as the limit of \eqref{eq:obspotential2} as the temperature $\theta\to 0$, and we refer to the double obstacle problem \eqref{eq:prob2} as the \textit{deep quench limit} of the logarithmic problem \eqref{eq:prob3}. It is then natural to ask whether solutions $(u^\theta, w^\theta)$, where we now explicitly write the dependence on $\theta$, to \eqref{eq:prob3} converge, as $\theta\to 0$, to the unique solution of \eqref{eq:prob2}. Our aim in this short section is to prove that this is indeed the case.

We start by noticing that, from the estimates in Section \ref{sec:log}, it follows that $u^\theta$, $w^\theta$ and $\md u^\theta$ are uniformly bounded in $L^\infty_{H^1}$, $L^2_{H^1}$ and $L^2_{H^{-1}}$, respectively, from where we obtain functions $u, w\in L^2_{H^1}$ with $\md u\in L^2_{H^{-1}}$ and we have convergences, as $\theta\to 0$,  
\begin{align}
u^\theta\overset{*}{\rightharpoonup}u \, \text{ in } \, L^\infty_{H^1}, \quad w^\theta\rightharpoonup w \, \text{ in } \, L^2_{H^1}, \quad \md u^\theta\rightharpoonup \md u \, \text{ in } \, L^2_{H^{-1}}.
\end{align}
By the Aubin-Lions Lemma in Proposition \ref{prop:inclusions}(c) we also obtain the strong convergence $u^\theta \to u$ in $L^2_{L^2}$. 

\begin{theorem}
The limit pair $(u,w)$ is the unique solution of the double obstacle problem \eqref{eq:prob2}.
\end{theorem}

\begin{proof}
It follows immediately that $(u,w)$ satisfies, for all $\eta\in L^2_{H^1}$, 
\begin{align}
m_*(\md u, \eta) + g(u, \eta) + a_N(u, \eta) + a_S(w,\eta) = 0,
\end{align}
and combining $u^\theta\to u$ pointwise a.e. with $|u^\theta|<1$ we also obtain $u\in K$. 

To show that $(u,w)$ also satisfies the variational inequality, we consider $\eta\in K$, a small parameter $\alpha\in (0,1)$ and define $\eta_\alpha=(1-\alpha)\eta$, so that $|\eta|\leq 1-\alpha<1$ almost everywhere. Fix also $\xi\in C^0([0,T])$ such that $\xi\geq 0$. Testing the second equation with $\eta=\xi\left(\eta_\alpha-u^\theta\right)$ and integrating over $[0,T]$ leads to
\begin{align}
\int_0^T\xi \int_{\Gamma(t)} & \tgrad u^\theta\cdot\tgrad  (\eta_\alpha-u^\theta) -\int_0^T \xi \int_{\Gamma(t)} u^\theta (\eta_\alpha-u^\theta) - \int_0^T \xi\int_{\Gamma(t)} w^\theta (\eta_\alpha-u^\theta) \\
&= \int_0^T \xi \int_{\Gamma(t)} \varphi_\theta(u^\theta)(u^\theta - \eta_\alpha) \\
&= \int_0^T \xi \int_{\Gamma(t)} \left(\varphi_\theta(u^\theta)-\varphi_\theta(\eta_\alpha)\right) \left(u^\theta-\eta_\alpha\right) + \int_0^T \xi \int_{\Gamma(t)} \varphi_\theta(\eta_\alpha) \left(u^\theta-\eta_\alpha\right) \\
&\geq \int_0^T \xi \int_{\Gamma(t)} \varphi_\theta(\eta_\alpha) (u^\theta-\eta_\alpha),
\end{align}
due to monotonicity of $\varphi_\theta$. 

We now pass to the limit $\theta\to 0$. For the integral on the right hand side, we use $|\varphi_\theta (\eta_\alpha)|\leq C_\alpha \theta$, so that 
\begin{align}
\left|\int_0^T \xi \int_{\Gamma(t)} \varphi_\theta(\eta_\alpha) (u^\theta-\eta_\alpha) \right| \leq  C_{\alpha, \xi} \, \theta \int_0^T \int_{\Gamma(t)} \left| u^\theta-\eta_\alpha\right| \to 0 \quad \text{ as } \theta\to 0,
\end{align}
since $u^\theta\to u$ strongly in $L^2_{L^2}$. As for the terms on the left hand side, we have directly from the convergence results for $u^\theta$ and $w^\theta$ that 
\begin{align}
\int_0^T \xi \int_{\Gamma(t)} u^\theta (\eta_\alpha - u^\theta) \to \int_0^T \xi \int_{\Gamma(t)} u (\eta_\alpha - u)
\end{align}
and
\begin{align}
\int_0^T \xi \int_{\Gamma(t)} w^\theta (\eta_\alpha - u^\theta) \to \int_0^T \xi \int_{\Gamma(t)} w (\eta_\alpha - u),
\end{align}
and from weak lower semicontinuity of the norm we also have 
\begin{align}
\int_0^T \xi \int_{\Gamma(t)} \tgrad u\cdot \tgrad (\eta_\alpha - u) \geq \liminf_{\theta\to 0} \int_0^T \xi \int_{\Gamma(t)} \tgrad u^\theta\cdot\tgrad (\eta_\alpha-u^\theta).
\end{align}
Combining the three displayed identities above leads to 
\begin{align}
\int_0^T \xi \int_{\Gamma(t)} \tgrad u\cdot \tgrad (\eta_\alpha - u)  - \int_0^T \xi \int_{\Gamma(t)} u (\eta_\alpha - u) \geq \int_0^T \xi \int_{\Gamma(t)} w (\eta_\alpha - u),
\end{align}
and $\alpha\to 0$ yields the variational inequality 
\begin{align}
\int_0^T \xi \int_{\Gamma(t)} \tgrad u\cdot \tgrad (\eta - u)  - \int_0^T \xi \int_{\Gamma(t)} u (\eta - u) \geq \int_0^T \xi \int_{\Gamma(t)} w (\eta - u),
\end{align}
Since $\xi\geq 0$ is arbitrary we obtain, for almost all $t\in [0,T]$, the desired variational inequality 
\begin{align}
\int_{\Gamma(t)} \tgrad u\cdot \tgrad (\eta - u)  -\int_{\Gamma(t)} u (\eta - u) \geq \int_{\Gamma(t)} w (\eta - u),
\end{align}
finishing the proof. 
\end{proof}

\section{Global well-posedness for a related model}\label{sec:alternative}

The aim of this section is to analyse an alternative derivation of the Cahn-Hilliard equation on an evolving surface presented in the recent article \cite{ZimTosLan19} (we refer the reader to Sections 4.1, 4.2, 5.2, and Appendix A therein). For this alternative problem, we establish, using the same techniques as above, global well-posedness results which we can now prove, even for the singular potentials, without any additional assumptions on the evolution of the surfaces or the initial data. For simplicity, and since this does not change any of the results, we assume that there are no tangential velocities in the model (i.e. we take $\mathbf{V}_\tau = \mathbf{V}_a = 0$). 

\subsection{Well-posedness}

As explained in the Introduction, we are interested in finding a pair $(c,w)$ satisfying, for all $\eta\in L^2_{H^1}$, 
\begin{align}\label{eq:newsystem}\tag{CH$_\rho$}
\begin{split}
\left \langle \rho \md c, \eta\right\rangle_{H^{-1}, \, H^1} + \int_{\Gamma(t)} \rho \tgrad w \cdot \tgrad \eta  &= 0, \\
 \int_{\Gamma(t)} \tgrad c \cdot \tgrad \eta  +  \int_{\Gamma(t)} \rho F'(c) \eta   &= \int_{\Gamma(t)} \rho w \eta ,
 \end{split}
\end{align}
where the weight function $\rho$ is determined by the differential equation
\begin{align}\label{eq:rho}
\md \rho + \rho \tgrad\cdot\mathbf{V} = 0.
\end{align}
A detailed derivation of the model above can be found in \cite{ZimTosLan19}, but the main idea is similar to that of Section \ref{sec:model}. The weight term $\rho$ represents the total density of the system, and \eqref{eq:rho} is simply conservation of total mass. The system of equations above is obtained by considering a conservation law now for the quantity 
\begin{align}
\int_{\Gamma(t)} \rho c ,
\end{align}
which is, as a consequence, preserved along the system (simply test the first equation with $\eta=1$). The chemical potential can be seen as the functional derivative of the following altered Cahn-Hilliard free energy functional
\begin{align}\label{eq:newenergy}
\mathrm{E}_{\mathrm{CH}}^\rho(c) := \int_{\Gamma(t)} \dfrac{|\tgrad c|^2}{2} + \rho  F(c) .
\end{align}

Before we establish the a priori bounds for \eqref{eq:newsystem}, we note that \eqref{eq:rho} gives an explicit formula for the weight function $\rho$. Indeed, by evaluating the equation along the flow, we obtain, for all $p\in \Gamma_0$ and $t\in [0,T]$, 
\begin{align}
\dfrac{d}{dt} \rho\left(t, \Phi(t,p)\right) + \rho\left(t, \Phi(t,p)\right) \tgrad \cdot \mathbf{V}\left(t, \Phi(t,p)\right) = 0,
\end{align}
from where
\begin{align}\label{eq:formularho}
\rho\left(t, \Phi(t,p)\right) = \exp \left\{- \int_0^t\tgrad \cdot \mathbf{V}\left(s, \Phi(s,p)\right) \right\}.
\end{align}
In terms of elements $x\in \Gamma(t)$ at time $t\in [0,T]$, the above reads as
\begin{align}
\rho\left(t, x\right) = \exp \left\{- \int_0^t\tgrad \cdot \mathbf{V}\left(s, x\right) \right\}.
\end{align}
In particular, there exists $C_\rho>1$ such that
\begin{align}
0 < \dfrac{1}{C_\rho} \leq \rho \leq C_\rho \quad \text{ and } \quad |\tgrad \rho|\leq C_\rho.
\end{align}
It is also worthwhile comparing \eqref{eq:formularho} with \eqref{eq:formulaJ}, which shows that
\begin{align}\label{eq:rhoisJ}
\rho\left(t, \Phi(t,p)\right) = (J_t^0(p))^{-1}, \, \forall p\in\Gamma_0
\end{align}
and therefore the change of variables formulae in \eqref{eq:changeofvariables} imply that
\begin{align}
\int_{\Gamma(t)} c(x) = \int_{\Gamma_0} \dfrac{\tilde c(p)}{\rho\left(t, \Phi(t,p)\right)} , \quad
\int_{\Gamma_0} \tilde c(p)  = \int_{\Gamma(t)} c(x) \rho(t, x).
\end{align}

This makes the model particularly interesting for the cases of a logarithmic and a double obstacle potential. Recall that we started both Sections \ref{sec:log}, \ref{sec:obs} by showing that a condition relating the initial data and the areas of the surfaces was necessary for well-posedness of the systems, and this was a consequence of the fact that the integral of the solution was preserved along the evolution. This is intuitively clear: preservation of the integral together with a decrease in the areas of the surfaces would force an increase on the magnitude of the solutions, which is precluded by the condition $|u|\leq 1$ in either case. This simple argument no longer holds true for the problem \eqref{eq:newsystem}. Due to \eqref{eq:rhoisJ}, preservation of the integral of $\rho c$ can be compatible with restrictions on $|c|$, since the term $\rho$ can now account for local stretching or compressing of the surfaces. 

We now show that indeed the presence of the weight term allows for global well-posedness results for the system \eqref{eq:newsystem} for the three different types of nonlinearities we considered in the previous sections.

\subsection{A priori estimates}

In this section we obtain some a priori estimates for solutions of the problem above. The calculations below are just formal, but can be made rigorous by arguing as in the previous sections (e.g reasoning by approximation). So let us denote by $(c,w)$ a solution pair to \eqref{eq:newsystem}. By differentiating the energy \eqref{eq:newenergy} along this solution pair, we have
\begin{align}
\dfrac{d}{dt} \mathrm E_{\mathrm{CH}}^\rho(c) &= \int_{\Gamma(t)} \tgrad c \cdot \tgrad \md c + \rho F'(c) \md c \, d\Gamma  + \int_{\Gamma(t)} B(\mathbf{V}) |\tgrad c|^2 \\
&= \int_{\Gamma(t)} \rho w \md c  + \int_{\Gamma(t)}  B(\mathbf{V}) |\tgrad c|^2 \,  \\
&= -\int_{\Gamma(t)} \rho |\tgrad w|^2  +  \int_{\Gamma(t)} B(\mathbf{V}) |\tgrad c|^2,
\end{align}
and thus
\begin{align}
\dfrac{d}{dt} \mathrm{E}_{\mathrm{CH}}^\rho (c) + \int_{\Gamma(t)} \rho|\tgrad w|^2  =  \int_{\Gamma(t)} B(\mathbf{V}) |\tgrad c|^2 \leq C_{\mathbf{V}} \int_{\Gamma(t)} |\tgrad c|^2  \leq C\, \mathrm{E}_\mathrm{CH}^\rho(c).
\end{align}
Recalling that $\rho$ is strictly positive, an application of Gronwall's inequality immediately gives the energy estimate
\begin{align}\label{eq:newenergyestimate}
\esssup_{t\in [0,T]} \mathrm{E}_{\mathrm{CH}}^\rho (c) + \int_0^T \int_{\Gamma(t)} |\tgrad w|^2  \leq \mathrm{E}_{\mathrm{CH}}^\rho (c_0) + C,
\end{align}
Asssuming $\mathrm{E}_{\mathrm{CH}}^\rho (c_0)$ to be finite, this implies in particular
\begin{align}
\tgrad c \in L^\infty_{L^2} \quad \text{ and } \quad \tgrad w \in L^2_{L^2}
\end{align}

We make use of \eqref{eq:newenergyestimate} to estimate the remaining quantities. Observe that, using the first equation with $\eta=\rho^{-1}$,
\begin{align}
\dfrac{d}{dt}\int_{\Gamma(t)} c  &= \int_{\Gamma(t)} \md c  + \int_{\Gamma(t)} c \tgrad\cdot\mathbf{V}  \\
&= -\int_{\Gamma(t)} \rho \tgrad w \cdot \tgrad \left(\dfrac{1}{\rho}\right)  - \int_{\Gamma(t)} \tgrad c \cdot \mathbf{V},
\end{align}
so that
\begin{align}
\int_{\Gamma(t)} c = \int_{\Gamma_0} c_0 - \int_0^t \int_{\Gamma(s)} \rho \tgrad w \cdot \tgrad \left(\dfrac{1}{\rho}\right) - \int_0^t \int_{\Gamma(t)} \tgrad c \cdot \mathbf{V} 
\end{align}
and thus 
\begin{align}
\left| \int_{\Gamma(t)} c\right| \leq \left| \int_{\Gamma_0} c_0 \right | + CT \leq C'.
\end{align}
Combining the above, \eqref{eq:newenergyestimate}, and Poincar\'{e}'s inequality then implies 
\begin{align}\label{eq:newu}
c\in L^\infty_{H^1}.
\end{align}
Using the Sobolev embedding $H^1(\Gamma(t))\hookrightarrow L^p(\Gamma(t))$, for all $t\in [0,T]$ and $p\in [1,+\infty)$, this in particular implies
\begin{align}
c\in L^\infty_{L^p}, \quad \text{ for all } p\in [1, +\infty).
\end{align}
We can also easily estimate the material derivative of $c$ by noticing that, for any $\eta\in L^2_{H^1}$, 
\begin{align}
\langle \md c, \eta\rangle_{L^2_{H^{-1}}, \, L^2_{H^1}} &= \int_0^T \int_{\Gamma(t)} \md c(t) \eta(t) \, \\
&= -\int_0^T \int_{\Gamma(t)} \rho \nabla_{\Gamma(t)} w(t) \cdot \nabla_{\Gamma(t)} \left(\dfrac{\eta(t)}{\rho(t)}\right)  \\
&= -\int_0^T \int_{\Gamma(t)} \tgrad w \cdot \tgrad \eta + \int_0^T \int_{\Gamma(t)} \dfrac{\eta}{\rho} \tgrad w \cdot \tgrad \rho \\
&\leq C_1 \|\tgrad \eta\|_{L^2_{L^2}} + C_2 \|\eta\|_{L^2_{L^2}} \\
&\leq C_3 \|\eta\|_{L^2_{H^1}},
\end{align}
which gives
\begin{align}\label{eq:newmd}
\md c\in L^2_{H^{-1}}.
\end{align}
Combining \eqref{eq:newu} and \eqref{eq:newmd} also gives the extra regularity
\begin{align}
c\in C^0_{L^2}.
\end{align}

The estimate for $w$ in $L^2_{L^2}$ is slightly more involved. 

In the case that the nonlinearity $F$ satisfies Assumptions (A1)-(A2) as in Section \ref{sec:smooth}, we obtain from the second equation with $\eta=1/\rho$  
\begin{align}
\left| \int_{\Gamma(t)} w \right| &= \left| \int_{\Gamma(t)} \tgrad c \cdot \tgrad \left(\dfrac{1}{\rho}\right) + \int_{\Gamma(t)} F'(c) \right| \\
&\leq C_1 + C_2 \|c\|^q_{H^1} \\
&\leq C_3,
\end{align}
which combined with Poincar\'e's inequality and the energy estimate implies a bound for $w$ in $L^2_{H^1}$. 

Let us now focus in the case of a singular potential, where we assume $(c_0)_{\Gamma_0}\in (-1,1)$. As in Sections \ref{sec:log}, \ref{sec:obs} we obtain the a priori bounds for a regularised problem, where the potential $F$ is approximated by a regular $F_\delta$ with polynomial growth. Denoting $\varphi_\delta = F'_\delta$, the approximations are chosen so that, for all $r\in \mathbb R$,
\begin{align}\label{eq:asspotential}
|\varphi_\delta(r)| \leq r\varphi_\delta(r) + 1.
\end{align}

Our aim is to estimate the integral of $w$. We have
\begin{align}
 \left| \int_{\Gamma(t)} w \right| &\leq \left| \int_{\Gamma(t)} \tgrad c \cdot \tgrad \left(\dfrac{1}{\rho}\right) \right|+ \left| \int_{\Gamma(t)} \varphi_\delta(c) \right| \leq C + \int_{\Gamma(t)} |\varphi_\delta(c)|.
\end{align}
Using \eqref{eq:asspotential} it follows that
\begin{align}
\int_{\Gamma(t)} |\varphi_\delta(c)| 
&\leq \int_{\Gamma(t)} c\varphi_\delta(c)+1  \\
&\leq C_\rho \int_{\Gamma(t)} \rho (c \varphi_\delta(c)+1)  \\
&\leq  C_\rho \int_{\Gamma(t)} \rho c \varphi_\delta(c)  + C_\rho \int_{\Gamma(t)} \rho \\
&\leq C_1\int_{\Gamma(t)} \rho c \varphi_\delta(c) + C_2
\end{align}
and from the second equation tested with $\eta=c$ we also obtain 
\begin{align}\label{eq:third_est}
\left|  \int_{\Gamma(t)} \rho \varphi_\delta(c)c  \right| &= \left|\int_{\Gamma(t)} \rho c w - \int_{\Gamma(t)}  |\tgrad c|^2  \right| \nonumber \\
&\leq \left|\int_{\Gamma(t)} \rho c w \right| + \left|\int_{\Gamma(t)}  |\tgrad c|^2  \right| \nonumber \\
&\leq C + \left|\int_{\Gamma(t)} \rho c w \right|.
\end{align}
Combining the three previous estimates leads to 
\begin{align}\label{eq:w_integral}
 \left| \int_{\Gamma(t)} w \right| \leq C_1 + C_2 \left|\int_{\Gamma(t)} \rho c w \right|.
\end{align}

We now focus on estimating the integral of $\rho c w$. Our strategy will be to pullback the integrals to the reference domain $\Gamma_0$ and estimate the involved quantities there. This is helpful since we now have no control on the size of 
\begin{align}
(c(t))_{\Gamma(t)} = \dfrac{1}{|\Gamma(t)|} \int_{\Gamma(t)} c(t) 
\end{align}
but rather on its pullback
\begin{align}
(\tilde c(t))_{\Gamma_0} = \dfrac{1}{|\Gamma_0|} \int_{\Gamma_0} \tilde{c}(t) =  \dfrac{1}{|\Gamma_0|} \int_{\Gamma(t)} c(t) \rho(t)  = \dfrac{1}{|\Gamma_0|} \int_{\Gamma_0} c_0  = (c_0)_{\Gamma_0} \in (-1, 1).
\end{align}
We then write 
\begin{align}
 \int_{\Gamma(t)} \rho c w  &= \int_{\Gamma_0} \tilde c \tilde w \\
&= \int_{\Gamma_0} \tilde w \big(\tilde c - (\tilde c)_{\Gamma_0}\big) +  (\tilde c)_{\Gamma_0} \int_{\Gamma_0} \tilde w  \\
&= \int_{\Gamma_0} \big(\tilde w- (\tilde w)_{\Gamma_0}\big) \big(\tilde c - (\tilde c)_{\Gamma_0}\big) + (c_0)_{\Gamma_0}  \int_{\Gamma(t)} \rho w
\end{align}
so that
\begin{align}
\left| \int_{\Gamma(t)} \rho c w \right| &\leq C_1 \|\tgrad \tilde w\|_{L^2(\Gamma_0)} \|\tgrad \tilde c\|_{L^2(\Gamma_0))}+ | (c_0)_{\Gamma_0} | \, \left| \int_{\Gamma(t)} \rho w \right| \\
&\leq C_2\|\tgrad \tilde w\|_{L^2(\Gamma_0)} + |(c_0)_{\Gamma_0}| \, \left| \int_{\Gamma(t)} \rho w \right| \\
&\leq C_3 \|\tgrad w\|_{L^2(\Gamma(t))} + |(c_0)_{\Gamma_0}| \, \left| \int_{\Gamma(t)} \rho w \right|
 \end{align}
But we also have, testing the second equation with $\eta = 1$ and again by \eqref{eq:asspotential},
\begin{align}
\left|\int_{\Gamma(t)} \rho w \right| &= \left| \int_{\Gamma(t)} \rho \varphi_\delta(c)\right| \leq \int_{\Gamma(t)} \rho |\varphi_\delta(c)| \leq C_1 + \int_{\Gamma(t)} \rho \varphi_\delta(c) c  \leq C_2 + \left| \int_{\Gamma(t)} \rho c w \right|
\end{align}
where the last inequality is \eqref{eq:third_est}. Combining the two estimates above leads to 
\begin{align}
\left| \int_{\Gamma(t)} \rho c w\right| \leq C_1 + C_2\|\tgrad w\|_{L^2(\Gamma(t))} + |(c_0)_{\Gamma_0}| \, \left|  \int_{\Gamma(t)} \rho c w\right|,
\end{align}
from where
\begin{align}
\left|  \int_{\Gamma(t)} \rho c w\right| \leq \dfrac{C_1}{1-|(c_0)_{\Gamma_0}|} + \dfrac{C_2}{1-|(c_0)_{\Gamma_0}|} \|\tgrad w\|_{L^2(\Gamma(t))}.
\end{align}
This then gives, from \eqref{eq:w_integral},
\begin{align}
 \left| \int_{\Gamma(t)} w \right| \leq C_1 + C_2 \|\tgrad w\|_{L^2(\Gamma(t))}.
\end{align}
By Poincar\' e's inequality, 
\begin{align*}
\int_0^T \|w\|_{H^1(\Gamma(t))}^2 &\leq C_1 \int_0^T \|\tgrad w\|_{L^2(\Gamma(t))}^2 + C_2 \int_0^T \left| \int_{\Gamma(t)} w \right|^2 \\
&\leq C_3 + C_4 \int_0^T \|\tgrad w\|^2_{L^2(\Gamma(t))}\\ 
&\leq C_5
\end{align*}
which implies the desired uniform bound for $w$ in $L^2_{H^1}$.

The a priori bounds being established, we can then argue as in Sections \ref{sec:smooth}, \ref{sec:log}, \ref{sec:obs} to obtain well-posedness results for \eqref{eq:newsystem}. We denote
\begin{align}
(\cdot)_{\rho, \Gamma_0} := (\rho \, \cdot)_{\Gamma_0} \quad \text{ and } \quad \|\cdot\|^2_{\rho, -1} := \|\rho \, \cdot\|^2_{-1}.
\end{align}

We then have the analogue of Theorem \ref{thm:wellposedsmooth}.

\begin{theorem}
Let $c_0\in H^1(\Gamma_0)$, $F\colon\R\to\R$ be a smooth potential satisfying (A1)-(A2), and let the density function $\rho$ satisfy
\begin{align}
\md \rho + \rho \tgrad\cdot \mathbf{V} = 0.
\end{align}
Then, there exists a unique pair $(c, w)$ with $$c\in H^1_{H^{-1}}\cap L^\infty_{H^1} \quad \text{ and } \quad w\in L^2_{H^1}$$ satisfying the system \eqref{eq:newsystem}, and the initial condition $c(0)=c_0$ almost everywhere in $\Gamma_0$. The solution $c$ satisfies the additional regularity 
\begin{align}
c\in C^0_{L^2} \cap L^2_{H^2} \cap L^\infty_{L^p}, \quad \text{ for all } p\in [1,+\infty).
\end{align}

Furthermore, if $c_{0,1}, c_{0,2}\in H^1(\Gamma_0)$ satisfy $(c_{0,1})_{\rho, \Gamma_0}=(c_{0,2})_{\rho ,\Gamma_0}$, and $c_1, c_2$ are the solutions of the system with $c_1(0)=c_{1,0}$, $c_2(0)=c_{0,2}$, then there exists a constant $C>0$ independent of $t$ such that, for almost all $t\in [0,T]$,
\begin{align}
\|c_1(t)-c_2(t)\|_{\rho, -1} \leq e^{Ct} \|c_{0,1}-c_{0,2}\|_{{\rho, -1}}.
\end{align}
\end{theorem}

The proof is essentially the same as that of Theorem \ref{thm:wellposedsmooth} apart from the uniqueness result, where one should now work with the weighted inverse Laplacian defined, for any $f\in H^{-1}(\Gamma(t))$ such that $m_\ast^\rho(f, 1)=0$, as the unique solution $c=(\Delta_{\Gamma, \rho})^{-1} f\in H^1(\Gamma(t))$ to the problem
\begin{align}
a_S (c, \eta) &= m_\ast^\rho(f, \eta) \\ 
m(c, 1) &= 0.
\end{align}
Under further assumptions on both $F$ and the evolution of the surfaces, we can also obtain extra regularity for $c$ as in Theorem \ref{thm:regularity}.

Similarly, we have a well-posedness result for the logarithmic potential model as in Theorem \ref{thm:wellposedlog}.

\begin{theorem}
Let $c_0\in H^1(\Gamma_0)$ satisfy $|c_0|\leq 1$, $(c_0)_{\Gamma_0}\in (-1,1)$, and $\mathrm{E}_{\mathrm{CH}}^\rho(c_0)<\infty$, $F\colon[-1,1]\to\R$ be the logarithmic potential \eqref{eq:logpotential}, and let the density function $\rho$ satisfy
\begin{align}
\md \rho + \rho \tgrad\cdot \mathbf{V} = 0.
\end{align}
Then, there exists a unique pair $(c, w)$ with $$u\in H^1_{H^{-1}}\cap L^\infty_{H^1} \quad \text{ and } \quad w\in L^2_{H^1}$$ satisfying, for almost all $t\in [0,T]$, $|c(t)|<1$ a.e. in $\Gamma(t)$, the system \eqref{eq:newsystem}, and the initial condition $c(0)=c_0$ almost everywhere in $\Gamma_0$. The solution $c$ satisfies the additional regularity 
\begin{align}
c\in C^0_{L^2} \cap L^2_{H^2} \cap L^\infty_{L^p}, \quad \text{ for all } p\in [1,+\infty).
\end{align}

Furthermore, if $c_{0,1}, c_{0,2}\in H^1(\Gamma_0)$ satisfy $(c_{0,1})_{\rho, \Gamma_0}=(c_{0,2})_{\rho ,\Gamma_0}$, and $c_1, c_2$ are the solutions of the system with $c_1(0)=c_{1,0}$, $c_2(0)=c_{0,2}$, then there exists a constant $C>0$ independent of $t$ such that, for almost all $t\in [0,T]$,
\begin{align}
\|c_1(t)-c_2(t)\|_{\rho, -1} \leq e^{Ct} \|c_{0,1}-c_{0,2}\|_{{\rho, -1}}.
\end{align}
\end{theorem}

Also in this case the proof carries over unchanged from that of Theorem \ref{thm:wellposedlog}, with the exception of the proof of uniqueness which should use, as explained before, the weighted inverse Laplacian $(\Delta_{\Gamma, \rho})^{-1}$. 

Finally, we state the result for the double obstacle potential which is analogous to Theorem \ref{thm:wellposeddoubleobs}. Reusing the notation
\begin{align}
K := \{\eta\in L^2_{H^1} \colon |\eta(t)| \leq 1 \text{ a.e. on } \Gamma(t) \text{ for a.e. } t\in [0, T]\},
\end{align}
the second equation becomes a variational inequality due to non-differentiability of the potential, and we consider the problem
\begin{align}\label{eq:newdoubleobs}
\begin{split}
\left \langle \rho \md c, \eta\right\rangle_{H^{-1}, \, H^1} + \int_{\Gamma(t)} \rho \tgrad w \cdot \tgrad \eta  = 0, \quad &\forall \eta \in L^2_{H^1},\\
 \int_{\Gamma(t)} \tgrad c \cdot \tgrad (\eta-c)  - \int_{\Gamma(t)} c(\eta - c) \geq \int_{\Gamma(t)} \rho w (\eta-c) , \quad &\forall \eta\in K.
 \end{split}
\end{align}

\begin{theorem}
Let $c_0\in H^1(\Gamma_0)$ satisfy $|c_0|\leq 1$ and $(c_0)_{\Gamma_0}\in (-1,1)$, and let the density function $\rho$ satisfy
\begin{align}
\md \rho + \rho \tgrad\cdot \mathbf{V} = 0.
\end{align} 
There exists a unique pair $(c, w)$ with $$c\in H^1_{H^{-1}}\cap K\cap L^\infty_{H^1} \quad \text{ and } \quad w\in L^2_{H^1}$$ satisfying \eqref{eq:newdoubleobs} and $c(0)=c_0$ almost everywhere in $\Gamma_0$. The solution $c$ satisfies the additional regularity 
\begin{align}
c\in C^0_{L^2} \cap L^2_{H^2} \cap L^\infty_{L^p}, \quad \text{ for all } p\in [1,+\infty).
\end{align}

Furthermore, if $c_{0,1}, c_{0,2}\in H^1(\Gamma_0)$ satisfy $(c_{0,1})_{\rho, \Gamma_0}=(c_{0,2})_{\rho ,\Gamma_0}$, and $c_1, c_2$ are the solutions of the system with $c_1(0)=c_{1,0}$, $c_2(0)=c_{0,2}$, then there exists a constant $C>0$ independent of $t$ such that, for almost all $t\in [0,T]$,
\begin{align}
\|c_1(t)-c_2(t)\|_{\rho, -1} \leq e^{Ct} \|c_{0,1}-c_{0,2}\|_{{\rho, -1}}.
\end{align}
\end{theorem} 

Again, most proofs carry over with minor adaptations.

\section{Examples}

\numberwithin{equation}{section}

In this section, we collect five simple examples to illustrate our main results.

\begin{example}\label{eg:example1}
The simplest example to consider is the case $\mathbf{V}\equiv 0$, which implies that $\Gamma(t)\equiv \Gamma_0$, for all $t\in [0,T]$, and our work provides a proof for well-posedness of the Cahn-Hilliard system on the stationary surface $\Gamma_0$ for initial data $u_0\in H^1(\Gamma_0)$ and satisfying, for the singular potentials, the condition $|(u_0)_{\Gamma_0}| = m_{u_0}(0) < 1$.
\end{example}

\begin{example}\label{eg:example2}
Consider a velocity field $\mathbf V$ satisfying $\tgrad \cdot \mathbf V = 0$, implying $|\Gamma(t)|=|\Gamma_0|$ for all $t\in [0,T]$. This assumption arises in models for inextensible membranes. In this case, our work proves well-posedness for the Cahn-Hilliard system an evolving surface $\{\Gamma(t)\}_{t\in [0,T]}$ for initial data $u_0\in H^1(\Gamma_0)$ and satisfying, for the singular potentials, the condition $|(u_0)_{\Gamma_0}| = m_{u_0}(0) < 1$. As for the model in Section \ref{sec:alternative}, in this case we obtain $\md \rho = 0$, and therefore (up to multiplication by a constant) the two models we presented are the same. This is also the setting considered in \cite{YusQuaOls20}, and our work completes their results by providing a proof of well-posedness for the continuous model.
\end{example}

\begin{example}\label{eg:example3}
Let the nonlinearity $F$ be the usual quartic potential
\begin{align}\label{quartic}
F(r) = \dfrac{(r^2-1)^2}{4} = \dfrac{r^4}{4} + \dfrac{1-2r^2}{4} =: F_1(r) + F_2(r), \quad r\in \R.
\end{align}
Then it is easy to check that assumptions (A1)-(A2) are satisfied with (e.g.) the parameters
\begin{align}
\alpha_1 = \alpha_3 = 1/4, \,\, \alpha_2=3, \,\, \alpha_4 = 1, \\
\beta_0 = \beta_1 = \beta_3 = \beta_4 = 0, \,\, \beta_2 = 1/4, \\
q=4,
\end{align}
and therefore Theorem \ref{thm:wellposedsmooth} implies that, given any initial data $u_0\in H^1(\Gamma_0)$, there exists a unique weak solution pair to the Cahn-Hilliard system with nonlinearity given by \eqref{quartic}. This recovers in particular the result in \cite{EllRan15} regarding well-posedness of the continuous problem. 
\end{example}

\begin{example}\label{eg:example4}
Let $T>0$ and suppose that the evolving surface $\{\Gamma(t)\}_{t\in [0,T]}$ is such that $\Gamma_0 = S^2(1)$ and, at some $t\in [0,T]$, $\Gamma(t) = S^2(1/2)$, where we denote by $S^2(r)$ the sphere of radius $r>0$ in $\R^3$. If we take $u_0\equiv 1/2\in H^1(\Gamma_0)$, then
\begin{align}
m_{u_0}(t) = \dfrac{1}{|\Gamma(t)|} \left|\int_{\Gamma_0} u_0\right| = \dfrac{1}{\pi} \dfrac{4\pi}{2} = 2 > 1,
\end{align}
and so Proposition \ref{prop:illposed1} (or Corollary \ref{cor:corillposed1}) implies that there is no (global in time) solution to the Cahn-Hilliard system with a logarithmic potential.
\end{example}

\begin{example}\label{eg:example5}
For a given $T>0$, suppose $\{\Gamma(t)\}_{t\in [0,T]}$ evolves in such a way that $|\Gamma(t)| \geq |\Gamma_0|$ for all $t\in [0, T]$. Then for any $u_0\in H^1(\Gamma_0)$ with finite initial energy and $|(u_0)_{\Gamma_0}|<1$, we have
\begin{align*}
m_{u_0}(t) = \dfrac{1}{|\Gamma(t)|} \left|\int_{\Gamma_0} u_0 \right| \leq \dfrac{1}{|\Gamma_0|} \left|\int_{\Gamma_0} u_0 \right| < 1,
\end{align*}
so that the Cahn-Hilliard systems with both the logarithmic and the double obstacle potentials with initial condition $u_0$ are well-posed. 
\end{example}

For some more concrete examples and illustrations we refer to \cite[Section 4]{YusQuaOls20}, where the authors present some numerical simulations and plots of the evolution of solutions to the Cahn-Hilliard system.

\section{Concluding remarks}

Our aim in this work has been to present a rigorous derivation of Cahn-Hilliard equations on an evolving surface, to establish well-posedness for the typical smooth, logarithmic and double obstacle potentials and to analyse the effect of the evolving nature of the domains in the solutions. This evolution for the surfaces is assumed to be given \textit{a priori}. 

We found that the system with regular potentials is well-posed for any choice of initial data $u_0\in H^1(\Gamma_0)$, and this has been proved by an evolving space Galerkin method. The class of nonlinearities considered includes the usual quartic potentials considered in the literature. As we have mentioned already, even though these do not produce physically meaningful solutions (as $u$ might leave the interval $[-1,1]$ even if $|u_0|\leq 1$), the quartic potentials are considered good approximations for the model describing a shallow quench of the system. 

The case of the singular potentials is more interesting, as it turns out that well-posedness of the systems relies on an interplay between the evolution of the surfaces, the initial data and the Cahn-Hilliard dynamics which force preservation of the integral of the solution. Here we identified a set of admissible initial conditions for which the system is well-posed, which essentially forces the mean value of the solution to remain in the interval $(-1,1)$. The proof for the logarithmic potential is achieved by regularisation of the nonlinearity close to the singularities, and for the double obstacle potential by use of a penalty method. We also showed that the double obstacle system can be obtained as the deep quench limit of the logarithmic problem. 

In Section \ref{sec:alternative}, we considered an alternative derivation for the model similar to what is proposed in recent work \cite{OlsXuYus21, YusQuaOls20, ZimTosLan19} which allows for general well-posedness statements in all cases without restricting the set of initial conditions. The reason for this is that while our first model preserves the integral of the order parameter $u$, which in the singular cases lies in $(-1,1)$, this alternative system conserves the integral of $\rho c$, where now $c\in (-1,1)$ and the weight function $\rho$ counterbalances the changes in the domains.

We finish with some comments on our results and some questions that remain to be addressed.

All of the work in this article was done for surfaces, i.e. in dimension $2$, as this is the most relevant dimension for applications. No changes are needed for curves in $\R^2$ and in higher dimensions only minor adaptations need to be made. For the regular potentials, some conditions are needed on the polynomial growth order; $q$ needs to be such that $|u|^q$ is integrable. For the singular potentials, since the regularisations are quadratic, no extra conditions are needed. Of course, the regularity results for all cases would also change, and this would also restrict the exponents $p\in [1,+\infty)$ for which $u\in L^\infty_{L^p}$.

Throughout Sections \ref{sec:log} and \ref{sec:obs} we have always assumed that the initial condition $u_0$ satisfies $(u_0)_{\Gamma_0}\in (-1,1)$, which excludes the cases $u_0\equiv \pm 1$. In a stationary domain, since preservation of integral implies $u\equiv 1$ for all times, there is no solution for the logarithmic problem, and the solution remains constant for the double obstacle problem. This is consistent with the physical interpretation of the model, and it would be interesting to establish the same result in our framework. This is another manifestation of the fact that the surfaces evolve with time -- preservation of the integral is not equivalent to preservation of the mean value.

Our result does not cover either the case in which $m_{u_0}$ never hits the value $1$ on a set of zero measure. As a concrete example, suppose there exists $t_*\in (0,T)$ such that $m_{u_0}(t_*)=1$ but $m_{u_0} < 1$ on $[0,t_*)\cup (t_*, T]$. We have well-posedness on $[0,t_*]$, and it is natural to ask whether the solution can be extended to $[0,T]$, effectively resolving a 'singularity' at $t_*$. This is an interesting question, which we aim to address in the future.

We conclude by mentioning some possible future research directions. 

It would be interesting to consider more general initial data, say $u_0\in L^2(\Gamma_0)$; for this case we expect well-posedness under the same type of assumptions and to observe instantaneous smoothing of the solutions. We have also to address higher regularity results for the logarithmic potential, which as we explained requires proving separation of the solutions from the pure phases, and for the double obstacle case which leads to the study of regularity for variational inequalities. Another natural question is that of long-term dynamics of the system. For instance, if $\Gamma(t)$ are defined for all $t\in [0,+\infty)$ and converge to some surface $\Gamma_\infty$ as $t\to\infty$, it is natural to try to identify the possible limits of the solution $u(t)$. In this article we have also only considered a constant mobility for the system, and it is an interesting and challenging problem to formulate and analyse the system with the phase-dependent mobility $M(u)=1-u^2$, which leads to the degenerate Cahn-Hilliard equation. Even in the classical setting, there are still many open questions on this problem, as well as for general degenerate fourth order parabolic PDEs. Another possible avenue is to drop the assumption that the evolution of the domains is given \textit{a priori}, and to couple the surface motion with the Cahn-Hilliard system on the surfaces, for instance by considering $(L^2, H^{-1})$-gradient flow for the Cahn-Hilliard energy $\mathrm E^{\mathrm{CH}} = \mathrm E^{\mathrm{CH}}(\Gamma, u)$. We leave all of these for future work.

\subsection*{Acknowledgements}

The work of CME was partially supported by the Royal Society via a Wolfson Research Merit Award. 

\appendix

\section{A generalised Gronwall Lemma}

In this appendix, we present a generalised Gronwall lemma.

\begin{lemma}[Generalised Gronwall Lemma]\label{lem:gengronwall}
Fix $T\in (0, \infty)$ and $t_\ast\in (0,T)$. For $M\in\N$, let $\alpha_M\colon [0,t_\ast] \to [0,+\infty)$ be a sequence of nonnegative differentiable functions. Suppose that there exists $\tilde{M}\in \N$ such that, for any $t\in [0,t_\ast]$ and $M\geq \tilde M$,
\begin{align}
\alpha_M'(t) &\leq C \left(C_0 + \alpha_M(t) + \eps \, \alpha_M^{q+1}(t)\right), \\
\alpha_M(0) &= \alpha_0 \geq 0,
\end{align}
where $C>0$, $C_0\geq 0$, $\eps>0$, and $q\in \N$ are independent of $M$ and $t_\ast$. If $\eps$ is small enough so that 
\begin{align}\label{eq:epsilon}
\eps \, e^{TCq}(\alpha_0+C_0)^q < 1,
\end{align}
it holds that, for $M\geq \tilde{M}$,
\begin{align}
\alpha_M(t) \leq \dfrac{(\alpha_0+C_0) \, e^{TC}}{\sqrt[q]{1 - \eps\, e^{TCq} (\alpha_0+C_0)^q}} - C_0.
\end{align}
%

\end{lemma}

\begin{proof}
Suppose first that $C_0=0$, then observe that 
\begin{align}
\dfrac{1}{q}\dfrac{d}{dt} \log \left(\dfrac{\alpha_M(t)^q}{1+\eps\,\alpha_M(t)^q}\right) = \dfrac{\alpha'_M(t)}{\alpha_M(t) + \eps\,\alpha_M(t)^{q+1}} \leq C,
\end{align}
and we can integrate the above over $[0,t]$, for any $t\in [0,t_\ast]$, to obtain
\begin{align}
\log \left(\dfrac{\alpha_M(t)^q}{1+\eps\,\alpha_M(t)^q}\right)  \leq \log \left(\dfrac{\alpha_0^q}{1+\eps\,\alpha_0^q}\right) + t_\ast Cq \leq \log \left(\dfrac{\alpha_0^q}{1+\eps\,\alpha_0^q}\right) + TCq
\end{align}
Now this implies that
\begin{align}
\dfrac{\alpha_M(t)^q}{1+\eps\,\alpha_M(t)^q} \leq e^{TCq} \, \dfrac{\alpha_0^q}{1+\eps\,\alpha_0^q} \leq e^{TCq} \, \alpha_0^q,
\end{align}
which we can solve for $\alpha_M(t)$ to obtain 
\begin{align}
\left( 1 - \eps\,e^{TCq} \alpha_0^q\right) \alpha_M(t)^q \leq \alpha_0^q \, e^{TCq}.
\end{align}
Using \eqref{eq:epsilon}, it follows that
\begin{align}
\alpha_M(t) \leq \dfrac{\alpha_0 \, e^{TC}}{\sqrt[q]{1 -  \, \eps\,e^{TCq}\alpha_0^q}} \leq \dfrac{\alpha_0 \, e^{TC}}{\sqrt[q]{1 - \, \eps\,e^{TCq} \alpha_0^q}}, \quad \text{ for } M\geq \widetilde{M},
\end{align}
as desired.

If $C_0\neq 0$, then define $\widetilde{\alpha}_M(t) := \alpha_M(t) + C_0$, so that
\begin{align}
\widetilde{\alpha}_M'(t) = \alpha_M'(t) \leq C \left(\widetilde{\alpha}_M(t) + \eps\,\alpha_M^{q+1}\right) \leq C \left(\widetilde{\alpha}_M(t) + \eps \, \widetilde{\alpha}_M^{q+1}\right),
\end{align}
and the previous case gives, for some $\widetilde{M}\in\N$,
\begin{align}
\widetilde{\alpha}_M(t) \leq \dfrac{\widetilde{\alpha}_0 \, e^{TC}}{\sqrt[q]{1 - \, \eps \, e^{TCq} \widetilde{\alpha}_0^q}} = \dfrac{(\alpha_0+C_0) \, e^{TC}}{\sqrt[q]{1 - \, \eps\, e^{TCq} (\alpha_0+C_0)^q}}, \quad \forall M\geq \tilde{M},
\end{align}
which translates to
\begin{align}
\alpha_M(t) \leq \dfrac{(\alpha_0+C_0) \, e^{TC}}{\sqrt[q]{1 - \, \eps \, e^{TCq} (\alpha_0+C_0)^q}} - C_0,
\end{align}
\end{proof}

\section{Dominated convergence theorems on evolving surfaces}

In this Appendix we return to the setting of an evolving surface $\{\Gamma(t)\}_{t\in [0,T]}$ under the same assumptions as in Section \ref{sec:prelim}. We start by establishing an evolving space version of Lebesgue's classical dominated convergence theorem.

\begin{theorem}[Dominated Convergence Theorem]\label{thm:evolv_dct}
Suppose $(f_M)_M\subset L^2_{L^2}$ be a sequence satisfying
\begin{itemize}
\item[(i)] for almost all $t\in [0,T]$, we have $f_M(t,x)\to f(t,x)$ as $M\to\infty$ for a.e. $x\in \Gamma(t)$;
\item[(ii)] there exists $g\in L^2_{L^2}$ such that, for a.a. $t\in [0,T]$ and all $M\in \N$, $|f_M(t,x)|\leq g(t,x)$ for a.e. $x\in \Gamma(t)$.
\end{itemize}
Then $f\in L^2_{L^2}$ and $\|f_M\|_{L^2_{L^2}} \to \|f\|_{L^2_{L^2}}$ as $M\to\infty$.
\end{theorem}

\begin{proof}
The proof follows from consecutive applications of the classical dominated convergence theorem (first in space, then in time). 

For the first step, by fixing $t$ in the set of positive measure for which $f_M(t)\to f(t)$ a.e. in $\Gamma(t)$, $|f_M(t)|\leq |g(t)|$ and $g(t)\in L^2(\Gamma(t))$ we can apply the DCT to obtain $f(t)\in L^2(\Gamma(t))$ and $\|f_M(t)\|_{L^2(\Gamma(t))} \to \|f(t)\|_{L^2(\Gamma(t))}$, for almost all $t\in [0,T]$. 

For the second step, define $\tilde f_M(t) = \|f_M(t)\|_{L^2(\Gamma(t))}$ and $\tilde f(t)=\|f(t)\|_{L^2(\Gamma(t))}$. Then the first step shows that $\tilde f_M(t)\to \tilde f(t)$ for a.a. $t\in [0,T]$, and by assumption $|\tilde f_M(t)| \leq \| g(t)\|_{L^2(\Gamma(t))}\in L^2(0,T)$, so that again an application of the DCT implies that $\tilde f\in L^2(0,T)$ and $\|\tilde f_M\|_{L^2(0,T)} \to \|\tilde f\|_{L^2(0,T)}$. But this is exactly the statement that $f\in L^2_{L^2}$ with $f_M\to f$ in $L^2_{L^2}$, concluding the proof.
\end{proof}

We now prove a version of Lebesgue's dominated convergence theorem for weak convergence. The statement and proof are a direct adaptation of \cite[Lemma 8.3]{Rob01}. 

\begin{theorem}[Generalised Dominated Convergence Theorem]\label{lem:weak_dct}
Suppose that $(g_M)_M\subset L^2_{L^2}$ is such that there exists $C>0$ such that, for all $M\in\N$, 
\begin{align}\label{eq:ass_app_1}
\|g_M\|_{L^2_{L^2}} \leq C.
\end{align} If $g\in L^2_{L^2}$ and $g_M\to g$ almost everywhere, then $g_M\rightharpoonup g$ in $L^2_{L^2}$. 
\end{theorem}

\begin{proof}
Define, for each $M\in M$ and $t\in [0,T]$, the family of sets
\begin{align*}
\Gamma_M(t) := \left\{x\in \Gamma(t) \colon |g_j(t, x)-g(t, x)| \leq 1, \, \text{for all } j\geq M\right\}.
\end{align*}
It is clear that $\Gamma_M(t)\subset \Gamma_{M+1}(t)$ and since $g_M(t)\to g(t)$ almost everywhere we have $|\Gamma_M(t)| \nearrow |\Gamma(t)|$ as $M\to\infty$. Now consider the sets
\begin{align*}
U_M(t) = \left\{ \varphi\in L^2_{L^2} \colon \text{supp } \varphi(t) \subset \Gamma_M(t) \right\} \quad \text{ and } \quad U(t) = \bigcup_{M=1}^\infty U_M(t),
\end{align*}
and define also 
\begin{align*}
L^2_{U_M} := \{\varphi\in L^2_{L^2} \colon \varphi(t)\in U_M(t)\} \quad \text{ and } \quad L^2_U = \bigcup_{M\in \N} L^2_{U_M}.
\end{align*}

We start by proving that $L^2_U$ is dense in $L^2_{L^2}$. Let $\varphi\in L^2_{L^2}$ and define a  sequence $\varphi_M(t) = \chi_{\Gamma_M(t)} \varphi(t)$, where $\chi_{\Gamma_M(t)}$ denotes the characteristic function of $V_{\Gamma_M(t)}$, so that $\varphi_M\in L^2_{U_M}$. For a.a. $t\in [0,T]$, the definition of $\Gamma_M(t)$ implies that $\varphi_M(t)\to \varphi(t)$ a.e. in $\Gamma(t)$, and we also have
\begin{align*}
|\varphi_M(t)|\leq |\varphi(t)| \quad \text{ and } \quad \varphi \in L^2_{L^2},
\end{align*}
and therefore by the Dominated Convergence Theorem \ref{thm:evolv_dct} it follows that $\varphi_M\to \varphi$ in $L^2(\Gamma)$. 

Now let $\varphi\in L^2_U$. We claim that as $M\to\infty$ we have
\begin{align}\label{eq:app_1}
\int_0^T\int_{\Gamma(t)} \varphi (g_M- g)  \to 0.
\end{align}
This is again an application of the DCT; there exists $M_0\in \N$ such that $\varphi\in L^2_{U_{M_0}}$, and thus 
\begin{align*}
\left|\varphi (g_M-g)\right| \leq |\varphi| \in L^2_{L^2}, \quad \text{for } M\geq M_0,
\end{align*}
and of course $\varphi(g_M-g)\to 0$ a.e., so \eqref{eq:app_1} follows again from Theorem \ref{thm:evolv_dct}.

To establish the weak convergence $g_M\rightharpoonup g$ we combine \eqref{eq:app_1} with the density of $L^2_U$ in $L^2_{L^2}$. Let $\eps>0$ and $\eta\in L^2_{L^2}$ be arbitrary, and pick $\varphi\in L^2_U$ such that 
\begin{align*}
\|\eta-\varphi\|_{L^2_{L^2}} \leq \dfrac{\eps}{4C},
\end{align*}
where $C$ is the constant in \eqref{eq:ass_app_1}. By \eqref{eq:app_1} we can also take $M_0\in \N$ large enough so that
\begin{align*}
\int_0^T \int_{\Gamma(t)} \varphi (g_M -g)  \leq \dfrac{\eps}{2}, \quad \text{ for } M\geq M_0.
\end{align*}
Finally, we also apply Fatou's lemma combined with \eqref{eq:ass_app_1} to obtain
\begin{align*}
\|g\|_{L^2_{L^2}} \leq \liminf_{M\to\infty} \|g_M\|_{L^2_{L^2}} \leq C,
\end{align*}
and therefore
\begin{align*}
\left| \int_0^T\int_{\Gamma(t)} \eta (g_M-g)  \right| &\leq \left|\int_0^T \int_{\Gamma(t)} (\eta-\varphi) (g_M - g)  \right| +\left|\int_0^T \int_{\Gamma(t)} \varphi (g_M - g)  \right| \\
&\leq \|\eta-\varphi\|_{L^2_{L^2}}  \left( \|g_M\|_{L^2_{L^2}} + \|g\|_{L^2_{L^2}} \right) + \dfrac{\eps}{2} \\
&\leq \dfrac{\eps}{4C} 2C + \dfrac{\eps}{2} \\
&= \eps,
\end{align*}
as desired.
\end{proof}

\begin{remark}
The statements above can be extended to the more general Banach spaces $L^p_{L^q}$, with $p,q\in [1,+\infty]$, with similar proofs, but we only need the case $p=q=2$. Although we have not introduced these spaces in Section \ref{sec:prelim}, they are defined in exactly the same way with the natural changes (see \cite{AlpEllSti15a}). The assumptions that the bounds in both results hold for all $M\in\N$ can of course be relaxed; it suffices that they hold only for large enough $M$.
\end{remark}

\section{The inverse Laplacian}\label{app:invlaplacian}

Fix $t\in [0,T]$. For each $z\in H^{-1}(\Gamma(t))$ with $$\left\langle z, 1 \right\rangle_{H^{-1}, \, H^1} = 0$$
we define the \textit{inverse Laplacian} of $z$, which we denote by $\mathcal{G}(t)z$, the unique solution of the problem
\begin{align}
\int_{\Gamma(t)} \tgrad \mathcal{G}(t)z \cdot \tgrad \eta &= \left\langle z, \eta\right\rangle_{H^{-1}, \, H^1} \quad \forall \eta\in H^1(\Gamma(t)), \\
\int_{\Gamma(t)} \mathcal G(t) z  &= 0.
\end{align} 
This defines a map $\mathcal{G}(t)\colon H^{-1}(\Gamma(t)) \to H^1(\Gamma(t))$, which we use to define a norm in $H^{-1}(\Gamma(t))$ by
\begin{align}
\|z\|_{-1,t}^2 := \|\nabla_{\Gamma(t)} \mathcal{G}(t)z \|_{L^2(\Gamma(t))}^2 = m_\ast (t; z, \mathcal{G}(t)z).
\end{align} 

More generally, setting 
\begin{align}
L^2_{H^{1}, \, 0} := \left\{ u \in L^2_{H^{1}} \colon \int_0^T \int_{\Gamma(t)} u(t) = 0 \right\}, \quad L^2_{H^{-1}, \, 0} := \left\{ z \in L^2_{H^{-1}} \colon \left\langle z, 1\right\rangle_{L^2_{H^{-1}}, \, L^2_{H^1}} = 0 \right\}
\end{align}
the above allows us to define a map
\begin{align}
\mathcal{G}\colon L^2_{H^{-1}, \, 0} \to L^2_{H^1, \, 0}, \quad (\mathcal{G}z)(t) := \mathcal{G}(t)z(t).
\end{align}
The following result will be useful to prove uniqueness of solutions.

\begin{lemma}\label{lem:invlaplace}
If $z\in H^1_{H^{-1}}$, then $\mathcal{G}z\in H^1_{H^1}$, i.e. $\mathcal G z\in H^1_{H^{-1}}$ with $\md \mathcal G z\in L^2_{H^1}$.
\end{lemma}

For a proof, see \cite[Lemma 4.3]{EllRan15}.

\bibliographystyle{alpha}
\bibliography{biblio}

\newcommand{\etalchar}[1]{$^{#1}$}
\begin{thebibliography}{EAK{\etalchar{+}}01}

\bibitem[AES15a]{AlpEllSti15a}
A.~Alphonse, C.~M. Elliott, and B.~Stinner.
\newblock An abstract framework for parabolic {PDEs} on evolving spaces.
\newblock {\em Portugaliae Mathematica}, 71(1):1--46, 2015.

\bibitem[AES15b]{AlpEllSti15b}
A.~Alphonse, C.~M. Elliott, and B.~Stinner.
\newblock On some linear parabolic {PDEs} on moving hypersurfaces.
\newblock {\em Interfaces and Free Boundaries}, 17:157--187, 2015.

\bibitem[AET17]{AlpEllTer17}
A.~Alphonse, C.~M. Elliott, and J.~Terra.
\newblock A coupled ligand-receptor bulk-surface system on a moving domain:
  Well posedness, regularity, and convergence to equilibrium.
\newblock {\em SIAM Journal on Mathematical Analysis}, 50:1544--1592, 2017.

\bibitem[Bai65]{Bai65}
C.~Baiocchi.
\newblock {Regolarit{\`a} e unicit{\`a} della soluzione di una equazione
  differenziale astratta}.
\newblock {\em {Rendiconti del Seminario Matematico delle Universit{\`a} di
  Padova}}, 35(2):380--417, 1965.

\bibitem[BE91]{BloEll91}
J.~F. Blowey and C.~M. Elliott.
\newblock The {C}ahn--{H}illiard gradient theory for phase separation with
  non--smooth free energy {P}art {I}: {M}athematical {A}nalysis.
\newblock {\em European J. Applied Mathematics}, 2:233--280, 1991.

\bibitem[BE92]{BloEll92}
J.~F. Blowey and C.~M. Elliott.
\newblock The {C}ahn--{H}illiard gradient theory for phase separation with
  non--smooth free energy {P}art {II}: {N}umerical {A}nalysis.
\newblock {\em European J. Applied Mathematics}, 3:147--179, 1992.

\bibitem[BEM11]{BarEllMad11}
R.~Barreira, C.~M. Elliott, and A.~Madzvamuse.
\newblock The surface finite element method for pattern formation on evolving
  biological surfaces.
\newblock {\em Journal of Mathematical Biology}, 63:1095--1119, 2011.

\bibitem[Cah61]{Cah61}
J.~W. Cahn.
\newblock On spinodal decomposition.
\newblock {\em Acta Metall. Mater.}, 9:795--801, 1961.

\bibitem[CENC96]{CahEllNov96}
J.~W. Cahn, C.~M. Elliott, and A.~Novick-Cohen.
\newblock The {C}ahn--{H}illiard equation with a concentration dependent
  mobility: motion by minus the laplacian of the mean curvature.
\newblock {\em European J. Applied Mathematics}, 7:287--301, 1996.

\bibitem[CH58]{CahHil58}
J.~W. Cahn and J.~E. Hilliard.
\newblock {Free energy of a nonuniform system I. Interfacial free energy}.
\newblock {\em J. Chem. Phys.}, 28:258--267, 1958.

\bibitem[CMZ11]{CheMirZel11}
L.~Cherfile, A.~Miranville, and S.~Zelik.
\newblock {The Cahn-Hilliard Equation with Logarithmic Potentials}.
\newblock {\em Milan Journal of Mathematics}, 79(2):561--596, Dec 2011.

\bibitem[DD95]{DebDet95}
A.~Debussche and L.~Dettori.
\newblock On the {C}ahn-{H}illiard equation with a logarithmic free energy.
\newblock {\em Nonlinear Anal. - Theor.}, 24(10):1491--1514, 1995.

\bibitem[DD14]{DaiDu14}
S.~Dai and Q.~Du.
\newblock Coarsening mechanism for systems governed by the {Cahn-Hilliard}
  equation with degenerate diffusion mobility.
\newblock {\em Multiscal Model. Simul.}, 12(4):1870--1889, 2014.

\bibitem[DD16]{DaiDu16}
S.~Dai and Q.~Du.
\newblock Weak solutions for the {Cahn-Hilliard} equation with degenerate
  mobility.
\newblock {\em Arch. Rational Mech. Anal}, 219:1161--1184, 2016.

\bibitem[DDE05]{DecDziEll05-a}
K.~Deckelnick, G.~Dziuk, and C.~M. Elliott.
\newblock Computation of geometric partial differential equations and mean
  curvature flow.
\newblock {\em Acta Numerica}, 14:139--232, 2005.

\bibitem[DE07]{DziEll07-a}
G.~Dziuk and C.~M. Elliott.
\newblock Finite elements on evolving surfaces.
\newblock {\em IMA Journal Numerical Analysis}, 25:385--407, 2007.

\bibitem[DE13]{DziEll13-a}
G.~Dziuk and C.~M. Elliott.
\newblock Finite element methods for surface partial differential equations.
\newblock {\em Acta Numerica}, 22:289--396, 2013.

\bibitem[EAK{\etalchar{+}}01]{ErlAziKar01}
J.~Erlebacher, M.J. Aziz, A.~Karma, N.~Dimitrov, and K.~Sieradzki.
\newblock Evolution of nanoporosity in dealloying.
\newblock {\em Nature}, 410:450--453, 2001.

\bibitem[EE08]{EilEll08}
C.~Eilks and C.~M. Elliott.
\newblock Numerical simulation of dealloying by surface dissolution via the
  evolving surface finite element method.
\newblock {\em Journal of Computational Physics}, 227(23):9727--9741, 12 2008.

\bibitem[EF89]{EllFre89}
C.~M. Elliott and D.~A. French.
\newblock A nonconforming finite element method for the two-dimensional
  {C}ahn-{H}illiard equation.
\newblock {\em SIAM J. Numer. Anal.}, 24(4):884--903, 1989.

\bibitem[EFM89]{EllFreMil89}
C.~M. Elliott, D.~A. French, and F.~A. Milner.
\newblock A second order splitting method for the {C}ahn-{H}illiard equation.
\newblock {\em Numerische Mathematik}, 54(5):575--590, 1989.

\bibitem[EG96]{EllGar96a}
C.~M. Elliott and H.~Garcke.
\newblock On the {C}ahn-{H}illiard equation with degenerate mobility.
\newblock {\em SIAM Journal on Mathematical Analysis}, 27(2):404--423, 1996.

\bibitem[EL91]{EllLuc91}
C.~M. Elliott and S.~Luckhaus.
\newblock A generalised diffusion equation for phase separation of a
  multi-component mixture with interfacial free energy.
\newblock {\em IMA Preprint Series}, (887), 1991.

\bibitem[EL92]{EllLar92}
C.~M. Elliott and S.~Larsson.
\newblock Error estimates with smooth and nonsmooth data for a finite element
  method for the {C}ahn-{H}illiard equation.
\newblock {\em Math. Comput.}, 58(198):603--630, 1992.

\bibitem[Ell89]{Ell89}
C.~M. Elliott.
\newblock {\em The {C}ahn-{H}illiard Model for the Kinetics of Phase
  Separation}, pages 35--73.
\newblock Number~88 in International Series of Numerical Mathematics.
  Birkh{\"a}user Verlag, Basel, Germany, 1989.

\bibitem[ER15]{EllRan15}
C.~M. Elliott and T.~Ranner.
\newblock Evolving surface finite element method for the {Cahn--Hilliard}
  equation.
\newblock {\em Numerische Mathematik}, 129(3):483--534, Mar 2015.

\bibitem[ER20]{EllRan20}
C.~M. Elliott and T.~Ranner.
\newblock A unified theory for continuous-in-time evolving finite element space
  approximations to partial differential equations in evolving domains.
\newblock {\em IMA Journal of Numerical Analysis}, November 2020.
\newblock {\copyright} The Author(s) 2020. Published by Oxford University Press
  on behalf of the Institute of Mathematics and its Applications. All rights
  reserved. This is an author produced version of an article published in .
  Uploaded in accordance with the publisher's self-archiving policy.

\bibitem[ESV12]{EllStiVen12}
C.~M. Elliott, B.~Stinner, and C.~Venkataraman.
\newblock Modelling cell motility and chemotaxis with evolving surface finite
  elements.
\newblock {\em Journal of the Royal Society Interface}, 9(76):3027--3044, 2012.

\bibitem[EZ86]{EllSon86}
C.~M. Elliott and S.~Zheng.
\newblock On the {C}ahn-{H}illiard equation.
\newblock {\em Archive for Rational Mechanics and Analysis}, 96(4):339--357,
  1986.

\bibitem[Gar13]{Gar13}
H.~Garcke.
\newblock Curvature driven interface evolution.
\newblock {\em Jahresbericht der Deutschen Mathematiker-Vereinigung},
  115(2):63--100, Sep 2013.

\bibitem[GLS14]{GarLamSti14}
H.~Garcke, K.~F. Lam, and B.~Stinner.
\newblock Diffuse interface modelling of soluble surfactants in two-phase flow.
\newblock {\em Commun. Math. Sci.}, 12(8):1475--1522, 2014.

\bibitem[GT98]{GilTru98}
D.~Gilbarg and N.~S. Trudinger.
\newblock {\em Elliptic partial differential equations of second order}.
\newblock Grundlehren der mathematischen Wissenschaften. Springer Verlag, 1998.

\bibitem[Hei15]{Hei15}
M.~Heida.
\newblock {Existence of solutions for two types of generalized versions of the
  Cahn-Hilliard equation}.
\newblock {\em Applications of Mathematics}, 60(1):51--90, 2015.

\bibitem[HGK20]{GarKno20}
Harald H.~Garcke and P.~Knopf.
\newblock Weak solutions of the cahn-hilliard system with dynamic boundary
  conditions: A gradient flow approach.
\newblock {\em SIAM Journal on Mathematical Analysis}, 52(1):340--369, 2020.

\bibitem[KS80]{KinSta80}
D.~Kinderlehrer and G.~Stampacchia.
\newblock {\em An introduction to variational inequalities and their
  applications}.
\newblock Academic Press, New York - London, 1980.

\bibitem[Leo09]{Leo09}
G.~Leoni.
\newblock {\em A First Course in Sobolev Spaces}.
\newblock American Mathematical Society, Providence, RI, 2009.

\bibitem[Lio57]{Lio57}
Jacques-Louis Lions.
\newblock {Sur les probl{\`e}mes mixtes pour certains syst{\`e}mes paraboliques
  dans les ouverts non cylindriques}.
\newblock {\em {Annales de l'institut Fourier}}, 7:143--182, 1957.

\bibitem[LSTT21]{LanSonTanThu21}
D.~Lan, D.~T. Son, B.~Q. Tang, and L.~T. Thuy.
\newblock Quasilinear parabolic equations with first order terms and l1-data in
  moving domains.
\newblock {\em Nonlinear Analysis}, 206:112233, 2021.

\bibitem[Mir19]{Mir19}
A.~Miranville.
\newblock {\em The Cahn-Hilliard equation: Recent Advances and Applications}.
\newblock CBMS-NSF Regional Conference Series in Applied Mathematics, Society
  for Industrial and Applied Mathematics (SIAM), Philadelphia, PA, 2019.

\bibitem[Nae15]{Nae15}
P.~Naegele.
\newblock {\em {Monotone operator theory for unsteady problems on
  non-cylindrical domains}}.
\newblock PhD thesis, Uni. Freiburg, 2015.

\bibitem[NCS84]{NovSeg84}
A.~Novick-Cohen and L.~A. Segel.
\newblock Nonlinear aspects of the {Cahn-Hilliard} equation.
\newblock {\em Physica D: Nonlinear Phenomena}, 10(3):277 -- 298, 1984.

\bibitem[OS16]{OcoSti16}
D.~O'Connor and B.~Stinner.
\newblock {The Cahn-Hilliard equation on an evolving surface}.
\newblock {\em arXiv e-prints}, page arXiv:1607.05627, Jul 2016.

\bibitem[OXY21]{OlsXuYus21}
M.~Olshanskii, X.~Xu, and V.~Yushutin.
\newblock {A finite element method for Allen-Cahn equation on deforming
  surface}.
\newblock {\em {Computers and Mathematics with Applications}}, 90:148--158,
  2021.

\bibitem[PD96]{PraDeb96}
G.~Da Prato and A.~Debussche.
\newblock Stochastic {Cahn-Hilliard} equation.
\newblock {\em Nonlinear Analysis: Theory, Methods \& Applications}, 26(2):241
  -- 263, 1996.

\bibitem[Peg]{Peg86}
R.~L. Pego.
\newblock Front migration in the nonlinear {Cahn-Hilliard} equation.
\newblock {\em Proceedings of the Royal Society of London. Series A,
  Mathematical and Physical Sciences}, 422:261--278.

\bibitem[Rob01]{Rob01}
J.~C. Robinson.
\newblock {\em Infinite dimensional dynamical systems}.
\newblock Cambridge Texts in Apllied Mathematics. Cambridge, 2001.

\bibitem[RT01]{RakTem01}
J.~M. Rakotoson and R.~Temam.
\newblock An optimal compactness theorem and application to elliptic-parabolic
  systems.
\newblock {\em Applied Mathematics Letters}, 14:303--306, April 2001.

\bibitem[Vie14]{Vie14}
M.~Vierling.
\newblock Parabolic optimal control problems on evolving surfaces subject to
  point-wise box constraints on the control - theory and numerical realization.
\newblock {\em Interfaces and Free Boundaries}, 16:137--173, 2014.

\bibitem[VSG{\etalchar{+}}11]{VenSekGaf11}
C.~Venkataraman, T.~Sekimura, E.~A. Gaffney, P.~K. Maini, and A.~Madzvamuse.
\newblock Modeling parr-mark pattern formation during the early development of
  amago trout.
\newblock {\em Phys. Rev. E}, 84:041923, Oct 2011.

\bibitem[YQO20]{YusQuaOls20}
V.~Yushutin, A.~Quaini, and M.~Olshanskii.
\newblock Numerical modeling of phase separation on dynamic surfaces.
\newblock {\em Journal of Computational Physics}, 407:109126, 2020.

\bibitem[Zei90]{Zei90}
E.~Zeidler.
\newblock {\em Nonlinear functional analysis and its applications. II/B:
  Nonlinear monotone operators}.
\newblock Springer-Verlag, New York, 1990.

\bibitem[ZTL{\etalchar{+}}19]{ZimTosLan19}
C.~Zimmermann, D.~Toshniwal, C.~M. Landis, T.~J.~R. Hughes, K.~K. Mandadapu,
  and R.~A. Sauer.
\newblock An isogeometric finite element formulation for phase transitions on
  deforming surfaces.
\newblock {\em Comput. Methods Appl. Mech. Engrg.}, 351:441--477, 2019.

\end{thebibliography}

\end{document}